\documentclass[a4paper,10pt]{amsart}

\usepackage{amsmath}
\usepackage{amsthm}
\usepackage{amssymb}
\usepackage{amsfonts}
\usepackage{enumitem} 
\usepackage[bookmarks=true,hyperindex,pdftex,colorlinks,citecolor=blue]{hyperref}

\makeatletter
\renewcommand{\tocsection}[3]{%
  \indentlabel{\@ifnotempty{#2}{\bfseries\ignorespaces#1 #2\quad}}\bfseries#3}
\renewcommand{\tocsubsection}[3]{%
  \indentlabel{\@ifnotempty{#2}{\ignorespaces#1 #2\quad}}#3}

\newcommand\@dotsep{4.5}
\def\@tocline#1#2#3#4#5#6#7{\relax
  \ifnum #1>\c@tocdepth 
  \else
    \par \addpenalty\@secpenalty\addvspace{#2}%
    \begingroup \hyphenpenalty\@M
    \@ifempty{#4}{%
      \@tempdima\csname r@tocindent\number#1\endcsname\relax
    }{%
      \@tempdima#4\relax
    }%
    \parindent\z@ \leftskip#3\relax \advance\leftskip\@tempdima\relax
    \rightskip\@pnumwidth plus1em \parfillskip-\@pnumwidth
    #5\leavevmode\hskip-\@tempdima{#6}\nobreak
    \leaders\hbox{$\m@th\mkern \@dotsep mu\hbox{.}\mkern \@dotsep mu$}\hfill
    \nobreak
    \hbox to\@pnumwidth{\@tocpagenum{\ifnum#1=1\bfseries\fi#7}}\par
    \nobreak
    \endgroup
  \fi}
\AtBeginDocument{%
\expandafter\renewcommand\csname r@tocindent0\endcsname{0pt}
}
\def\l@subsection{\@tocline{2}{0pt}{2.5pc}{5pc}{}}
\makeatother

\DeclareMathOperator{\lspan}{span}                          
\DeclareMathOperator{\conv}{conv}                           
\DeclareMathOperator{\supp}{supp}                           
\DeclareMathOperator{\diam}{diam}                           
\DeclareMathOperator{\rad}{rad}                             
\DeclareMathOperator{\ext}{ext}                             
\DeclareMathOperator{\Lip}{Lip}                             
\DeclareMathOperator{\sgn}{sgn}                     

\newcommand{\NN}{\mathbb{N}}             
\newcommand{\ZZ}{\mathbb{Z}}             
\newcommand{\RR}{\mathbb{R}}             

\newcommand{\abs}[1]{\left|{#1}\right|}                     
\newcommand{\pare}[1]{\left({#1}\right)}                    
\newcommand{\set}[1]{\left\{{#1}\right\}}                   
\newcommand{\norm}[1]{\left\|{#1}\right\|}                  
\newcommand{\dual}[1]{{#1}^\ast}                            
\newcommand{\ddual}[1]{{#1}^{\ast\ast}}                     
\newcommand{\ball}[1]{B_{{#1}}}                             
\newcommand{\sphere}[1]{S_{{#1}}}                           
\newcommand{\duality}[1]{\left<{#1}\right>}                 
\newcommand{\cl}[1]{\overline{#1}}                          
\newcommand{\wscl}[1]{\overline{#1}^{\dual{w}}}             
\newcommand{\weaks}{\textit{w}$^\ast$}                      
\newcommand{\wsconv}{\stackrel{\dual{w}}{\longrightarrow}}  
\newcommand{\ucl}[1]{\overline{#1}^{\mathcal{U}}}           
\newcommand{\restrict}{\mathord{\upharpoonright}}           

\newcommand{\lipfree}[1]{\mathcal{F}({#1})}                 
\newcommand{\lipnorm}[1]{\norm{#1}_L}                       
\newcommand{\ideal}[1]{\mathcal{I}({#1})}                   
\newcommand{\pos}[1]{#1^+}                                  
\newcommand{\opint}{\mathcal{L}}                            
\newcommand{\meas}[1]{\mathcal{M}({#1})}                    
\newcommand{\measz}[1]{\mathcal{M}_0({#1})}                 
\newcommand{\ucomp}[1]{#1^\mathcal{U}}                      
\newcommand{\rcomp}[1]{{#1}^\mathcal{R}}                    
\newcommand{\esupp}[1]{\mathcal{S}(#1)}                     
\newcommand{\wop}[1]{\mathcal{T}_{#1}}                      
\newcommand{\kaltonsum}[1]{{#1}_s}                          
\newcommand{\bidualfree}[1]{\dual{\Lip_0({#1})}}            

\theoremstyle{plain}
\newtheorem{theorem}{Theorem}[section]
\newtheorem{lemma}[theorem]{Lemma}
\newtheorem{corollary}[theorem]{Corollary}
\newtheorem{proposition}[theorem]{Proposition}

\newtheorem*{claim*}{Claim}

\theoremstyle{definition}
\newtheorem*{definition*}{Definition}
\newtheorem{definition}[theorem]{Definition}
\newtheorem{example}[theorem]{Example}
\newtheorem{question}{Question}

\theoremstyle{remark}
\newtheorem{remark}[theorem]{Remark}

\numberwithin{equation}{section}

\begin{document}

\title[Integral representation of functionals on Lipschitz spaces]{Integral representation and supports of functionals on Lipschitz spaces}

\author[R. J. Aliaga]{Ram\'on J. Aliaga}
\address[R. J. Aliaga]{Instituto Universitario de Matem\'atica Pura y Aplicada, Universitat Polit\`ecnica de Val\`encia, Camino de Vera S/N, 46022 Valencia, Spain}
\email{raalva@upvnet.upv.es}

\author[E. Perneck\'a]{Eva Perneck\'a}
\address[E. Perneck\'a]{Faculty of Information Technology, Czech Technical University in Prague, Th\'akurova 9, 160 00, Prague 6, Czech Republic}
\email{perneeva@fit.cvut.cz}

\date{} 


\begin{abstract}
We analyze the relationship between Borel measures and continuous linear functionals on the space $\mathrm{Lip}_0(M)$ of Lipschitz functions on a complete metric space $M$. In particular, we describe continuous functionals arising from measures and vice versa. In the case of weak$^\ast$ continuous functionals, i.e. members of the Lipschitz-free space $\mathcal{F}(M)$, measures on $M$ are considered. For the general case, we show that the appropriate setting is rather the uniform (or Samuel) compactification of $M$ and that it is consistent with the treatment of $\mathcal{F}(M)$. This setting also allows us to give a definition of support for all elements of $\mathrm{Lip}_0(M)^\ast$ with similar properties to those in $\mathcal{F}(M)$, and we show that it coincides with the support of the representing measure when such a measure exists. We deduce that the members of $\mathrm{Lip}_0(M)^\ast$ that can be expressed as the difference of two positive functionals admit a Jordan-like decomposition into a positive and a negative part.
\end{abstract}

\subjclass[2010]{Primary 46B20, 46E27; Secondary 46B40, 46B10}

\keywords{Integral representation, Jordan decomposition, Lipschitz-free space, Lipschitz realcompactification, majorizable functional, Radon measure, real-valued Lipschitz function, support, uniform compactification}

\maketitle


\tableofcontents

\section{Introduction}

Let $(M,d)$ be a complete metric space and denote by $\Lip(M)$ the space of all real-valued Lipschitz functions on $M$, i.e. those $f:M\rightarrow\RR$ whose (optimal) Lipschitz constant
$$
\lipnorm{f}=\sup\set{\frac{\abs{f(x)-f(y)}}{d(x,y)}:x\neq y\in M}
$$
is finite. Further, select a base point $0\in M$ (such $M$ is called \textit{pointed}) and let $\Lip_0(M)$ be the space of those $f\in\Lip(M)$ such that $f(0)=0$. Then $(\Lip_0(M),\lipnorm{\,\cdot\,})$ is a dual Banach space referred to as \textit{Lipschitz space}. The space
$$
\lipfree{M}=\cl{\lspan}\set{\delta(x):x\in M}\subset\dual{\Lip_0(M)} ,
$$
where $\delta(x)\in\dual{\Lip_0(M)}$ is the evaluation functional on $x\in M$, usually receives the name \textit{Lipschitz-free space over $M$} (or, more rarely, \textit{Arens-Eells space}). It has the following fundamental properties:
\begin{itemize}
\item It contains an isometric copy $\delta(M)$ of $M$ that is linearly dense.
\item It is the canonical predual of $\Lip_0(M)$ (it is in fact conjectured to be its only predual, although this has only been proved under additional conditions, e.g. when $M$ is bounded or a Banach space \cite{Weaver_2018}).
\item It satisfies the following extension property: any Lipschitz mapping from $M$ into a Banach space $X$ can be extended to a bounded linear operator from $\lipfree{M}$ into $X$ (insofar as $M$ is identified with its copy $\delta(M)$ in $\lipfree{M}$). In particular, any Lipschitz mapping between two metric spaces $M$ and $N$ can be extended to a bounded linear operator between the Lipschitz-free spaces $\lipfree{M}$ and $\lipfree{N}$. Moreover, the norm of the operator is equal to the Lipschitz constant of the original mapping.
\end{itemize}
As an immediate application of the extension property, one can dismiss the existence of bi-Lipschitz homeomorphisms between two metric spaces by proving that their Lipschitz-free spaces are not isomorphic to each other. If we take $M$ to be a Banach space with the metric induced by its norm, the extension property leads to plenty of deeper applications to the study of the nonlinear geometry of Banach spaces. Some of the most celebrated ones were proved in the paper \cite{GoKa_2003} by Godefroy and Kalton, such as the following:
\begin{itemize}
\item The bounded approximation property of Banach spaces is stable under bi-Lipschitz homeomorphisms.
\item If a Banach space contains an isometric copy of another separable Banach space, then it actually contains a \textit{linearly} isometric copy.
\end{itemize}
In the wake of these results, Lipschitz-free spaces have been the subject of very active research in the last two decades. See for instance Weaver's monograph \cite{Weaver2} for a detailed analysis of their linear, algebraic and order properties, and Godefroy's paper \cite{Godefroy_2015} for a survey of their applications in nonlinear functional analysis.

If $K$ is a closed subset of $M$, it follows from the extension property that $\lipfree{K}$ can be identified with the subspace of $\lipfree{M}$ generated by the evaluations on points of $K$. In our previous work \cite{AlPe_2020}, we established an intuitive but hitherto unnoticed fact: the intersection of Lipschitz-free spaces $\lipfree{K_i}$ over subsets $K_i\subset M$ is exactly the space $\lipfree{\bigcap K_i}$ (this was first proved for bounded $M$, then extended to the general case in \cite{APPP_2020}). This intersection property allows us to give a natural definition of the \textit{support} $\supp(m)$ of an arbitrary element $m$ of $\lipfree{M}$. Several equivalent definitions of this object are described in \cite{APPP_2020}. Let us state a very illustrative one: \textit{$\supp(m)$ is the smallest closed set with the property that $\duality{m,f}=\duality{m,g}$ whenever $f,g\in\Lip_0(M)$ coincide on $\supp(m)$}.

Notice the strong similarity between the supports defined thusly and the usual notion of support of a measure on $M$. This, together with the fact that measures also act (by integration) as functionals on the space $\Lip_0(M)$, makes it natural to analyze the relationship between both concepts of support, framed within a more general analysis of the analogy between $\lipfree{M}$ and spaces of measures on $M$.

Let us briefly report on some previous related results to be found in the existing literature. Godefroy and Kalton state in \cite[p. 123]{GoKa_2003} that finite Borel measures on $M$ with compact supports can be identified with elements of $\lipfree{M}$, even as Bochner integrals. On the other hand, Weaver shows in \cite[Theorem 3.19]{Weaver2} that not all elements of $\lipfree{M}$ correspond to such measures when $M$ is compact and infinite. In \cite[Proposition 2.7]{AmPu_2016}, Ambrosio and Puglisi show that finite Borel measures on $M$ with finite first moment induce elements of $\lipfree{M}$ and conversely, any positive element $m\in\lipfree{M}$, i.e. such that $\duality{m,f}\geq 0$ for any $f\geq 0$, can be represented by a Borel measure $\mu$ on $M$ (we warn the reader that this statement contains a typo according to which $\mu$ is always a probability measure, which is not true as $\mu$ does not even have to be finite -- see Remark \ref{rm:ambrosio_puglisi} below). The paper \cite{HiWo_2009} by Hille and Worm is also worth mentioning, wherein the authors do not analyze $\Lip_0(M)$ but rather the related space $\Lip_e(M)=\Lip_0(M)\oplus_1\RR$ and the space $\mathrm{BL}(M)$ of bounded Lipschitz functions with norm $\norm{f}=\lipnorm{f}+\norm{f}_\infty$. Both spaces are shown to admit preduals whose positive elements can be represented by finite Borel measures on $M$.

In the present paper, we carry out a comprehensive analysis of the relationship between measures and elements of $\lipfree{M}$ for a general complete metric space $M$ extending the studies mentioned in the previous paragraph. In particular, we give characterizations for those elements of $\lipfree{M}$ that can be represented as a Radon or (not necessarily finite) Borel measure on $M$ and vice versa, and show that the supports agree up to the base point. Moreover, we identify those metric spaces $M$ for which every element of $\lipfree{M}$ can be represented by a Borel measure.

We also consider the \textit{majorizable} elements of $\lipfree{M}$, i.e. those that can be expressed as a difference between two positive elements. Up to some technicalities, they can be represented by measures. This allows us to deduce that any majorizable element of $\lipfree{M}$ admits a canonical optimal decomposition as a difference of two positive elements, similar to the Jordan decomposition for measures. That is, it is possible to identify its ``positive part'' and its ``negative part''.

More importantly, we are able to generalize this whole analysis to the bidual $\ddual{\lipfree{M}}=\dual{\Lip_0(M)}$. One key issue to be solved here is determining an appropriate base space for the measures. This space cannot be $M$ as it is well known that the evaluations on certain elements of the Stone-\v{C}ech compactification $\beta M$ are also elements of $\dual{\Lip_0(M)}$ (see e.g. \cite[p. 36]{Weaver1}). It is tempting to consider measures on $\beta M$, but this leads to an unsatisfying theory because $\beta M$ is ``too big'' and Lipschitz functions on $M$ do not separate points of $\beta M$ in general. Instead, the correct choice is the lesser-known \textit{Samuel} or \textit{uniform compactification} $\ucomp{M}$, that can be interpreted as the smallest compactification of $M$ that allows extension of all bounded Lipschitz functions. The use of this object for the study of Lipschitz and Lipschitz-free spaces has been pioneered by Weaver in \cite{Weaver2} (see Chapter 7 therein).

Once the correct setting has been chosen, it is possible to generalize the techniques and arguments employed for $\lipfree{M}$ to $\bidualfree{M}$. In fact, we derive most results for $\lipfree{M}$ as particular cases of the results for $\bidualfree{M}$. It is essential in our arguments to consider measures on $M$ to be just particular cases of measures on $\ucomp{M}$ that are concentrated on $M$. Since $\ucomp{M}$ is compact, the space of Radon measures on $\ucomp{M}$ is a dual Banach space and this allows the use of weak$^\ast$ compactness arguments. Moreover our key technical result, Theorem \ref{th:normal_measure}, states that a functional in $\bidualfree{M}$ represented by a measure on $\ucomp{M}$ is weak$^\ast$ continuous if and only if the measure is concentrated on $M$; this is a vast generalization of \cite[Proposition 2.1.6]{Weaver1}. These facts lie silently but crucially behind our characterizations of measure-induced functionals and all results on majorizable functionals.

On the way to proving our main results, we obtain some new knowledge about the structure of the bidual $\bidualfree{M}$. This space has not received as much attention as the Lipschitz-free space $\lipfree{M}$ and its study has been limited to particular cases (most notably $M=\RR^n$ in \cite{CKK_2019}), but it plays a fundamental role in several open problems on Lipschitz-free spaces. The most prominent one would be determining when $\lipfree{M}$ is complemented in its bidual $\ddual{\lipfree{M}}$, or even L-embedded. The complementability holds for $M=\RR^n$ as C\'uth, Kalenda and Kaplick\'y have shown in \cite{CKK_2019}, and the L-embeddability is satisfied when $M$ is compact and purely 1-unrectifiable (see \cite[Section 3]{AGPP_2021}), but not much is known otherwise. That includes the case of $\lipfree{\ell_1}$, the complementability of which would yield a positive answer to the important open problem whether $\ell_1$ is determined by its Lipschitz structure (see e.g. \cite[Problem 16]{GLZ_2014}). Analyzing the structure of $\bidualfree{M}$ in the general case could conceivably provide insight into these questions.

In our previous work \cite{AP_2020_normality}, we started this analysis by showing that the annular decomposition for elements of $\lipfree{M}$ introduced by Kalton in \cite{Kalton_2004} is also partially valid in $\bidualfree{M}$. Here, we take these ideas further and prove that any functional in $\bidualfree{M}$ can be canonically decomposed into a part that is ``concentrated at infinity'', a part that is ``concentrated at the base point'', and a part that is compatible with Kalton's decomposition. The latter class contains $\lipfree{M}$ and all elements represented by measures.

We also define an \textit{extended support} for functionals in $\bidualfree{M}$ as a certain subset of $\ucomp{M}$ that has similar properties as the support for $\lipfree{M}$.
We make heavy use of the algebraic structure of $\Lip_0(M)$ in order to establish the stronger properties of extended supports, which only hold for functionals that are compatible with Kalton's decomposition. Let us point out that the resulting theory bears some strong similarities with other areas of analysis, in particular with the theory of Schwartz distributions. For instance, compare Proposition \ref{pr:support_derivation} below with \cite[Theorem 6.25]{Rudin_FA}, describing functionals supported on a single point as a combination of evaluations and derivations at said point. Similarly, Corollary \ref{cr:positive_elements_free} below states that all positive elements of $\lipfree{M}$ can be represented as integration against a positive measure, a fact that also holds true for positive distributions (see \cite[Exercise 6.4]{Rudin_FA}).

\subsection{Summary of the paper}

Let us now briefly summarize the contents of this paper.

After this introduction, we gather in Section 2 all the prerequisite facts about Lipschitz-free spaces, including the theory of weighting operators and supports that was developed by the authors in previous papers \cite{AlPe_2020,AP_2020_normality,APPP_2020}, and about Radon measures on metric and compact spaces. We also define the uniform compactification and state its basic properties. Finally, we recall the notion of derivations as functionals on Lipschitz spaces.

In Section 3 we analyze the structure of the space $\bidualfree{M}$. In particular, we study the decomposition of its elements into parts that lie ``at infinity'' and ``away from infinity'' (see Corollary \ref{cor:decomposition_A_infinity}), and the validity of an extension of the concept of support in $\lipfree{M}$ to $\bidualfree{M}$. We find that the properties of such an extended support are somewhat less sharp than those of the previous notion, and they may fail at infinity (see Theorem \ref{th:extended_support_minimal}).

Section 4 is devoted to proving characterizations of those functionals in $\lipfree{M}$, resp. $\bidualfree{M}$, that are represented by measures on $M$, resp. $\ucomp{M}$ (see Theorems \ref{th:induced_elements} and \ref{th:induced_elements_bidual}), and of measures that yield continuous functionals (see Propositions \ref{pr:elements_induced_by_measure} and \ref{pr:elements_induced_by_measure_bidual}). Moreover we show that the concepts of support for measures and for functionals on Lipschitz spaces coincide. As a very important tool for handling $\lipfree{M}$ as a subspace of $\dual{\Lip_0(M)}$ we also show that a measure on $\ucomp{M}$ can only represent a functional in $\lipfree{M}$ if it is concentrated on $M$ to begin with (see Theorem \ref{th:normal_measure}).

Section 5 deals with the majorizable functionals on $\Lip_0(M)$. We prove that all majorizable elements of $\lipfree{M}$ can be expressed by measures on $M$. In $\bidualfree{M}$, the corresponding result only holds for functionals that are ``away from infinity'', and only up to a derivation at the base point (see Theorem \ref{th:majorizable_characterization}). We also show that all majorizable functionals on $\Lip_0(M)$ admit Jordan decompositions (see Corollary \ref{cr:minimal_majorant_general}) and use that fact to define an analog of the total variation for these functionals (see Definition \ref{def:variation}).

In Section 6, we provide a purely metric characterization of metric spaces $M$ for which every element of $\lipfree{M}$ is majorizable or can be represented as a Radon measure (see Corollary \ref{cr:rud}).

Finally, in Section 7 we collect some remarks and unanswered questions.

\section{Preliminaries}

Let us start by describing our notation and recalling basic concepts and facts that will be used throughout the paper. As usual, the closed unit ball of a Banach space $X$ will be denoted by $\ball{X}$, its unit sphere by $\sphere{X}$, and the dual action of $\dual{x}\in\dual{X}$ on $x\in X$ as $\duality{x,\dual{x}}=\dual{x}(x)$. We will write $\lspan(S)$, resp. $\conv(S)$, for the set of finite linear, resp. convex, combinations of elements of a subset $S\subset X$. Only real scalars will be considered.

\subsection{Metric and Lipschitz-free spaces}

In what follows, and unless specified otherwise, $M$ will always denote a complete pointed metric space with metric $d$ and base point $0$. The open ball of radius $r$ around $p\in M$ will be denoted $B(p,r)$. We will also use the notation
\begin{align*}
d(x,A) &= \inf\set{d(x,a):a\in A} \\
d(A,B) &= \inf\set{d(a,b):a\in A,b\in B} \\
\rad(A) &= \sup\set{d(a,0):a\in A} \\
\diam(A) &= \sup\set{d(a,b):a,b\in A}
\end{align*}
for $x\in M$ and $A,B\subset M$; the last two quantities will be called the \textit{radius} and \textit{diameter} of $A$, respectively.

The space of all real-valued Lipschitz functions on $M$ will be denoted $\Lip(M)$, and $\Lip_0(M)$ will consist of all $f\in\Lip(M)$ such that $f(0)=0$. For $f\in\Lip(M)$ we will denote its Lipschitz constant by $\lipnorm{f}$ and by its support we will mean the closed subset of $M$ given by
$$
\supp(f)=\cl{\set{x\in M: f(x)\neq 0}}.
$$
We will frequently use the function $\rho\in\Lip_0(M)$ defined by
$$
\rho(x)=d(x,0)
$$
for $x\in M$; obviously $\lipnorm{\rho}=1$. We shall also use McShane's extension theorem: \textit{if $N\subset M$ then any $f\in\Lip(N)$ admits at least one extension $F\in\Lip(M)$ such that $F\restrict_N=f$, $\lipnorm{F}=\lipnorm{f}$, $\sup F=\sup f$ and $\inf F=\inf f$}. 

It is well known that $\lipnorm{f}$ is a complete norm on $\Lip_0(M)$. For $x\in M$, the evaluation functional $f\mapsto f(x)$ will be denoted by $\delta(x)$. Then $\delta$ is an isometric embedding of $M$ into $\dual{\Lip_0(M)}$. By a \textit{finitely supported} functional on $\Lip_0(M)$ we will mean a finite linear combination of such evaluation functionals, i.e. an element of $\lspan\,\delta(M)$. The closed space generated by them
$$
\lipfree{M}=\cl{\lspan}\,\delta(M)
$$
will be called \textit{Lipschitz-free space} over $M$, and is easily seen to be a predual of $\Lip_0(M)$. The weak$^\ast$ topology induced by $\lipfree{M}$ on $\ball{\Lip_0(M)}$ coincides with the topology of pointwise convergence. It is currently not known whether $\lipfree{M}$ is always the unique predual of $\Lip_0(M)$, but any mention of the weak$^\ast$ topology of $\Lip_0(M)$ will always make reference to this predual. For further reference we recommend Weaver's book \cite{Weaver2} (where $\lipfree{M}$ is denoted $\text{\AE}(M)$).

Recall that the pointwise order is a vector space order in $\Lip_0(M)$ and $\Lip(M)$ that turns them into lattices with the operations $\vee$, $\wedge$ of pointwise maximum and minimum. We will thus say that $f\in\Lip(M)$ is positive if $f\geq 0$. An element $\phi\in\bidualfree{M}$ (or $\lipfree{M}$) is positive if $\duality{f,\phi}\geq 0$ for any $f\in\Lip_0(M)$ such that $f\geq 0$, and we denote it as $\phi\geq 0$; this also induces a vector space order in $\bidualfree{M}$ and $\lipfree{M}$. The set of all positive elements of an ordered Banach space $X$ will be denoted by $\pos{X}$, and the set of positive elements of $\ball{X}$ by $\pos{\ball{X}}$.  An isomorphism $T$ between ordered Banach spaces $X$ and $Y$ will be called \textit{order-preserving} if $T(X^+)=Y^+$.
Let us note that, although $\Lip_0(M)$ is a Banach space and a vector lattice, it is not a Banach lattice, i.e. the norm is not monotone. Indeed, consider for example functions $f(x)=x$ and $g(x)=0\vee(2x-1)$ on $M=[0,1]$.

Since $\rho\geq f$ for any $f\in\ball{\Lip_0(M)}$, we have $\norm{\phi}=\duality{\rho,\phi}$ for any $\phi\in\pos{(\bidualfree{M})}$. Consequently, the norm of the sum of positive elements of $\bidualfree{M}$ is just the sum of the norms. We will use these facts repeatedly without further notice. The next result will also be needed. The statement for $\lipfree{M}$ is folklore and can be found e.g. in \cite[Lemma 2.6]{AmPu_2016}; the argument for $\bidualfree{M}$ is essentially the same but we could not locate it elsewhere, so we include its proof for reference.

\begin{lemma}
\label{lm:positive_sum_closure}
Every $m\in\pos{\lipfree{M}}$ is the limit of a sequence $(m_n)$ of positive, finitely supported elements of $\lipfree{M}$, and every $\phi\in\pos{(\bidualfree{M})}$ is the weak$^\ast$ limit of a net $(m_i)$ of positive, finitely supported elements of $\lipfree{M}$.
\end{lemma}

\begin{proof}
We will only give the proof of the second statement. Let $A$ be the set of positive, finitely supported elements of $\lipfree{M}$, and suppose that there exists $\phi\in\pos{(\bidualfree{M})}$ such that $\phi\notin\wscl{A}$. By the Hahn-Banach separation theorem, there is $f\in\Lip_0(M)$ such that
$$
\duality{f,\phi}>\sup\set{\duality{f,\psi}:\psi\in\wscl{A}}\geq 0 .
$$
In particular, taking $\psi=a\delta(x)$ for $a>0$ and $x\in M$, we get $af(x)<\duality{f,\phi}$. Since this is true for any $a>0$, it follows that $f(x)\leq 0$. Hence $f\leq 0$, and the positivity of $\phi$ implies that $\duality{f,\phi}\leq 0$, a contradiction.
\end{proof}

We will also need the following fact:

\begin{lemma}[{\cite[Lemma 3.7]{APPP_2020}}]
\label{lm:positive_functional_lemma}
Let $\phi\in\bidualfree{M}$ and suppose that $0\leq\phi\leq m$ for some $m\in\lipfree{M}$. Then $\phi\in\lipfree{M}$.
\end{lemma}

Finally, let us recall that the pointwise product of Lipschitz functions is not Lipschitz in general. However, for $f,g\in\Lip(M)$ we have
$$
\lipnorm{fg}\leq\lipnorm{f}\norm{g}_\infty+\norm{f}_\infty\lipnorm{g}
$$
and therefore the product of \textit{bounded} Lipschitz functions is again Lipschitz. In particular, for bounded $M$,
$$
\lipnorm{fg}\leq2\rad(M)\lipnorm{f}\lipnorm{g},
$$
so the space $\Lip_0(M)$ is closed under pointwise products and it becomes an algebra. Note that it is not a Banach algebra in general, but the product is continuous and $\Lip_0(M)$ can be equivalently renormed with a submultiplicative norm.

\subsection{Supports in Lipschitz-free spaces}
\label{subsection:intersection_property}

For a subset $K$ of $M$ (that contains $0$), the space $\lipfree{K}$ can and will be identified with the closed subspace $\cl{\lspan}\,\delta(K)$ of $\lipfree{M}$. It is also known that $\lipfree{K}^\perp=\ideal{K}$ and $\ideal{K}_\perp=\lipfree{K}$ where
\begin{equation}
\label{eq:ideal_def}
\ideal{K}=\set{f\in\Lip_0(M):f(x)=0\text{ for all }x\in K} .
\end{equation}
In \cite{AlPe_2020,APPP_2020}, it was shown that Lipschitz-free spaces satisfy the intersection property 
\begin{equation}
\label{eq: intersection_property}
\bigcap_i\lipfree{K_i}=\mathcal{F}\left(\bigcap_i K_i\right)
\end{equation}
for any family $(K_i)$ of closed subsets of $M$, and the following concept was introduced: the \textit{support} of an element $m\in\lipfree{M}$ is defined as
$$
\supp(m)=\bigcap\set{K\subset M: \text{$K$ is closed and $m\in\lipfree{K}$}} .
$$
This set is closed and separable, as can be readily seen by approximating $m$ with a sequence of finitely supported elements, and the intersection property implies that $m$ is always contained in $\lipfree{\supp(m)}$. Moreover $\supp(m)=\varnothing$ if and only if $m=0$. Notice that $\supp(m)$ is a finite set if and only if $m\in\lspan\,\delta(M)$, so the use of the term ``finitely supported'' is consistent.

The following two propositions describe alternative characterizations of the support that will be relevant later.

\begin{proposition}[{\cite[Proposition 2.6]{APPP_2020}}]
\label{pr:equiv_char_support}
Let $m\in\lipfree{M}$ and $K\subset M$ be closed. Then $\supp(m)\subset K$ if and only if $\duality{m,f}=\duality{m,g}$ for any $f,g\in\Lip_0(M)$ such that $f\restrict_K=g\restrict_K$.
\end{proposition}

\begin{proposition}[{\cite[Proposition 2.7]{APPP_2020}}]
\label{pr:equiv_points_support}
Let $m\in\lipfree{M}$ and $x\in M$. Then $x\in\supp(m)$ if and only if for every neighborhood $U$ of $x$ there exists a function $f\in\Lip_0(M)$ whose support is contained in $U$ and such that $\duality{m,f}\neq 0$.
\end{proposition}

We will also make extensive use of the weighting operation described in \cite[Lemma 2.3]{APPP_2020} as follows. Let $h\in\Lip(M)$ have bounded support, then the operator $\wop{h}\colon \Lip_0(M)\rightarrow\Lip_0(M)$ defined by $\wop{h}(f)=fh$ for $f\in\Lip_0(M)$ satisfies
\begin{equation}
\label{eq:weighting_op_norm}
\norm{\wop{h}} \leq \norm{h}_\infty + \rad(\supp(h))\lipnorm{h}
\end{equation}
and is \weaks-\weaks-continuous, hence the adjoint operator $\dual{(\wop{h})}$ takes $\phi\in\bidualfree{M}$ into $\dual{(\wop{h})}(\phi)=\phi\circ\wop{h}\in\bidualfree{M}$, and if $m\in\lipfree{M}$ then $m\circ\wop{h}\in\lipfree{M}$ as well. Moreover, it follows easily from Proposition \ref{pr:equiv_char_support} that
\begin{equation}
\label{eq:support_Th}
\supp(m\circ\wop{h})\subset\supp(m)\cap\supp(h) .
\end{equation}

Weighting will primarily be used with Urysohn-lemma-like functions defined in the following way. Let $A,B$ be two subsets of $M$ such that $d(A,B)>0$; $A$ represents a ``region of interest'' where we want to focus, and $B$ is a region that we want to ignore. If $A$ is bounded, we define the function $h(x)=0\vee(1- d(x,A)/d(A,B))$ for every $x\in M$. Then clearly $0\leq h\leq 1$, $h=1$ in $A$, $h=0$ in $B$, $\lipnorm{h}\leq 1/d(A,B)$ and the support of $h$ is bounded. Hence, $\wop{h}$ is an operator on $\Lip_0(M)$ and for every $m\in\lipfree{M}$, one has $\duality{m\circ\wop{h},f}=\duality{m,f}$ whenever $\supp(f)\subset A$ and $\duality{m\circ\wop{h},f}=0$ whenever $\supp(f)\subset B$. If on the other hand $B$ is bounded, then defining $h(x)=1\wedge d(x,B)/d(A,B)$ for every $x\in M$ yields a bounded support for the function $1-h$, and $\wop{h}=I-\wop{1-h}$ is again an operator on $\Lip_0(M)$ with the same properties (here $I$ denotes the identity operator). If neither $A$ nor $B$ are bounded, it is in general not possible to construct such an operator.

Let us now define some standard weighting functions that will be used often in this paper, where the regions of interest are balls or annuli centered at the base point, or their complements. For all $x\in M$ and $n\in\ZZ$, let
\begin{equation}
\label{eq:H_n}
H_n(x) = \begin{cases}
1 &\text{, if } \rho(x)\leq 2^n \\
2-2^{-n}\rho(x) &\text{, if } 2^n\leq \rho(x)\leq 2^{n+1} \\
0 &\text{, if } 2^{n+1}\leq \rho(x) .
\end{cases}
\end{equation}
Further, let
\begin{equation}
\label{eq:G_n}
G_n(x) = 1-H_n(x)
\end{equation}
and
\begin{equation}
\label{eq:Lambda_n}
\Lambda_n(x) = H_n(x)-H_{n-1}(x)=G_{n-1}(x)H_n(x)
\end{equation}
and for $n\in\NN$ let
\begin{equation}
\label{eq:Pi_n}
\Pi_n(x) = H_n(x)-H_{-(n+1)}(x) = \sum_{k=-n}^n\Lambda_{k}(x).
\end{equation}
Notice that $\lipnorm{H_n}\leq 2^{-n}$ and $\rad(\supp(H_n))\leq 2^{n+1}$, so \eqref{eq:weighting_op_norm} yields $\norm{\wop{H_n}}\leq 3$ and $\norm{\wop{G_n}}\leq 1+\norm{\wop{H_n}}\leq 4$. Similarly we get $\norm{\wop{\Lambda_n}}\leq 5$. Note also that $\norm{\wop{\Pi_n}}\leq\norm{\wop{H_n}}+\norm{\wop{H_{-(n+1)}}}\leq 6$.
So all of these functions generate operators on $\Lip_0(M)$. Moreover, for every $f\in\Lip_0(M)$ the sequences $(\wop{H_n}(f))$, $(\wop{G_{-n}}(f))$ and $(\wop{\Pi_n}(f))$ (where $n\in\NN$) are bounded and converge pointwise, hence weak$^\ast$, to $f$, and if $f\geq 0$ then convergence is monotonic. Finally, notice that $\wop{H_m}\wop{H_n}=\wop{H_{\min\set{m,n}}}$, $\wop{G_m}\wop{G_n}=\wop{G_{\max\set{m,n}}}$ and $\wop{\Pi_m}\wop{\Pi_n}=\wop{\Pi_{\min\set{m,n}}}$.

The functions $\Lambda_n$ may be used to obtain a decomposition of functionals $\phi\in\bidualfree{M}$ similar to that constructed by Kalton in \cite[Section 4]{Kalton_2004}. Indeed, by \cite[Lemma 3]{AP_2020_normality} we have
\begin{equation}
\label{eq:kalton_absconv}
\sum_{n\in\ZZ}\norm{\phi\circ\wop{\Lambda_n}}\leq 45\norm{\phi}
\end{equation}
so that
\begin{equation}
\label{eq:kalton_finite_part}
\kaltonsum{\phi} := \sum_{n\in\ZZ}\phi\circ\wop{\Lambda_n} = \lim_{n\rightarrow\infty}\phi\circ\wop{\Pi_n}
\end{equation}
is an element of $\bidualfree{M}$ for every $\phi\in\bidualfree{M}$, and convergence of the limit and the series is in norm. Moreover, if $\phi\in\lipfree{M}$ then $\kaltonsum{\phi}=\phi$.

\subsection{Radon measures}

Measure-theoretic notions such as regularity or Radonness appear not to be uniquely defined in the literature, so let us briefly establish the terminology that will be used throughout this document.
Let $X$ be a Hausdorff space and let $\mu$ be a Borel measure on $X$, i.e. a measure defined on the Borel $\sigma$-algebra of $X$. If $\mu$ is positive (but not necessarily finite), then we will say that it is
\begin{itemize}
\item \textit{inner regular} if $\mu(E)=\sup\set{\mu(K):K\subset E\text{ compact}}$ for every Borel set $E\subset X$,
\item \textit{outer regular} if $\mu(E)=\inf\set{\mu(U):U\supset E\text{ open}}$ for every Borel set $E\subset X$,
\item \textit{regular} if it is both inner and outer regular,
\item \textit{Radon} if it is regular and finite.
\end{itemize}
Notice that if $\mu$ is finite then inner regularity implies regularity, and if additionally $X$ is compact then each of the four conditions above implies the other three. If $\mu$ is a finite signed (real-valued) measure instead, then it is said to have each of the previous properties if the total variation measure $\abs{\mu}$, defined by
$$
\abs{\mu}(A)=\sup\set{\sum_{i=1}^n\abs{\mu(A_i)}: A_i \textup{ pairwise disjoint Borel and } A=\bigcup_{i=1}^n A_i}
$$
for any Borel set $A\subset X$, does have the property.

Given a Borel measure $\mu$ on $X$ and a Borel set $A\subset X$, we will denote by $\mu\restrict_A$ the restriction of $\mu$ to $A$ defined by $\mu\restrict_A(E)=\mu(E\cap A)$ for any Borel set $E\subset X$. This is again a Borel measure on $X$ and it satisfies $\abs{\mu\restrict_A}=\abs{\mu}\restrict_A$. We will say that $\mu$ is \textit{concentrated on} $A$ for some Borel set $A\subset X$ if $\mu=\mu\restrict_A$, or, equivalently, $\abs{\mu}(X\setminus A)=0$. On the other hand, if $\mu$ is a Borel measure on some Borel subset $A$ of $X$, we define the extension $\nu$ of $\mu$ to $X$ by $\nu(E)=\mu(E\cap A)$ for every $E\subset X$. Then $\nu$ is a Borel measure on $X$.

Recall that every finite signed Borel measure $\mu$ on $X$ admits a unique \textit{Jordan decomposition} $\mu=\mu^+-\mu^-$ where $\mu^+,\mu^-$ are finite positive Borel measures that are minimal with respect to that property, i.e. if $\mu=\lambda^+-\lambda^-$ for positive measures $\lambda^+,\lambda^-$ then $\lambda^\pm\geq\mu^\pm$. These measures also satisfy the identity $\abs{\mu}=\mu^++\mu^-$. Moreover, there is at least one \textit{Hahn decomposition} of $X$ associated to $\mu$, i.e. a partition $X=A^+\cup A^-$ into two disjoint Borel subsets such that $\mu^+=\mu\restrict_{A^+}$ and $\mu^-=-\mu\restrict_{A^-}$.

The \textit{support} of a measure $\mu$ is defined as the set $\supp(\mu)$ of points $x\in X$ such that $\abs{\mu}(U)>0$ for every open neighborhood $U$ of $x$. It is always a closed set. If $\mu$ is Radon, then $\supp(\mu)=\varnothing$ if and only if $\mu=0$ \cite[Theorem 7.2.9]{Bogachev}.

In our analysis, we will restrict ourselves to Borel measures defined either on metric spaces $M$ or on compact spaces $X$. We will not assume measures to be finite, positive, or regular unless stated explicitly. Note that if $\mu$ is any finite Borel measure defined on $M$ then it is automatically outer regular, and if $M$ is complete and separable then $\mu$ is also Radon (see Theorems 1.4.8 and 7.1.7 in \cite{Bogachev}, where a slightly different notation is used). Let us also mention that the support of every finite Borel measure on a metric space is separable (see e.g. \cite[Lemma 2.1]{Lawson_2017}).

\subsection{The uniform compactification}

Recall that a \textit{compactification} of a completely regular Hausdorff space, in particular of a metric space $M$, is a compact Hausdorff space that contains a dense subset that is homeomorphic to $M$ (and can thus be identified with $M$). Compactifications $X,Y$ of $M$ may be partially ordered by declaring that $X\leq Y$ if there is a continuous map from $Y$ onto $X$ whose restriction to $M$ is the identity; $X$ and $Y$ are then equivalent if $X\leq Y$ and $Y\leq X$. The largest compactification under this ordering is the well-known Stone-\v{C}ech compactification $\beta M$, which is characterized (up to equivalence) by the fact that any bounded continuous function on $M$ can be extended to a continuous function on $\beta M$. This renders it a useful tool in the study of spaces of continuous functions, and in particular of Lipschitz spaces; it has been used, for instance, in \cite{AlGu_2019,AlPe_2020,Weaver_1996,Weaver2}.

In this paper, it will be useful to consider metric spaces $M$ to be embedded into some compactification $X$ because that will allow us to consider measures on $M$ as elements of the space $\meas{X}$ of Radon measures on $X$, which is then a dual Banach space. However, the common Stone-\v{C}ech compactification carries an important drawback for our intended purposes, namely the fact that Lipschitz functions do not necessarily separate points of $\beta M$. In fact, by \cite[Theorem 3.4]{Woods_1995} they only do if there is a compact subset $K$ of $M$ such that $M\setminus U$ is uniformly discrete for every open set $U\supset K$. Let us illustrate this phenomenon with a simple example:

\begin{example}
\label{ex:lip_dont_separate}
Let $M\subset\RR$ consist of $0$ and the points $x_n=n$ and $y_n=n+2^{-n}$ for $n\in\NN$. Then there is a subnet $(x_{n_i},y_{n_i})$ of the sequence $(x_n,y_n)$ such that $x_{n_i}$ and $y_{n_i}$ converge to points $\xi$ and $\eta$ of $\beta M$, respectively. Since $M$ is topologically discrete, any function on $M$ is continuous and we may take $f\in C(M)$ such that $f(x_n)=0$ and $f(y_n)=1$ for every $n$. Then $f(\xi)=0$ and $f(\eta)=1$, hence $\xi\neq\eta$. However $\xi$ and $\eta$ cannot be separated by Lipschitz functions: indeed, if $f\in\Lip(M)$ then $\abs{f(x_n)-f(y_n)}\leq 2^{-n}\lipnorm{f}$, and taking limits yields $f(\xi)=f(\eta)$.
\end{example}

It will thus be more convenient to use a smaller compactification than $\beta M$. A classical result by Gelfand and \v{S}ilov (see e.g. \cite[Proposition 14.2.2]{Semadeni}) establishes a correspondence between compactifications $X$ of $M$ and closed subalgebras $\mathcal{A}$ of bounded continuous functions on $M$ that contain the constant functions and separate points from closed subsets of $M$, such that $\mathcal{A}$ consists precisely of those bounded continuous functions that can be extended continuously to $X$. In our setting, the algebra of interest consists of the bounded Lipschitz functions on $M$ (or rather its closure, the bounded uniformly continuous functions), and the corresponding compactification is the \textit{uniform} or \textit{Samuel compactification} of $M$, which we will denote as $\ucomp{M}$. The suggested reference for information about this object is the paper \cite{Woods_1995} by Woods. The following statement collects its defining properties:

\begin{proposition}[{\cite[Corollary 2.4]{Woods_1995}}]
\label{pr:uc_definition}
Let $M$ be a metric space. Then there exists a compactification $\ucomp{M}$ of $M$ with the following properties:
\begin{enumerate}[label={\upshape{(\roman*)}}]
\item \label{uc:extension} Every bounded, uniformly continuous function $f\colon M\rightarrow\RR$ can be extended uniquely to a continuous function $\ucomp{f}\colon\ucomp{M}\rightarrow\RR$.
\item \label{uc:separation} Given two subsets $A,B\subset M$, their closures in $\ucomp{M}$ are disjoint if and only if $d(A,B)>0$.
\end{enumerate}
Moreover, these properties determine $\ucomp{M}$ uniquely up to equivalence.
\end{proposition}

\noindent We will denote the closure of $A\subset M$ in $\ucomp{M}$ by $\ucl{A}$. Note that $\ucl{A}$ and $\ucomp{A}$ are equivalent compactifications of the metric space $A$ by \cite[Theorem 2.9]{Woods_1995}, so this notation shall lead to no confusion. We will also need the following converse to property \ref{uc:extension}:

\begin{proposition}[{\cite[Theorem 2.5]{Woods_1995}}]
\label{pr:uc_restriction}
If $f\colon M\rightarrow\RR$ is bounded and can be extended continuously to $\ucomp{M}$, then $f$ is uniformly continuous.
\end{proposition}

Combining property \ref{uc:separation} with the usual separation axioms immediately yields the following metric separation property, that we will be using repeatedly:

\begin{proposition}
\label{pr:uc_separation}
Let $K,L$ be disjoint closed subsets of $\ucomp{M}$. Then there are disjoint open neighborhoods $V,W$ of $K,L$ such that $d(V\cap M,W\cap M)>0$.
\end{proposition}

\noindent Thus disjoint closed subsets of $\ucomp{M}$ can be separated by (extensions of) Lipschitz functions on $M$. In particular, Lipschitz functions separate points of $\ucomp{M}$ as required. In fact, $\ucomp{M}$ may be identified with the quotient space of $\beta M$ obtained by identifying those points that cannot be separated by Lipschitz functions.

Property \ref{uc:extension} only covers the extension of bounded uniformly continuous (in particular, Lipschitz) functions on $M$, but in general we will need to deal with extensions of unbounded functions, too. It is proved (for a more general setting) in \cite[Proposition 1.4]{GaJa_2004} that they can be continuously extended to $\ucomp{M}$ if we enlarge the range from $\RR$ to its one-point compactification $\RR\cup\set{\infty}$; see also Section 1 of \cite{GaMe_2017}. The following version of that property will be more useful for our purposes:

\begin{proposition}
Every Lipschitz function $f\colon M\rightarrow\RR$ can be extended uniquely to a continuous function $\ucomp{f}\colon\ucomp{M}\rightarrow [-\infty,+\infty]$.
\end{proposition}

\noindent Indeed, this follows easily using \ref{uc:separation} from either the $\RR\cup\set{\infty}$ version or, more directly, from the Taimanov extension theorem (see e.g. \cite[Theorem 3.2.1]{Engelking}).

Finally, let us consider the elements $\zeta\in\ucomp{M}$ that satisfy any of the following equivalent conditions:
\begin{itemize}
\item $\zeta$ is the limit of a bounded net in $M$,
\item $\ucomp{\rho}(\zeta)<\infty$ (for any choice of base point in $M$),
\item $\abs{\ucomp{f}(\zeta)}<\infty$ for all $f\in\Lip(M)$.
\end{itemize}
The set $\rcomp{M}$ of all such elements is called the \textit{Lipschitz realcompactification} of $M$ in \cite{GaMe_2017}. As implied by the name, $\rcomp{M}$ is realcompact, i.e. homeomorphic to a closed subset of a product of real lines (see the construction in \cite{GaMe_2017} and \cite[Theorem 3.11.3]{Engelking}), but not compact in general.
In fact, we have $M\subset\rcomp{M}\subset\ucomp{M}$, and it is clear that $\rcomp{M}=\ucomp{M}$ if and only if $M$ is bounded and $\rcomp{M}=M$ if and only if $M$ is proper, i.e. it has the Heine-Borel property. Notice also that the evaluation functional $\delta(\zeta)\colon f\mapsto\ucomp{f}(\zeta)$ is an element of $\bidualfree{M}$ if and only if $\zeta\in\rcomp{M}$, and its norm is $\norm{\delta(\zeta)}=\ucomp{\rho}(\zeta)$. Let us also mention that $\ucomp{(fg)}(\zeta)=\ucomp{f}(\zeta)\ucomp{g}(\zeta)$ for any $\zeta\in\rcomp{M}$ and $f,g\in\Lip(M)$ such that $fg\in\Lip(M)$; this is not valid for $\zeta\notin\rcomp{M}$, as an indeterminate limit of type $0\times\infty$ may appear.

\subsection{Derivations}

Let us now introduce a class of functionals on $\Lip_0(M)$ that will play a role in what follows.
In \cite[Section 7.5]{Weaver2}, Weaver defines a derivation at a point $x\in M$ (or $\ucomp{M}$, more generally) as an element $\phi\in\bidualfree{M}$ that satisfies the relation
$$
\duality{fg,\phi} = \duality{f,\phi}\cdot g(x) + f(x)\cdot\duality{g,\phi}
$$
for any $f,g\in\Lip_0(M)$. In general $fg$ is not a Lipschitz function, and this is likely one of the reasons why the original definition is only given for bounded $M$; another one is that $f(x)$ may not be finite if $x\notin\rcomp{M}$, which is possible when $M$ is unbounded. In order to eliminate these restrictions and extend the domain of the definition, we prefer to use the following alternative formulation:

\begin{definition}
\label{def:derivation}
Let $\phi\in\bidualfree{M}$ and $\zeta\in\ucomp{M}$. We say that $\phi$ is a \textit{derivation at $\zeta$} if $\duality{f,\phi}=0$ for any $f\in\Lip_0(M)$ such that $\ucomp{f}$ is constant in a neighborhood of $\zeta$.
\end{definition}

Lemma 7.47 in \cite{Weaver2} asserts that both definitions are equivalent in the original setting, i.e. when $M$ is bounded, and it follows easily that they are also equivalent in general for $\zeta\in\rcomp{M}$. This formulation makes it obvious that nontrivial derivations at $x\in M$ can only exist if $x$ is not an isolated point. Also, derivations at $\zeta\in\rcomp{M}\setminus\set{0}$ can never be positive, since given $f\in\Lip_0(M)$ it is easy to construct $g\in\pos{\Lip_0(M)}$ such that $g-f$ is constant in a neighborhood of $\zeta$. On the other hand, there are always positive derivations at $0$, assuming it is not isolated: let $x_n\rightarrow 0$ and $m_n=\delta(x_n)/d(x_n,0)\in\ball{\lipfree{M}}$, then $(m_n)$ must have a subnet that converges weak$^\ast$ to $\phi\in\bidualfree{M}$. This $\phi$ is clearly a positive derivation at $0$ such that $\norm{\phi}=\duality{\rho,\phi}=1$.

It is interesting to note that derivations at different points of $\rcomp{M}$ are ``orthogonal'' to each other and to weak$^\ast$ continuous functionals in the following sense:

\begin{proposition}
\label{pr:norm_sum_deriv}
Let $m\in\lipfree{M}$, and let $(\phi_i)\subset\bidualfree{M}$ be a sequence of derivations at different points of $\rcomp{M}$ such that $\sum\norm{\phi_i}<\infty$. Then
$$
\norm{m+\sum_{i=1}^\infty\phi_i} = \norm{m}+\sum_{i=1}^\infty\norm{\phi_i} .
$$
\end{proposition}

\begin{proof}
It is clearly enough to prove the theorem for a finite sum $m+\phi_1+\ldots+\phi_n$, and we may assume that each $\phi_i$ is nonzero. Fix $\varepsilon>0$ and choose $f,g_i\in\sphere{\Lip_0(M)}$ such that $\duality{m,f}=\norm{m}$ and $\duality{g_i,\phi_i}>\norm{\phi_i}-\frac{\varepsilon}{n}$ for $i=1,\ldots,n$. Suppose that $\phi_i$ is a derivation at $\zeta_i\in\rcomp{M}$ and let $A$ be the set of those $\zeta_i$ that belong to $M$. Let $\mathfrak{F}$ be the family of all finite subsets of $M\setminus (A\cup\set{0})$, directed by inclusion. We will construct a net $(h_E)_{E\in\mathfrak{F}}$ in $\Lip_0(M)$ such that $\lipnorm{h_E}\leq 1+4\varepsilon$ for every $E\in\mathfrak{F}$.

Fix $E\in\mathfrak{F}$. Since the $\zeta_i$ are all different from each other and not contained in $E$, we may find disjoint open neighborhoods $U_i$ of $\zeta_i$ such that the sets $U_i\cap M$ are at a positive distance from each other and from $E$; also, if $\zeta_i\neq 0$ then we also assume that $d(U_i\cap M,0)>0$. Let $r>0$ be smaller than all of those distances. For each $i=1,\ldots,n$ let
$$
V_i=\set{\xi\in\ucomp{M}: \abs{\ucomp{g}_i(\xi)-\ucomp{g}_i(\zeta_i)}<\varepsilon r \text{ and } \abs{\ucomp{f}(\xi)-\ucomp{f}(\zeta_i)}<\varepsilon r}
$$
which is clearly an open neighborhood of $\zeta_i$, and let $W_i=U_i\cap V_i$. Define the function $h_E$ on the set $\set{0}\cup E\cup\bigcup_{i=1}^n(W_i\cap M)$ by
$$
h_E(x)=\begin{cases}
f(x) &,\, x\in E\cup\set{0} \\
g_i(x)-\ucomp{g}_i(\zeta_i)+\ucomp{f}(\zeta_i) &,\, x\in W_i\cap M
\end{cases}
$$
for $i=1,\ldots,n$. Notice that if $\zeta_i=0$ for some $i$ then both cases yield the same value $h_E(0)=0$.

Now let us estimate $\lipnorm{h_E}$. It is clear that $\abs{h_E(x)-h_E(y)}\leq d(x,y)$ if $x,y$ belong to $E\cup\set{0}$ or to $W_i\cap M$ for the same $i$. 
If $x\in W_i\cap M$ and $y\in W_j\cap M$ for $i\neq j$ then $d(x,y)\geq r$ and
\begin{align*}
\abs{h_E(x)-h_E(y)} &\leq \abs{g_i(x)-\ucomp{g}_i(\zeta_i)}+\abs{g_j(y)-\ucomp{g}_j(\zeta_j)}+\abs{\ucomp{f}(\zeta_i)-\ucomp{f}(\zeta_j)} \\
&\leq 2\varepsilon r+\abs{\ucomp{f}(\zeta_i)-f(x)}+\abs{\ucomp{f}(\zeta_j)-f(y)}+\abs{f(x)-f(y)} \\
&\leq 4\varepsilon r+d(x,y) \\
&\leq (1+4\varepsilon)d(x,y)
\end{align*}
Otherwise, if $x\in W_i\cap M$ and $y\in E$ then $d(x,y)\geq r$ again and
\begin{align*}
\abs{h_E(x)-h_E(y)} &\leq \abs{g_i(x)-\ucomp{g}_i(\zeta_i)}+\abs{\ucomp{f}(\zeta_i)-f(y)} \\
&\leq \varepsilon r+\abs{\ucomp{f}(\zeta_i)-f(x)}+\abs{f(x)-f(y)} \\
&\leq 2\varepsilon r+d(x,y) \\
&\leq (1+2\varepsilon)d(x,y) .
\end{align*}
So $\lipnorm{h_E}\leq 1+4\varepsilon$ as claimed. Finally, extend $h_E$ to $M$ using McShane's theorem.

We have thus built a bounded net $(h_E)_{E\in\mathfrak{F}}$ in $\Lip_0(M)$. Now notice that $h_E$ converges pointwise to $f$. Indeed, $h_E(x)=f(x)$ for all $x\in A\cup\set{0}$ and all $E\in\mathfrak{F}$, and $h_E(x)=f(x)$ for any $x\in M\setminus (A\cup\set{0})$ whenever $E\supset\set{x}$, by construction. Thus $(h_E)$ converges weak$^\ast$ to $f$, and therefore we can choose $E\in\mathfrak{F}$ such that $\duality{m,h_E}>\duality{m,f}-\varepsilon$. By construction, $h_E-g_i$ is constant in a neighborhood of each $\zeta_i$, therefore $\duality{h_E,\phi_i}=\duality{g_i,\phi_i}$. Putting it all together, we have
$$
\duality{h_E,m+\sum_{i=1}^n\phi_i} > \duality{m,f}-\varepsilon+\sum_{i=1}^n\duality{g_i,\phi_i} > \norm{m}+\sum_{i=1}^n\norm{\phi_i}-2\varepsilon
$$
and hence
$$
\norm{m+\sum_{i=1}^n\phi_i} > \frac{\norm{m}+\sum_{i=1}^n\norm{\phi_i}-2\varepsilon}{\lipnorm{h_E}} \geq \frac{\norm{m}+\sum_{i=1}^n\norm{\phi_i}-2\varepsilon}{1+4\varepsilon} .
$$
Letting $\varepsilon\rightarrow 0$ yields the desired result.
\end{proof}

\section{A notion of support for elements of \texorpdfstring{$\bidualfree{M}$}{Lip0(M)*}}

In this section, we will propose a generalization of the concept of support for elements of $\lipfree{M}$ that is applicable to elements of its bidual. Unfortunately, some of its properties break down for functionals with content that ``lies at infinity''. In order to make this statement more precise, we will first analyze the structure of $\bidualfree{M}$ and establish a decomposition of general functionals into elements which concentrate at different domains.

\subsection{Decomposition of functionals by domains}

Let us first make the following observation: for any $\phi\in\bidualfree{M}$, the sequence $(\phi\circ\wop{H_n})$ is Cauchy. Indeed, for $m>n\in\NN$ we have that
$$
\norm{\phi\circ\wop{H_m}-\phi\circ\wop{H_n}} = \norm{\sum_{k=n+1}^m \phi\circ\wop{\Lambda_k}} \leq \sum_{k=n+1}^\infty\norm{\phi\circ\wop{\Lambda_k}}
$$
can be made arbitrarily small by \eqref{eq:kalton_absconv}. So $(\phi\circ\wop{H_n})$ converges in norm to a functional in $\bidualfree{M}$ that can be interpreted as ``the part of $\phi$ that is concentrated away from infinity'' and since $\phi\circ\wop{G_n}=\phi-\phi\circ\wop{H_n}$, the limit of $(\phi\circ\wop{G_n})$ also exists and can be thought of as ``the part of $\phi$ that lies at infinity''. Analogously, $(\phi\circ\wop{G_{-n}})$ and $(\phi\circ\wop{H_{-n}})$ converge in norm and the limits can be understood as ``the part of $\phi$ that is concentrated away from the base point'' and ``the part that lies at the base point'', respectively.

With this idea in mind, let us introduce some terminology for the following classes of functionals:

\begin{definition}
\label{def:separation_classes}
Let $\phi\in\bidualfree{M}$. We say that $\phi$
\begin{itemize}
\item \textit{is concentrated at infinity} if $\phi\circ\wop{H_n}=0$ for all $n\in\NN$,
\item \textit{avoids infinity} if $\lim_n \phi\circ\wop{H_n}=\phi$,
\item \textit{avoids infinity strongly}  if $\phi\circ\wop{H_n}=\phi$ for some $n\in\NN$ (we will usually abbreviate this and say that $\phi$ is \textit{strongly bounded}),
\item \textit{is concentrated at $0$} if $\phi\circ\wop{G_{-n}}=0$ for all $n\in\NN$,
\item \textit{avoids $0$} if $\lim_n \phi\circ\wop{G_{-n}}=\phi$,
\item \textit{avoids $0$ strongly} if $\phi\circ\wop{G_{-n}}=\phi$ for some $n\in\NN$.
\end{itemize}
\end{definition}

Before we continue, let us mention some easy facts about these classes. If $\phi\in\bidualfree{M}$ avoids $0$ (resp. infinity) strongly then it avoids $0$ (resp. infinity), and if $\phi$ is concentrated at infinity then it avoids $0$ strongly and vice versa. It is also not difficult to verify that $\phi\in\bidualfree{M}$ avoids $0$ and infinity if and only if $\phi=\lim_n\phi\circ\wop{\Pi_n}$. Indeed, for any $n\in\NN$, we can write
\begin{align*}
\phi &= \phi\circ\wop{H_{-(n+1)}} + \phi\circ\wop{\Pi_n} + \phi\circ\wop{G_n}\\
&=\phi-\phi\circ\wop{G_{-(n+1)}} + \phi\circ\wop{\Pi_n} + \phi-\phi\circ\wop{H_n},
\end{align*}
so clearly $\phi=\lim_n\phi\circ\wop{\Pi_n}$ whenever $\phi$ avoids $0$ and infinity. For the converse, assuming that $\phi=\lim_k\phi\circ\wop{\Pi_k}$, we obtain
\begin{align*}
\phi\circ\wop{H_n} &= \pare{\lim_{k\rightarrow\infty} \phi\circ\wop{\Pi_{k-1}}}\circ\wop{H_n} \\
&= \pare{\lim_{k\rightarrow\infty} \phi\circ\wop{G_{-k}H_{k-1}}}\circ\wop{H_n}= \lim_{k\rightarrow\infty} \phi\circ\wop{G_{-k}H_n}
\end{align*}
and therefore
$$
\lim_{n\rightarrow\infty} \phi\circ\wop{H_n} = \lim_{n\rightarrow\infty}\lim_{k\rightarrow\infty} \phi\circ\wop{G_{-k}H_n} = \lim_{n\rightarrow\infty}\lim_{k\rightarrow\infty} \sum_{i=-k+1}^n \phi\circ\wop{\Lambda_i} = \phi
$$
because the series \eqref{eq:kalton_finite_part} converges absolutely by \eqref{eq:kalton_absconv}.
Since $\wop{\Pi_n}(f)$ converges weak$^\ast$ to $f$ for every $f\in\Lip_0(M)$, it follows from the remark before Definition \ref{def:separation_classes} that every element of $\lipfree{M}$ avoids $0$ and infinity. 

Note next that if $0$ is an isolated point of $M$ then every functional $\phi\in\bidualfree{M}$ avoids $0$ strongly, and if $M$ is bounded then every $\phi$ is strongly bounded. In fact, strongly bounded functionals can be completely identified with functionals over some space $\Lip_0(N)$ where $N$ is bounded - only we are considering them as elements of $\dual{\Lip_0(M)}$ for some possibly unbounded overspace $M$ of $N$. Similarly, functionals that avoid $0$ strongly can be identified with functionals over some space $\Lip_0(N)$ where $N$ has an isolated base point. In view of Proposition \ref{pr:equiv_char_support}, it is also easy to see that an element of $\lipfree{M}$ is strongly bounded precisely when its support is bounded. This is not the case in $\bidualfree{M}$, as we shall see with Example \ref{ex:empty_extended_support}.

While the definitions above are given in terms of the auxiliary functions $H_n$ and $G_n$ defined in \eqref{eq:H_n} and \eqref{eq:G_n}, it is possible to express some of these notions equivalently in a much more general way. The following proposition contains some such characterizations; additional ones will be given in Proposition \ref{pr:separation_classes_equivalences_part2}.

\begin{proposition}
\label{pr:separation_classes_equivalences}
Let $\phi\in\bidualfree{M}$.
\begin{enumerate}[label={\upshape{(\alph*)}}]
\item $\phi$ is concentrated at infinity if and only if $\duality{f,\phi}=0$ for every $f\in\Lip_0(M)$ with bounded support.
\item $\phi$ is concentrated at $0$ if and only if it is a derivation at $0$.
\item $\phi$ is strongly bounded if and only if $\phi=\phi\circ\wop{h}$ for some $h\in\Lip(M)$ with bounded support.
\end{enumerate}
\end{proposition}

\begin{proof} 
(a) Suppose that $\phi$ is concentrated at infinity, and pick $f\in\Lip_0(M)$ with bounded support. Choose $n\in\NN$ such that $2^n>\rad(\supp(f))$. Then $f=fH_n$ and hence
$$
\duality{f,\phi} = \duality{fH_n,\phi} = \duality{f,\phi\circ\wop{H_n}} = 0 .
$$
For the converse implication, if $f\in\Lip_0(M)$ and $n\in\NN$ then $fH_n\in\Lip_0(M)$ has bounded support, hence $\duality{f,\phi\circ\wop{H_n}}=\duality{fH_n,\phi}=0$. Therefore $\phi\circ\wop{H_n}=0$.

(b) If $\phi$ is a derivation at $0$, then for every $n\in\NN$ and $f\in\Lip_0(M)$ we have $\duality{f,\phi\circ\wop{G_{-n}}}=\duality{fG_{-n},\phi}=0$ since $fG_{-n}=0$ in a neighborhood of $0$, so $\phi$ is concentrated at $0$. Conversely, suppose that $\phi$ is concentrated at $0$ and let $f\in\Lip_0(M)$ be constant, i.e. $0$, in a neighborhood of $0$. Then $f=fG_{-n}$ for $n$ large enough and
$$
\duality{f,\phi} = \duality{fG_{-n},\phi} = \duality{f,\phi\circ\wop{G_{-n}}} = 0. 
$$
Hence $\phi$ is a derivation at $0$.

(c) One implication is obvious taking $h=H_n$. For the converse, let $n$ be such that $2^n>\rad(\supp(h))$, then $hH_n=h$ and hence
$$
\phi\circ\wop{H_n} = (\phi\circ\wop{h})\circ\wop{H_n} = \phi\circ\wop{hH_n} = \phi\circ\wop{h} = \phi .
$$
\end{proof}

Every functional in $\bidualfree{M}$ can be canonically decomposed as a sum of elements of the classes introduced in Definition \ref{def:separation_classes} as follows:

\begin{theorem}
\label{th:decomp_0_infty}
Every $\phi\in\bidualfree{M}$ can be expressed as
\begin{equation}
\label{eq:decomp_0_infty}
\phi=\phi_0+\kaltonsum{\phi}+\phi_\infty
\end{equation}
where $\phi_0$ is a derivation at $0$, $\phi_\infty$ is concentrated at infinity, and $\kaltonsum{\phi}$ avoids $0$ and infinity. This expression is unique and
\begin{equation}
\label{eq:decomp_0_infty_components}
\begin{split}
\phi_0 &= \lim_{n\rightarrow\infty} \phi\circ\wop{H_{-n}} \\
\kaltonsum{\phi} &= \lim_{n\rightarrow\infty} \phi\circ\wop{\Pi_n} \\
\phi_\infty &= \lim_{n\rightarrow\infty} \phi\circ\wop{G_n}
\end{split} .
\end{equation}
Moreover we have
\begin{equation}
\label{eq:decomp_0_infty_norm}
\norm{\phi}=\norm{\phi_0}+\norm{\kaltonsum{\phi}}+\norm{\phi_\infty} .
\end{equation}
If $\phi$ is positive then so are $\phi_0$, $\kaltonsum{\phi}$ and $\phi_\infty$.
\end{theorem}

\begin{proof}
Let us first prove the existence of the decomposition. Let $\phi_0,\kaltonsum{\phi},\phi_\infty$ be given by \eqref{eq:decomp_0_infty_components}. We have already seen that all of these limits exist, and \eqref{eq:decomp_0_infty} follows from taking limits in the identity
$$
\phi = \phi\circ\wop{H_{-(n+1)}} + \phi\circ\wop{\Pi_n} + \phi\circ\wop{G_n}.
$$
Notice that, for any fixed $n\geq 1$
$$
\phi_0\circ\wop{G_{-n}} = \pare{\lim_{k\rightarrow\infty} \phi\circ\wop{H_{-k}}}\circ\wop{G_{-n}} = \lim_{k\rightarrow\infty} \phi\circ\wop{H_{-k}G_{-n}} = 0
$$
as $H_{-k}G_{-n}=0$ eventually, so $\phi_0$ is concentrated at $0$, hence it is a derivation at $0$ by Proposition \ref{pr:separation_classes_equivalences}(b). By an analogous argument, $\phi_\infty$ is concentrated at infinity and $\kaltonsum{\phi}=\lim_{n\rightarrow\infty} \kaltonsum{\phi}\circ\wop{\Pi_n}$, which means that $\kaltonsum{\phi}$ avoids $0$ and infinity by the comment below Definition \ref{def:separation_classes}. The statement about positivity follows from the fact that all weighting operators used in the construction preserve positivity.

Next we prove that the decomposition is unique. Assume that
$$
\phi=\phi_0+\kaltonsum{\phi}+\phi_\infty=\phi'_0+\kaltonsum{\phi}'+\phi'_\infty
$$
where $\phi_0,\phi'_0$ are concentrated at $0$, $\phi_\infty,\phi'_\infty$ are concentrated at infinity, and $\kaltonsum{\phi},\kaltonsum{\phi}'$ avoid $0$ and infinity. Then $\phi_0+\kaltonsum{\phi}-\phi'_0-\kaltonsum{\phi}'=\phi'_\infty-\phi_\infty$, and we have
\begin{align*}
0 = \lim_{n\rightarrow\infty}(\phi'_\infty-\phi_\infty)\circ\wop{H_n} &= \lim_{n\rightarrow\infty}(\phi_0-\phi'_0+\kaltonsum{\phi}-\kaltonsum{\phi}')\circ\wop{H_n} \\
&= \phi_0-\phi'_0 + \lim_{n\rightarrow\infty}(\kaltonsum{\phi}-\kaltonsum{\phi}')\circ\wop{H_n} \\
&= \phi_0-\phi'_0+\kaltonsum{\phi}-\kaltonsum{\phi}'
\end{align*}
so that $\phi'_0-\phi_0=\kaltonsum{\phi}-\kaltonsum{\phi}'$. It follows similarly that
$$
0 = \lim_{n\rightarrow\infty}(\phi'_0-\phi_0)\circ\wop{G_{-n}} = \lim_{n\rightarrow\infty}(\kaltonsum{\phi}-\kaltonsum{\phi}')\circ\wop{G_{-n}} = \kaltonsum{\phi}-\kaltonsum{\phi}' .
$$
Hence $\kaltonsum{\phi}=\kaltonsum{\phi}'$, and uniqueness follows.

Finally, we will prove \eqref{eq:decomp_0_infty_norm}. Fix $\varepsilon>0$, and choose functions $f_0,f_s,f_\infty\in\sphere{\Lip_0(M)}$ such that $\duality{f_0,\phi_0}>\norm{\phi_0}-\varepsilon$, $\duality{f_s,\kaltonsum{\phi}}>\norm{\kaltonsum{\phi}}-\varepsilon$ and $\duality{f_\infty,\phi_\infty}>\norm{\phi_\infty}-\varepsilon$. By the convergence of \eqref{eq:kalton_absconv}, we may find $n\in\NN$ such that
$$
\sum_{\substack{k\in\ZZ \\ \abs{k}>n}}\norm{\phi\circ\wop{\Lambda_k}}<\varepsilon .
$$
This implies that $\norm{\kaltonsum{\phi}-\phi\circ\wop{\Pi_n}}<\varepsilon$ and, more generally
\begin{equation}
\label{eq:decomp_proof_bound}
\norm{\phi\circ\wop{1-H_{-m}-\Pi_n-G_m}} = \norm{\sum_{k=-m+1}^{-(n+1)}\phi\circ\wop{\Lambda_k} + \sum_{k=n+1}^m\phi\circ\wop{\Lambda_k}} < \varepsilon
\end{equation}
for any $m>n+1$. Fix $m>n+2$, to be determined later. Since $\phi\circ\wop{H_{-m}}\rightarrow\phi_0$ and $\phi\circ\wop{G_m}\rightarrow\phi_\infty$ as $m\rightarrow\infty$, $m$ can be chosen to be large enough so that $\norm{\phi_0-\phi\circ\wop{H_{-m}}}<\varepsilon$ and $\norm{\phi_\infty-\phi\circ\wop{G_m}}<\varepsilon$.

Let $M'=M_0\cup M_s\cup M_\infty$ where
\begin{align*}
M_0 &= B(0,2^{-m+1}) \\
M_s &= B(0,2^{n+1})\setminus B(0,2^{-(n+1)}) \\
M_\infty &= M\setminus B(0,2^{m})
\end{align*}
and define $g:M'\rightarrow\RR$ by $g=f_0$ on $M_0$, $g=f_s$ on $M_s$ and $g=f_\infty$ on $M_\infty$. Let us estimate $\lipnorm{g}$. Clearly $\abs{g(x)-g(y)}\leq d(x,y)$ when $x,y$ belong to the same one of the disjoint sets $M_0$, $M_s$, $M_\infty$. Otherwise suppose $d(x,0)>d(y,0)$, then we actually have $d(x,0)\geq 2^{m-n-2}d(y,0)$ and hence
$$
\frac{\abs{g(x)-g(y)}}{d(x,y)} \leq \frac{\abs{g(x)}+\abs{g(y)}}{d(x,y)} \leq \frac{d(x,0)+d(y,0)}{d(x,0)-d(y,0)} \leq \frac{2^{m-n-2}+1}{2^{m-n-2}-1} .
$$
Thus, by choosing $m$ large enough we can guarantee that $\lipnorm{g}<1+\varepsilon$. Now extend $g$ to $M$ without increasing its norm. Then we have
\begin{multline}
\label{eq:decomp_proof_split}
\duality{g,\phi} = \duality{g,\phi\circ\wop{H_{-m}}} + \duality{g,\phi\circ\wop{\Pi_n}} \\ + \duality{g,\phi\circ\wop{G_m}} + \duality{g,\phi\circ\wop{1-H_{-m}-\Pi_n-G_m}} .
\end{multline}
Notice that
$$
\duality{g,\phi\circ\wop{H_{-m}}} = \duality{f_0,\phi\circ\wop{H_{-m}}} = \duality{f_0,\phi_0} - \duality{f_0,\phi_0-\phi\circ\wop{H_{-m}}} > \norm{\phi_0}-2\varepsilon
$$
and similarly the second and third terms in \eqref{eq:decomp_proof_split} are bounded below by $\norm{\kaltonsum{\phi}}-2\varepsilon$ and $\norm{\phi_\infty}-2\varepsilon$, respectively. Substituting into \eqref{eq:decomp_proof_split} and using \eqref{eq:decomp_proof_bound} we obtain
$$
\duality{g,\phi} > \norm{\phi_0}+\norm{\kaltonsum{\phi}}+\norm{\phi_\infty} - 6\varepsilon - \lipnorm{g}\varepsilon
$$
and it follows that
$$
\norm{\phi} \geq \frac{\duality{g,\phi}}{\lipnorm{g}} > \frac{\norm{\phi_0} + \norm{\kaltonsum{\phi}} + \norm{\phi_\infty} - 6\varepsilon}{1+\varepsilon} - \varepsilon .
$$
Letting $\varepsilon\rightarrow 0$ yields $\norm{\phi} \geq\norm{\phi_0} + \norm{\kaltonsum{\phi}} + \norm{\phi_\infty}$, and the converse inequality is obvious.
\end{proof}

By adding $\phi_A=\phi_0+\kaltonsum{\phi}$ (or following the reasoning in the proof of Theorem \ref{th:decomp_0_infty}), we obtain a similar decomposition of $\phi$ into a part that is concentrated at infinity and a part that avoids infinity, and they do not depend on the choice of base point:

\begin{corollary}
\label{cor:decomposition_A_infinity}
Every $\phi\in\bidualfree{M}$ can be expressed uniquely as $\phi=\phi_A+\phi_\infty$ where $\phi_A$ avoids infinity and $\phi_\infty$ is concentrated at infinity. These terms are given by $\phi_A = \lim_n \phi\circ\wop{H_n}$ and $\phi_\infty = \lim_n \phi\circ\wop{G_n}$. Moreover we have $\norm{\phi}=\norm{\phi_A}+\norm{\phi_\infty}$.
\end{corollary}

\noindent This shows, in particular, that the collection of all functionals of $\bidualfree{M}$ that avoid infinity (resp. avoid $0$ and infinity) forms a 1-complemented subspace of $\bidualfree{M}$.

\subsection{Extended supports}

Let us now deal with the generalized concept of support for elements of $\bidualfree{M}$. Before giving the definition, we will briefly discuss the reasoning behind the changes with respect to the existing notion for $\lipfree{M}$.

First of all, for any such notion, it is a reasonable expectation that the support of an evaluation functional is precisely the point where the evaluation takes place. That is the case for the support of elements $\delta(x)\in\lipfree{M}$, for $x\in M$. However, we have seen that the evaluation on any $\zeta\in\rcomp{M}\setminus M$ is also an element of $\bidualfree{M}$. In order to accommodate for that, the support should now be a subset of some compactification of $M$. And since any compactification larger than $\ucomp{M}$ could contain elements that are not separated by Lipschitz functions and therefore lead to inconsistencies when defining the support of their associated evaluation functionals, the appropriate choice is precisely $\ucomp{M}$.

Second, it is not possible to obtain a support that is sharp in the sense of Proposition \ref{pr:equiv_char_support}, meaning that the action of $\phi\in\bidualfree{M}$ on a function $f\in\Lip_0(M)$ would only depend on the values it (or its extension $\ucomp{f}$) takes on the support of $\phi$. The best example to illustrate this issue are derivations. For instance, let $\phi$ be the derivation at $0$ that we have constructed just after Definition \ref{def:derivation}. It does not make sense for the support of $\phi$ to contain any point other than $0$, since the value of $\duality{f,\phi}$ is independent of the behaviour of $f$ away from $0$, so the support should be just $\set{0}$. However, note that $\duality{f,\phi}$ does not depend just on the value of $f$ at $0$, it depends on its behaviour in a neighborhood of $0$ (any neighborhood is valid). 

This discussion suggests that the support of a functional $\phi\in\bidualfree{M}$ should be defined as a subset $S$ of $\ucomp{M}$ that satisfies the following property, weaker than Proposition \ref{pr:equiv_char_support}: \textit{$\duality{f,\phi}=\duality{g,\phi}$ whenever $\ucomp{f}=\ucomp{g}$ on some open set containing $S$} (cf. Theorems \ref{th:extended_support_bounded} and \ref{th:extended_support_minimal}). This motivates the next definition:

\begin{definition}
\label{def:extended_support}
Let $\phi\in\bidualfree{M}$. The set
\begin{align*}
\esupp{\phi} &= \bigcap\set{K\subset\ucomp{M}:\text{$K$ is compact and $\phi\in\ideal{K}^\perp$}} \\
&= \bigcap\set{K\subset\ucomp{M}:\text{$K$ is compact and $\duality{f,\phi}=0$ whenever $\ucomp{f}\restrict_K=0$}}
\end{align*}
will be called the \textit{extended support} of $\phi$.
\end{definition}

\noindent Here, we are extending the definition of $\ideal{K}$ from \eqref{eq:ideal_def} to subsets of $\ucomp{M}$ in the obvious way:
$$
\ideal{K} = \set{f\in\Lip_0(M):\ucomp{f}(\zeta)=0\text{ for all }\zeta\in K} .
$$
It is straightforward to check that Definition \ref{def:extended_support} is equivalent to
\begin{equation}
\label{eq:ext_support_on_M}
\esupp{\phi} = \bigcap\set{\ucl{A}:\text{$A\subset M$ and $\duality{f,\phi}=0$ whenever $f\restrict_A=0$}} .
\end{equation}
Note that the family of sets in either definition of $\esupp{\phi}$ is non-empty as it always contains $\ucomp{M}$ or $M$, respectively.

Before we move on, let us issue a word of warning: it is possible for the extended support of a non-zero functional $\phi\in\bidualfree{M}$ to be empty if $\phi$ is concentrated at infinity. This phenomenon arises when $\duality{f,\phi}\neq 0$ for \textit{unbounded} functions $f\in\Lip_0(M)$ that admit no continuous real-valued extension to $\ucomp{M}$. See Example \ref{ex:empty_extended_support} below.

Let us now establish some basic properties of extended supports. We start by checking that this definition really is a generalization of supports in $\lipfree{M}$ and admits a characterization similar to the one given by Proposition \ref{pr:equiv_points_support}.

\begin{proposition}
\label{pr:equiv_points_support_ext}
Let $\phi\in\bidualfree{M}$ and $\zeta\in\ucomp{M}$. Then $\zeta\in\esupp{\phi}$ if and only if for every neighborhood $U\subset\ucomp{M}$ of $\zeta$ there exists a function $f\in\Lip_0(M)$ whose support is contained in $U\cap M$ and such that $\duality{f,\phi}\neq 0$.
\end{proposition}

\begin{proof}
Suppose that $\zeta$ satisfies the condition in the statement, and let $K\subset\ucomp{M}$ be closed and such that $\phi\in\ideal{K}^\perp$. If $\zeta\notin K$ then let $U$ be a neighborhood of $\zeta$ such that $\ucl{U}\cap K=\varnothing$, then by assumption there is $f\in\Lip_0(M)$ supported on $U\cap M$ with $\duality{f,\phi}\neq 0$. But $\ucomp{f}$ vanishes on $K$ and hence $\duality{f,\phi}=0$, a contradiction. So $\zeta\in K$, and it follows that $\zeta\in\esupp{\phi}$.

Now suppose that $\zeta$ does not satisfy the condition. This means that there is an open neighborhood $U$ of $\zeta$ such that $\duality{f,\phi}=0$ for every $f\in\Lip_0(M)$ whose support is contained in $U\cap M$. Let $V$ be an open set with $\zeta\in V\subset\ucl{V}\subset U$ and $K=\ucomp{M}\setminus V$. Then $f\in\ideal{K}$ implies $\supp(f)\subset\cl{V\cap M}\subset U\cap M$ and therefore $\duality{f,\phi}=0$. Thus $\phi\in\ideal{K}^\perp$, hence $\esupp{\phi}\subset K$ and $\zeta\notin\esupp{\phi}$.
\end{proof}

This characterization yields several basic properties of extended supports almost immediately. First of all we check that, for elements of $\lipfree{M}$, supports and extended supports agree as much as possible, considering that the latter must be closed in $\ucomp{M}$.

\begin{corollary}
\label{cr:extended_support_fm}
If $m\in\lipfree{M}$ then $\supp(m)=\esupp{m}\cap M$ and $\esupp{m}=\ucl{\supp(m)}$.
\end{corollary}

\begin{proof}
The first equality follows immediately from comparing Proposition \ref{pr:equiv_points_support_ext} and Proposition \ref{pr:equiv_points_support}. This also yields $\ucl{\supp(m)}\subset\esupp{m}$, whereas the reverse inclusion follows from \eqref{eq:ext_support_on_M} and Proposition \ref{pr:equiv_char_support}.
\end{proof}

Recall that for any $m\in\lipfree{M}$ we know that $\supp(m)$ is a separable subset of $M$. In view of Corollary \ref{cr:extended_support_fm} this means that $\esupp{m}\cap M$ is a separable subset of $M$. We now show that this is true for any functional in $\dual{\Lip_0(M)}$.

\begin{corollary}
\label{cor:extended_support_separable}
For any $\phi\in\bidualfree{M}$, the subset $\esupp{\phi}\cap M$ of $M$ is separable.
\end{corollary}

\begin{proof}
Take an arbitrary family $\set{B(p_i,2r_i):i\in I}$ of open balls in $M$ such that $p_i\in\esupp{\phi}\cap M$, $r_i>0$, and the sets $B(p_i,2r_i)\cap\esupp{\phi}$ are pairwise disjoint. Then the balls $B(p_i,r_i)$ are disjoint in $M$. Indeed, suppose otherwise that there exists $x\in B(p_i,r_i)\cap B(p_j,r_j)$ for some $i\neq j\in I$. Assume $r_i\leq r_j$. Then
$$
d(p_i,p_j) \leq d(p_i,x) + d(x,p_j) < r_i + r_j \leq 2r_j
$$
therefore $p_i\in B(p_i,2r_i)\cap B(p_j,2r_j)$, contradicting disjointness.

For each $i\in I$ consider the function $\rho_i\in\Lip(M)$ given by $\rho_i:x\mapsto d(x,p_i)$, then $U_i=(\ucomp{\rho_i})^{-1}(-\infty,r_i)$ is an open neighborhood of $p_i$ in $\ucomp{M}$. Since $p_i\in\esupp{\phi}$, Proposition \ref{pr:equiv_points_support_ext} yields a function $f_i\in\ball{\Lip_0(M)}$ such that $\duality{f_i,\phi}\neq 0$ and $\supp(f_i)\subset U_i\cap M\subset B(p_i,r_i)$. Now apply \cite[Lemma 1.5]{CCGMR_2019}, which states that any sequence in $\ball{\Lip_0(M)}$ with pairwise disjoint supports is weakly null, to conclude that for any $k\in\NN$ there may only exist a finite amount of indices $i\in I$ such that $\abs{\duality{f_i,\phi}}\geq\frac{1}{k}$. Hence $I$ is countable.
\end{proof}

Another basic fact of extended supports is that they are compatible with finite sums:

\begin{corollary}
\label{cr:support_of_sum}
If $\phi,\psi\in\bidualfree{M}$, then $\esupp{\phi+\psi}\subset\esupp{\phi}\cup\esupp{\psi}$. If moreover $\esupp{\phi}\cap\esupp{\psi}=\varnothing$, then we get an equality.
\end{corollary}

\begin{proof}
The first part of the statement is straightforward. Assume therefore that  $\esupp{\phi}$ and $\esupp{\psi}$ are disjoint and let $\zeta\in\esupp{\phi}$. By Proposition \ref{pr:equiv_points_support_ext} there exists a neighborhood $V\subset\ucomp{M}$ of $\zeta$ such that for every function $f\in\Lip_0(M)$ supported in $V\cap M$ we have $\duality{f,\psi}=0$. If now $U\subset\ucomp{M}$ is any neighborhood of $\zeta$, by Proposition \ref{pr:equiv_points_support_ext} again, we may find a function $f\in\Lip_0(M)$ with the support contained in $U\cap V\cap M$ which satisfies $\duality{f,\phi}\neq 0$. Hence also $\duality{f,\phi+\psi}=\duality{f,\phi}\neq 0$. We conclude that $\zeta\in\esupp{\phi+\psi}$. 
\end{proof}

\noindent Compatibility of the extended support with infinite sums, or more generally with limits, requires extra hypotheses and so its discussion will be withheld until Lemma \ref{lm:extended_support_limit}.

As a final consequence of Proposition \ref{pr:equiv_points_support_ext}, we extend some facts about positive elements of $\lipfree{M}$ from Propositions 2.8 and 2.9 of \cite{APPP_2020} to the bidual.

\begin{corollary}
\label{cr:positive_facts_bidual}
Let $\phi,\psi$ be positive elements of $\bidualfree{M}$.
\begin{enumerate}[label={\upshape{(\alph*)}}]
\item If $f\in\Lip_0(M)$, $f\geq 0$ and $\duality{f,\phi}=0$, then $\ucomp{f}=0$ on $\esupp{\phi}\cap\rcomp{M}$.
\item If $f\in\ball{\Lip_0(M)}$ and $\duality{f,\phi}=\norm{\phi}$, then $\ucomp{f}=\ucomp{\rho}$ on $\esupp{\phi}\cap\rcomp{M}$.
\item If $\phi\leq\psi$ then $\esupp{\phi}\subset\esupp{\psi}$.
\end{enumerate}
\end{corollary}

\begin{proof}
(a) Suppose that $\ucomp{f}(\zeta)>0$ for some $\zeta\in\esupp{\phi}\cap\rcomp{M}$. Then there are $c>0$ and an open neighborhood $U$ of $\zeta$ such that $\ucomp{f}(\xi)\geq c$ for $\xi\in U$. Moreover, since $\zeta\in\rcomp{M}$, we may choose $U$ such that $U\cap M$ is bounded. By Proposition \ref{pr:equiv_points_support_ext} there exists $h\in\Lip_0(M)$ such that $\supp(h)\subset U\cap M$ and $\duality{h,\phi}\neq 0$. We may assume that $h\geq 0$ and $\duality{h,\phi}>0$ by replacing $h$ with $h^+$ or $h^-$. Since $h$ is bounded, we may also assume that $h\leq c$ by multiplying by a positive constant. Then $f-h\geq 0$ but $\duality{f-h,\phi}<0$, a contradiction.

(b) Since $\norm{\phi}=\duality{\rho,\phi}$, this is immediate from applying (a) to the function $\rho-f$.

(c) Fix $\zeta\in\esupp{\phi}$ and let $U$ be a neighborhood of $\zeta$. As in part (a), there is $h\in\Lip_0(M)$ such that $\supp(h)\subset U\cap M$, $h\geq 0$ and $\duality{h,\phi}>0$. Thus $\duality{h,\psi}\geq\duality{h,\phi}>0$, and Proposition \ref{pr:equiv_points_support_ext} implies that $\zeta\in\esupp{\psi}$.
\end{proof}

A generalized version of inclusion \eqref{eq:support_Th} also holds for supports of weighted functionals:

\begin{proposition}
\label{pr:support_Th}
Let $h\in\Lip(M)$ be such that $\wop{h}$ is a bounded linear operator on $\Lip_0(M)$. Then for every $\phi\in\bidualfree{M}$ we have $\esupp{\phi\circ\wop{h}}\subset\esupp{\phi}\cap\ucl{\supp(h)}$.
\end{proposition}

\begin{proof}
Let $f\in\Lip_0(M)$ and suppose $f=0$ on $\supp(h)$, then $fh=0$ and so $\duality{f,\phi\circ\wop{h}}=\duality{fh,\phi}=0$. Thus $\esupp{\phi\circ\wop{h}}\subset\ucl{\supp(h)}$ by \eqref{eq:ext_support_on_M}.
Now let $A\subset M$ be such that $\duality{f,\phi}=0$ for any $f\in\Lip_0(M)$ such that $f\restrict_A=0$. Then, for any such $f$ we also have $\duality{f,\phi\circ\wop{h}}=\duality{fh,\phi}=0$ since $(fh)\restrict_A=0$, and therefore $\esupp{\phi\circ\wop{h}}\subset\ucl{A}$ by \eqref{eq:ext_support_on_M}. Taking the intersection over all such $A$ we get $\esupp{\phi\circ\wop{h}}\subset\esupp{\phi}$.
\end{proof}

We now turn to the main result in this section, the localization property stated in the discussion before Definition \ref{def:extended_support}. An example provided there demonstrates a different behaviour of supports in $\lipfree{M}$ and in $\bidualfree{M}$. Since Proposition \ref{pr:equiv_char_support} is equivalent to the intersection property \eqref{eq: intersection_property}, another way to see this difference is that  a hypothetical bidual version of the intersection property, equating $\bigcap_i\mathcal{I}\pare{K_i}^\perp$ to $\mathcal{I}\pare{\bigcap_i K_i}^\perp$ this time, is not true in general. Indeed, if $\set{K_i}$ is the family of balls centered at $0$ then the derivations at $0$ are contained in $\bigcap_i\mathcal{I}\pare{K_i}^\perp$ but not in $\mathcal{I}\pare{\bigcap_i K_i}^\perp$.
This also suggests the accurate formulation of the property, which is that $\duality{f,\phi}$ depends only on the values of $f$ in a neighborhood of $\esupp{\phi}$. This turns out to be true for bounded functions, and hence also for functionals determined by those. Example \ref{ex:empty_extended_support} then shows that it does not hold in general.

\begin{theorem}
\label{th:extended_support_bounded}
Let $\phi\in\bidualfree{M}$. If $U$ is any open subset of $\ucomp{M}$ containing $\esupp{\phi}$, then $\duality{f,\phi}=0$ for any bounded $f\in\Lip_0(M)$ that vanishes on $U\cap M$.
\end{theorem}

\begin{proof}
We will adapt an argument from \cite[Lemma 7.28]{Weaver2}. Let
\begin{equation}
\label{eq:z}
\mathcal{Z}=\set{K\subset\ucomp{M}:\text{$K$ is compact and $\phi\in\ideal{K}^\perp$}}
\end{equation}
so that we have $\bigcap\mathcal{Z}=\esupp{\phi}$, and denote $K=\ucomp{M}\setminus U$. Since $K\cap\esupp{\phi}=\varnothing$, for each $\zeta\in K$ we may choose a compact $K_\zeta\in\mathcal{Z}$ such that $\zeta\notin K_\zeta$. Separate $\zeta$ and $K_\zeta$ by open neighbourhoods $V_\zeta$ and $W_\zeta$, respectively, whose intersections with $M$ are at a positive distance, and construct a function $g_\zeta\in\Lip(M)$ such that $0\leq g\leq 1$, $\ucomp{g_\zeta}=1$ on $V_\zeta$, and $\ucomp{g_\zeta}=0$ on $W_\zeta$.
The sets $\set{V_\zeta:\zeta\in K}$ form an open cover of the compact $K$, so we can extract a finite subcover $V_{\zeta_1}\cup\ldots\cup V_{\zeta_n}$. Now let $g = g_{\zeta_1} + \ldots + g_{\zeta_n}$. Then we have $\ucomp{g}(\zeta)\geq 1$ for all $\zeta\in K$,
therefore $1/g$ is a bounded Lipschitz function on $M\setminus U$, and by McShane's theorem there exists a bounded $h\in\Lip(M)$ such that $h=1/g$ on $M\setminus U$.

Now suppose that $f\in\Lip_0(M)$ is bounded and vanishes on $U\cap M$. Then $f-fgh=0$ on all of $M$: indeed, we have $f=0$ on $U\cap M$ by hypothesis and $gh=1$ on $M\setminus U$ by construction. Thus
$$
\duality{f,\phi} = \duality{fgh,\phi} = \duality{fg_{\zeta_1}h,\phi} + \ldots + \duality{fg_{\zeta_n}h,\phi} .
$$
Notice that each function $fg_{\zeta_i}h$ is a product of bounded Lipschitz functions and vanishes at $0$, hence belongs to $\Lip_0(M)$. For each $i=1,\ldots,n$ we have $fg_{\zeta_i}h\in\ideal{K_{\zeta_i}}$ and so $\duality{fg_{\zeta_i}h,\phi}=0$ since $K_{\zeta_i}\in\mathcal{Z}$. We conclude that $\duality{f,\phi}=0$.
\end{proof}

Thus we obtain a characterization of extended supports via localization property for functionals that avoid infinity.

\begin{theorem}
\label{th:extended_support_minimal}
Suppose that $\phi\in\bidualfree{M}$ avoids infinity. If $U$ is any open subset of $\ucomp{M}$ containing $\esupp{\phi}$, then $\duality{f,\phi}=0$ for any $f\in\Lip_0(M)$ that vanishes on $U\cap M$. Moreover, $\esupp{\phi}$ is the smallest closed subset of $\ucomp{M}$ with that property.
\end{theorem}

\begin{proof}
The first statement is an immediate corollary of the previous theorem. Indeed, if $f\in\Lip_0(M)$ vanishes on $U\cap M$, then so do the bounded functions $fH_n\in\Lip_0(M)$ for every $n\in\NN$, and Theorem \ref{th:extended_support_bounded} implies that $\duality{f,\phi}=\lim_n\duality{fH_n,\phi}=0.$

To prove the minimality, let $\mathcal{Z}$ be again as in \eqref{eq:z} so that $\esupp{\phi}=\bigcap\mathcal{Z}$. Suppose that $S'$ is a closed subset of $\ucomp{M}$ satisfying the hypothesis and that it does not contain $\esupp{\phi}$, i.e. there exists $\zeta\in \esupp{\phi}\setminus S'$. Let $U$ and $V$ be disjoint open sets in $\ucomp{M}$ containing $S'$ and $\zeta$, respectively. Given $f\in\Lip_0(M)$ that vanishes on $\ucomp{M}\setminus V$, it vanishes on $U$ in particular and, by the assumption on $S'$, we have $\duality{f,\phi}=0$. Thus $\ucomp{M}\setminus V\in\mathcal{Z}$ and $\zeta\in \esupp{\phi}\subset\ucomp{M}\setminus V$, a contradiction. Therefore we must have $\esupp{\phi}\subset S'$ as we wanted to prove.
\end{proof}

As an immediate consequence of Theorem \ref{th:extended_support_minimal} we obtain:

\begin{corollary}
\label{cr:extended_support_empty}
If $\phi\in\bidualfree{M}$ avoids infinity, then $\esupp{\phi}=\varnothing$ if and only if $\phi=0$.
\end{corollary}

Theorem \ref{th:extended_support_minimal} and Corollary \ref{cr:extended_support_empty} may fail in the general case as demonstrated by the next example. This also shows that, unlike in $\lipfree{M}$, a bounded extended support does not necessarily imply avoiding infinity, strongly or not.

\begin{example}
\label{ex:empty_extended_support}
Let $M\subset\RR^2$ consist of the points $x_n=(0,n)$ and $y_n=(1,n)$ for $n\geq 0$. Define $m_n=\delta(x_n)/n$ for $n\geq 1$. Then $\norm{m_n}=1$, so we may find a subnet $(x_{n_i})$ of the sequence $(x_n)$ that converges to some $\zeta\in\ucomp{M}$ and such that $(m_{n_i})$ converges weak$^\ast$ to some $\phi\in\bidualfree{M}$. Notice that $\phi\geq 0$ and $\norm{\phi}=\duality{\rho,\phi}=1$.

We claim that $\esupp{\phi}=\varnothing$. Indeed, let $X=\set{x_n:n\geq 0}$ and $Y=\set{y_n:n\geq 0}$, then $\ucomp{M}=\ucl{X}\cup\ucl{Y}$ is a disjoint union by Proposition \ref{pr:uc_definition} and it is clear that $\esupp{\phi}\subset\ucl{X}$ by \eqref{eq:ext_support_on_M}. Now let $U\subset\ucomp{M}$ be a neighborhood of $\ucl{X}$ such that $U\cap\ucl{Y}=\varnothing$ and suppose that $f\in\Lip_0(M)$ is supported on $U\cap M=X$. Then $f(y_n)=0$ and $\abs{f(x_n)}\leq\lipnorm{f}d(x_n,y_n)=\lipnorm{f}$ for every $n$, hence $\duality{f,\phi}=\lim_i f(x_{n_i})/n_i=0$. So $\esupp{\phi}\cap\ucl{X}=\varnothing$ by Proposition \ref{pr:equiv_points_support_ext}, therefore $\esupp{\phi}$ is empty.

Observe also that if $f\in\Lip_0(M)$ is constant on some neighborhood of $\zeta$, then in particular it is bounded on $(x_{n_i})$ and $\duality{f,\phi}=0$ again. Hence $\phi$ is a derivation at $\zeta$.
\end{example}

Note that Example \ref{ex:empty_extended_support} together with Corollary \ref{cr:support_of_sum} allow us to generate examples of functionals with non-empty extended supports failing Theorem \ref{th:extended_support_minimal}. Indeed, consider for instance any functional of the form $\phi+\delta(x_n)$. Its extended support is $\{x_n\}$ and the function $\rho G_n$ vanishes on a neighborhood of $\{x_n\}$, however $\duality{\rho G_n, \phi+\delta(x_n)}=1$.

The following statement, illustrated by Example \ref{ex:empty_extended_support}, also follows easily from Theorem \ref{th:extended_support_bounded}:

\begin{corollary}
Let $\phi\in\bidualfree{M}$ be such that $\esupp{\phi}=\varnothing$. Then $\duality{f,\phi}=0$ for every bounded $f\in\Lip_0(M)$.
\end{corollary}

Let us point out by the following example that the converse to this statement is not true, and that the minimum property of $\esupp{\phi}$ from Theorem \ref{th:extended_support_minimal} does not carry over to Theorem \ref{th:extended_support_bounded}.

\begin{example}
Let $M\subset\RR$ consist of $0$ and the points $x_n=2^n$ for $n\in\NN$. The functionals $\phi_n=\delta(x_n)/2^n$ are uniformly bounded, so there is a subnet $(x_{n_i})$ of $(x_n)$ such that $x_{n_i}$ converges to some $\zeta\in\ucomp{M}$ and $\phi_{n_i}$ converges weak$^\ast$ to some $\phi\in\bidualfree{M}$. It is clear that $\duality{f,\phi}=0$ whenever $f\in\Lip_0(M)$ is bounded. Let us check that $\zeta\in\esupp{\phi}$. To that end, fix a neighborhood $U$ of $\zeta$. By Proposition \ref{pr:equiv_points_support_ext} it will be enough to find $f\in\Lip_0(M)$ such that $\duality{f,\phi}=1$ and $\supp(f)\subset U\cap M$. We claim that this is satisfied by the function defined by $f(0)=0$ and
$$
f(x_n)=
\begin{cases}
2^n &\text{, if } x_n\in U \\
0 &\text{, if } x_n\notin U
\end{cases}
$$
for $n\in\NN$. Indeed, it is easy to check that $\lipnorm{f}\leq 2$. Moreover $\duality{f,\phi_n}=1$ for $x_n\in U$, so considering a subnet of $(x_{n_i})$ contained in $U$ allows us to conclude that $\duality{f,\phi}=1$. It follows that $\esupp{\phi}$ is non-empty, and therefore it is not the smallest closed set $S$ with the property that any bounded $f\in\Lip_0(M)$ vanishing on a neighborhood of $S$ also vanishes on $\phi$ (that minimum set would be $S=\varnothing$).
\end{example}

\subsection{Classification of functionals using supports}

We will show now how extended supports may be used to characterize some of the classes of functionals defined in Definition \ref{def:separation_classes}. The following proposition complements Proposition \ref{pr:separation_classes_equivalences}.

\begin{proposition}
\label{pr:separation_classes_equivalences_part2}
Let $\phi\in\bidualfree{M}$.
\begin{enumerate}[label={\upshape{(\alph*)}}]
\item $\phi$ is concentrated at infinity if and only if $\esupp{\phi}\cap\rcomp{M}=\varnothing$.
\item $\phi$ avoids $0$ strongly if and only if $0\notin\esupp{\phi}$.
\end{enumerate}
\end{proposition}

\begin{proof} 
(a) Suppose that $\phi$ is concentrated at infinity. Let $\zeta\in\rcomp{M}$ and take $n\in\NN$ such that $2^n>\ucomp{\rho}(\zeta)$. Then $U=\set{\xi\in\ucomp{M}:\ucomp{\rho}(\xi)<2^n}$ is a neighborhood of $\zeta$, and for every $f\in\Lip_0(M)$ with $\supp(f)\subset U$ we have $\duality{f,\phi}=\duality{fH_n,\phi}=\duality{f,\phi\circ\wop{H_n}}=0$. By Proposition \ref{pr:equiv_points_support_ext}, $\zeta\notin\esupp{\phi}$.

Now suppose that $\esupp{\phi}\cap\rcomp{M}=\varnothing$ and fix $n\in\NN$. Then by Proposition \ref{pr:support_Th}
$$
\esupp{\phi\circ\wop{H_n}} \subset \esupp{\phi}\cap\ucl{\supp(H_n)} \subset \ucl{B(0,2^{n+1})}\setminus\rcomp{M} = \varnothing .
$$
But $\phi\circ\wop{H_n}$ clearly avoids infinity (e.g. by Proposition \ref{pr:separation_classes_equivalences}(c)), so $\phi\circ\wop{H_n}=0$ by Corollary \ref{cr:extended_support_empty}.

(b) If $\phi$ avoids $0$ strongly then $0$ cannot be contained in $\esupp{\phi}=\esupp{\phi\circ\wop{G_{-n}}}\subset\ucl{\supp(G_{-n})}$ by Proposition \ref{pr:support_Th}. For the converse, assume $0\notin\esupp{\phi}$.
By Proposition \ref{pr:equiv_points_support_ext} there is a neighborhood $U\subset\ucomp{M}$ of $0$ such that $\duality{f,\phi}=0$ for every $f\in\Lip_0(M)$ supported in $U\cap M$. Let $n\in\NN$ be so large that $\supp(H_{-n})\subset U\cap M$. Then for every $f\in\Lip_0(M)$ we have 
$$\duality{f,\phi\circ\wop{G_{-n}}}=\duality{f,\phi}-\duality{f,\phi\circ\wop{H_{-n}}}=\duality{f,\phi}-\duality{fH_{-n},\phi}=\duality{f,\phi}$$
because $\supp(fH_{-n})\subset U\cap M$. Hence $\phi=\phi\circ\wop{G_{-n}}$.
\end{proof}

It is not possible to characterize all the classes in Definition \ref{def:separation_classes} via the extended supports since a functional may have a component concentrated at infinity which cannot be detected in its support (add for instance the functional from Example \ref{ex:empty_extended_support} with empty extended support to any element of $\bidualfree{M}$ and apply Corollary \ref{cr:support_of_sum}). However, under the assumption of avoiding infinity there is a mutual correspondence between the extended support of a functional and its classification. The next two propositions describe that link.

\begin{proposition}
\label{pr:extended_support_str_bounded}
If $\phi\in\bidualfree{M}$ is strongly bounded, then $\esupp{\phi}\subset\rcomp{M}$. If $\phi$ avoids infinity then the converse is also true.
\end{proposition}

\begin{proof}
Let $\phi\in\bidualfree{M}$ be strongly bounded, i.e. $\phi=\phi\circ\wop{H_n}$ for some $n\in\NN$. In view of Proposition \ref{pr:support_Th} then
$$\esupp{\phi}=\esupp{\phi\circ\wop{H_n}}\subset\ucl{\supp(H_n)}\subset\rcomp{M}.$$
For the converse, assume that $\phi$ avoids infinity and that $\esupp{\phi}\subset\rcomp{M}$. Since $\esupp{\phi}$ is compact and $\ucomp{\rho}\restrict_{\rcomp{M}}$ is a real-valued continuous function, it is bounded on $\esupp{\phi}$ and there exists $n\in\NN$ such that $U=\set{\zeta\in\ucomp{M}:\ucomp{\rho}(\zeta)<2^n}$ is an open neighborhood of $\esupp{\phi}$. For every $f\in\Lip_0(M)$ we thus get by Theorem \ref{th:extended_support_minimal} that
$$\duality{f,\phi}=\duality{fH_n,\phi}=\duality{f,\phi\circ\wop{H_n}}$$
because $f=fH_n$ on $U$. Hence $\phi=\phi\circ\wop{H_n}$.
\end{proof}

For the second proposition, we need an independent lemma describing the stability of extended supports with respect to weak$^\ast$ limits:

\begin{lemma}
\label{lm:extended_support_limit}
If $\phi\in\bidualfree{M}$ and there is a sequence $(\phi_n)\subset\bidualfree{M}$ of functionals which avoid infinity and such that $(\phi_n)$ weak$^\ast$ converges to $\phi$, then $\esupp{\phi}\subset\ucl{\bigcup_{n\in\NN}\esupp{\phi_n}}$.
\end{lemma}

\begin{proof}
Let $\zeta\notin\ucl{\bigcup_{n\in\NN}\esupp{\phi_n}}$ and let $U$ and $V$ be disjoint open neighborhoods of $\zeta$ and $\ucl{\bigcup_{n\in\NN}\esupp{\phi_n}}$, respectively. Then for any $f\in\Lip_0(M)$ with $\supp(f)\subset U$ we have by Theorem \ref{th:extended_support_minimal} that $\duality{f,\phi_n}=0$ for every $n\in\NN$. Hence also $\duality{f,\phi}=\lim_{n\to\infty}\duality{f,\phi_n}=0$ which proves that $\zeta\notin\esupp{\phi}$ by Proposition \ref{pr:equiv_points_support_ext}.
\end{proof}

\begin{proposition}
\label{pr:extended_support_avoid_inf}
If $\phi\in\bidualfree{M}$ avoids infinity then $\esupp{\phi}=\ucl{\esupp{\phi}\cap\rcomp{M}}$.
\end{proposition}

\begin{proof}
One inclusion is obvious. For the other one, recall that
$$
\esupp{\phi\circ\wop{H_n}}\subset\esupp{\phi}\cap \ucl{B(0,2^{n+1})}\subset\esupp{\phi}\cap\rcomp{M}
$$
according to Proposition \ref{pr:support_Th}. Applying Lemma \ref{lm:extended_support_limit} to $\phi=\lim_{n\to\infty}\phi\circ\wop{H_n}$ we conclude that $\esupp{\phi}\subset\ucl{\esupp{\phi}\cap\rcomp{M}}$ as desired.
\end{proof}

Compare Proposition \ref{pr:extended_support_avoid_inf} with Corollary \ref{cr:extended_support_fm}: if $\phi\in\bidualfree{M}$ is weak$^\ast$ continuous then its extended support is completely determined by its intersection with $M$, whereas if $\phi$ avoids infinity then it is completely determined by its intersection with $\rcomp{M}$.

Let us end this section by identifying those functionals in $\bidualfree{M}$ which are supported on just one point. In $\lipfree{M}$, these are easily seen to be just the (multiples of) evaluations on said point. In the case of $\bidualfree{M}$ the situation is more complicated, since we have evaluations (on elements of $\rcomp{M}$ this time) but also derivations. The next proposition states that, when we restrict ourselves to the functionals that avoid infinity, these cover all possible cases.

\begin{proposition}
\label{pr:support_derivation}
Suppose that $\phi\in\bidualfree{M}$ avoids infinity and let $\zeta\in\rcomp{M}$. Then the following are equivalent:
\begin{enumerate}[label={\upshape{(\roman*)}}]
\item $\esupp{\phi}\subset\set{\zeta}$,
\item $\phi=a\delta(\zeta)+\psi$, where $a\in\RR$ and $\psi$ is a derivation at $\zeta$.
\end{enumerate}
\end{proposition}

\begin{proof}
First assume (ii), and let $\xi\in\ucomp{M}$, $\xi\neq\zeta$. Choose neighborhoods $U$, $V$ of $\xi$, $\zeta$ such that $d(U\cap M,V\cap M)>0$. Suppose $f\in\Lip_0(M)$ is such that $\supp(f)\subset U\cap M$. Then $\ucomp{f}=0$ on $V$, so $\duality{f,\psi}=0$ by definition and so $\duality{f,\phi}=0$. By Proposition \ref{pr:equiv_points_support_ext} we get $\xi\notin\esupp{\phi}$, and (i) follows. 

Now assume (i). First suppose $\zeta=0$, and let $f\in\Lip_0(M)$ be such that $\ucomp{f}$ is constant in a neighborhood $U$ of $0$. This constant must obviously be $0$. By Theorem \ref{th:extended_support_minimal}, we have $\duality{f,\phi}=0$. This means that $\phi$ is a derivation at $0$.

Suppose now that $\zeta\neq 0$, and choose $h\in\Lip_0(M)$ such that $\ucomp{h}=1$ in a neighborhood of $\zeta$; e.g. take $h=(\frac{1}{b}\rho)\wedge 1$ where $b<\ucomp{\rho}(\zeta)$. Let $\psi=\phi-a\delta(\zeta)$ where $a=\duality{h,\phi}$. Then $\psi$ is a derivation at $\zeta$. Indeed, suppose $f\in\Lip_0(M)$ is such that $\ucomp{f}$ takes a constant value $c$ in a neighborhood $U$ of $\zeta$. Let $g=f-ch$, then $\ucomp{g}=0$ on $U$, and $\duality{g,\phi}=0$ by Theorem \ref{th:extended_support_minimal}, so
$$
\duality{f,\psi}=\duality{f,\phi}-a\ucomp{f}(\zeta)=\duality{f,\phi}-ac=\duality{f-ch,\phi}=\duality{g,\phi}=0 .
$$
\end{proof}

\begin{corollary}
Suppose that $\phi\in\bidualfree{M}$ avoids infinity. Then $\esupp{\phi}$ is finite if and only if $\phi$ is a finite linear combination of point evaluations and derivations on points of $\ucomp{M}$.
\end{corollary}

\begin{proof}
If $\esupp{\phi}$ is finite, then by virtue of Proposition \ref{pr:extended_support_avoid_inf} we have $\esupp{\phi}\subset\rcomp{M}$. Denote $\esupp{\phi}=\set{\zeta_1,\ldots,\zeta_k}$ and find neighborhoods $U_1,\ldots,U_k$ of $\zeta_1,\ldots,\zeta_k$ with pairwise disjoint closures. For each $i$ define $h_i\in\Lip_0(M)$ such that $\ucomp{h_i}=1$ on $U_i$ and $\ucomp{h_i}=0$ on $\bigcup_{j\neq i}U_j$. Since $\max_{i}\ucomp{\rho}(\zeta_i)<\infty$, we may moreover take $h_i$ to have a bounded support. By Theorem \ref{th:extended_support_minimal} then $$\phi=\phi\circ\wop{\sum_{i=1}^kh_i}=\sum_{i=1}^k\phi\circ\wop{h_i}.$$
But Proposition \ref{pr:support_Th} gives that $\esupp{\phi\circ\wop{h_i}}\subset\set{\zeta_i}$, so Proposition \ref{pr:support_derivation} yields
$$
\phi\circ\wop{h_i}=a_i\delta(\zeta_i)+\psi_i
$$
for some $a_i\in\RR$ and $\psi_i$ a derivation at $\zeta_i$, and the forward implication follows. On the other hand, if $\phi$ avoiding infinity is a nontrivial finite linear combination of evaluations and derivations, then these also avoid infinity by the uniqueness of the decomposition in Theorem \ref{th:decomp_0_infty}, and Proposition \ref{pr:support_derivation} implies that their supports are singletons. An appeal to Corollary \ref{cr:support_of_sum} yields that $\esupp{\phi}$ is finite.
\end{proof}

Combining Proposition \ref{pr:support_derivation} with Corollary \ref{cr:extended_support_fm} immediately implies the following, that can also be proved by a direct argument:

\begin{corollary}
If $\phi\in\lipfree{M}$ is a derivation then $\phi=0$.
\end{corollary}

\section{Elements of \texorpdfstring{$\lipfree{M}$}{F(M)} and \texorpdfstring{$\bidualfree{M}$}{Lip0(M)*} with integral representation}

The properties of supports in $\lipfree{M}$ and $\bidualfree{M}$ suggest a possible relationship with spaces of measures defined on the complete metric space $M$.
Indeed, a Borel measure $\mu$ defined on $M$ acts by integration on measurable functions, in particular on Lipschitz functions, and thus we may define a formal mapping $\opint\mu$ from $\Lip_0(M)$ to the reals or extended reals by
$$
\opint\mu(f)=\int_M f\,d\mu
$$
for $f\in\Lip_0(M)$; at least, for those $f$ that are $\mu$-integrable. More generally, if $\mu$ is a Borel measure on the uniform compactification $\ucomp{M}$ then we denote
$$
\opint\mu(f)=\int_{\ucomp{M}}\ucomp{f}\,d\mu
$$
for $f\in\Lip_0(M)$. Note that if the integral $\int_M f\,d\mu$ (resp. $\int_{\ucomp{M}} \ucomp{f}\,d\mu$) exists for every $f$ then this defines a linear functional $\opint\mu$ on $\Lip_0(M)$.

We will pay special attention to Radon measures, and denote by $\meas{M}$ and $\meas{\ucomp{M}}$ the Banach spaces of Radon measures on $M$ and $\ucomp{M}$, respectively. Recall that, since $\ucomp{M}$ is compact, $\meas{\ucomp{M}}$ is the dual of the space $C(\ucomp{M})$ of continuous functions on $\ucomp{M}$, which by Proposition \ref{pr:uc_restriction} can be identified with the space of bounded, uniformly continuous functions on $M$. Let us recall that finite Borel measures on $M$ or $\ucomp{M}$ are Radon as soon as they are inner regular.

Using the operator of extension of measures, it is possible to identify $\meas{M}$ with the subspace of $\meas{\ucomp{M}}$ consisting of those measures that are concentrated on $M$. This follows easily from the fact that $M$ is always a Borel subset of $\ucomp{M}$ (it is in fact a $G_\delta$ subset if we assume $M$ to be complete) and that a subset $K\subset M$ is compact as a topological subspace of $M$ if and only if it is compact as a topological subspace of $\ucomp{M}$. We will make use of this identification without further notice; let us just remark that the notation $\opint\mu$ is consistent with it.

We will say that a functional $\phi$ in $\lipfree{M}$ or $\bidualfree{M}$ is \textit{induced} or \textit{represented} by a measure $\mu$ if $\phi=\opint\mu$. We shall use the notation $\delta_\zeta$ for the Dirac measure on $\zeta\in\ucomp{M}$; note the difference with the functional $\delta(\zeta)$ of $\bidualfree{M}$ (if $\zeta\in\rcomp{M}$), which is obviously induced by $\delta_\zeta$. Let us remark that the measure $\delta_0$ induces the null functional because all $f\in\Lip_0(M)$ vanish at the base point. Thus, any functional that is represented by a measure $\mu$ is also represented by a measure
$$
\hat{\mu}=\mu-\mu(\set{0})\cdot\delta_0=\mu\restrict_{\ucomp{M}\setminus\set{0}}
$$
such that $\hat{\mu}(\set{0})=0$. We shall denote by $\measz{M}$ and $\measz{\ucomp{M}}$ the spaces of such Radon measures.

We will, in fact, consider a slightly more general class than the Radon measures:

\begin{definition}
\label{def:almost radon}
Let us say that a Borel measure $\mu$ on $M$ (resp. $\ucomp{M}$) is \textit{almost Radon} if $\mu(\set{0})=0$ and, for every closed subset $K$ of $M$ (resp. $\ucomp{M}$) such that $0\notin K$, the restriction $\mu\restrict_K$ is Radon.
\end{definition}

\noindent It is straightforward to check that every almost Radon measure $\mu$ on $M$ is inner regular by noticing that
$$
\mu(E) = \lim_{n\rightarrow\infty}\mu(E\setminus B(0,2^{-n})) = \lim_{n\rightarrow\infty}\mu\restrict_{M\setminus B(0,2^{-n})}(E)
$$
for every Borel set $E\subset M$, and similarly for $\ucomp{M}$ (note that we use the condition $\mu(\set{0})=0$ here). Thus, almost Radon measures are either signed Radon measures or $\sigma$-finite positive measures that are Radon except for a ``singularity'' at the base point. We must exclude the case of non-finite signed measures because of indeterminacies of the form $\infty-\infty$ around the base point - instead, such objects will be treated as the difference between two almost Radon measures.

Notice that, for an almost Radon measure $\mu$, we have $\supp(\mu)=\varnothing$ if and only if $\mu=0$. Indeed, for every closed set $K$ not containing $0$ we have $\supp(\mu\restrict_K)\subset\supp(\mu)=\varnothing$ and hence $\mu\restrict_K=0$ as $\mu\restrict_K$ is Radon. Also, any almost Radon measure on $M$ has separable support.

Finally, we observe that every almost Radon measure can be expressed canonically as a limit of Radon measures.

\begin{lemma}
\label{lm:setwise_convergence}
Let $\mu$ be an almost Radon measure on $\ucomp{M}$, and for $n\in\NN$ let $\mu_n$ be defined by $d\mu_n = \ucomp{G_{-n}}\,d\mu$. Then
\begin{enumerate}[label={\upshape{(\alph*)}}]
\item each $\mu_n$ is Radon,
\item $\mu_n$ converges setwise to $\mu$ (i.e. $\lim_n\mu_n(E)=\mu(E)$ for all Borel $E\subset\ucomp{M}$), and
\item if $\opint\mu\in\bidualfree{M}$ then $\opint\mu_n=\opint\mu\circ\wop{G_{-n}}$.
\end{enumerate}
\end{lemma}

\begin{proof}
Notice that $d\mu_n=\ucomp{G_{-n}}\,d(\mu\restrict_{R_n})$ where $R_n=\set{\zeta\in\ucomp{M}:\ucomp{\rho}(\zeta)\geq 2^{-n}}$, therefore $\mu\restrict_{R_n}$ and $\mu_n$ are Radon. This establishes (a), and (c) is obvious. For statement (b), observe that $(\ucomp{G_{-n}})$ converges pointwise and increasingly to the characteristic function of the set $\ucomp{M}\setminus\{0\}$. Therefore, for any Borel set $E\subset\ucomp{M}$ we have by Lebesgue's monotone convergence theorem that
$$
\mu(E) =
\mu(E\setminus\set{0}) = 
\int_{\ucomp{M}\setminus\set{0}}\chi_E\,d\mu =
\lim_{n\rightarrow\infty}\int_{\ucomp{M}\setminus\set{0}}\chi_E\cdot\ucomp{G_{-n}}\,d\mu
 = \lim_{n\rightarrow\infty}\mu_n(E)
$$
where $\chi_E$ is the characteristic function of $E$.
\end{proof}

\subsection{Measures that induce continuous functionals}

Let us start by identifying those measures that induce bounded functionals on $\Lip_0(M)$. The next two propositions show that they are precisely the measures with ``finite first moment''. This generalizes the first part of \cite[Proposition 2.7]{AmPu_2016}; see also \cite[Lemma 3.3]{HiWo_2009}.

\begin{proposition}
\label{pr:elements_induced_by_measure_bidual}
Let $\mu$ be a Borel measure on $\ucomp{M}$ such that $\mu(\set{0})=0$. Then the following are equivalent:
\begin{enumerate}[label={\upshape{(\roman*)}}]
\item $\opint\mu\in\bidualfree{M}$,
\item $\int_{\ucomp{M}}\abs{\ucomp{f}}\,d\abs{\mu}<\infty$ for every $f\in\Lip_0(M)$,
\item $\int_{\ucomp{M}}\ucomp{\rho}\,d\abs{\mu}<\infty$.
\end{enumerate}
If they hold then $\mu$ is concentrated on $\rcomp{M}$, $\abs{\mu}(K)<\infty$ for any closed $K\subset\ucomp{M}$ such that $0\notin K$, and $\mu$ is $\sigma$-finite. If $\mu$ is moreover inner regular then it is almost Radon.
\end{proposition}

\begin{proof}
The implication (i)$\Rightarrow$(ii) is obvious, as $\duality{f,\opint\mu}=\int_{\ucomp{M}}\ucomp{f}\,d\mu$ is finite for every $f\in\Lip_0(M)$. (ii)$\Rightarrow$(iii) is also obvious taking $f=\rho$. Now assume (iii), then for every $f\in\Lip_0(M)$ we have
$$
\abs{\duality{f,\opint\mu}} = \abs{\int_{\ucomp{M}}\ucomp{f}\,d\mu} \leq \int_{\ucomp{M}}\abs{\ucomp{f}}\,d\abs{\mu} \leq \lipnorm{f}\cdot\int_{\ucomp{M}}\ucomp{\rho}\,d\abs{\mu},
$$
hence we get (i) with $\norm{\opint\mu}\leq\int_{\ucomp{M}}\ucomp{\rho}\,d\abs{\mu}$. In particular, $\ucomp{\rho}<\infty$ a.e.($\abs{\mu}$) and therefore $\mu$ is concentrated on $\rcomp{M}$. Moreover, if $K$ is a closed subset of $\ucomp{M}$ such that $0\notin K$, then there is $r>0$ such that $\ucomp{\rho}\geq r$ on $K$ and so
$$
\int_{\ucomp{M}}\ucomp{\rho}\,d\abs{\mu} \geq \int_K\ucomp{\rho}\,d\abs{\mu} \geq r\cdot\abs{\mu}(K)
$$
hence $\abs{\mu}(K)<\infty$. Since $\ucomp{M}\setminus\set{0}=\bigcup_{n\in\NN}K_n$ where $K_n=\set{\zeta\in\ucomp{M}:\ucomp{\rho}(\zeta)\geq\frac{1}{n}}$ and $\mu(\set{0})=0$, we get that $\mu$ is $\sigma$-finite.
The last statement is obvious.
\end{proof}

In the version for measures on $M$, we require additional hypotheses but we also get more information.

\begin{proposition}
\label{pr:elements_induced_by_measure}
Let $\mu$ be a Borel measure on $M$ such that $\mu(\set{0})=0$. Suppose that either $M$ is separable or $\mu$ is inner regular. Then the following are equivalent:
\begin{enumerate}[label={\upshape{(\roman*)}}]
\item $\opint\mu\in\bidualfree{M}$,
\item $\opint\mu\in\lipfree{M}$,
\item $\int_M \abs{f}\,d\abs{\mu}<\infty$ for every $f\in\Lip_0(M)$,
\item $\int_M \rho\,d\abs{\mu}<\infty$,
\item $\displaystyle \opint\mu=\int_M \delta(x)\,d\mu(x)$ as a Bochner integral.
\end{enumerate}
If they hold then $\mu$ is almost Radon.
\end{proposition}

\begin{proof}
By considering $\mu$ as a Borel measure on $\ucomp{M}$ that is concentrated on $M$, the equivalence of (i), (iii) and (iv) follows from Proposition \ref{pr:elements_induced_by_measure_bidual}, as does the last statement if $\mu$ is inner regular. If $M$ is assumed to be separable instead then every closed $K\subset M$ with $0\notin K$ is complete and separable and hence the finiteness of $\mu\restrict_K$ implies that it is Radon, thus $\mu$ is again almost Radon.

The implication (ii)$\Rightarrow$(i) is obvious, as is (v)$\Rightarrow$(ii) since the integrand $\delta(x)$ belongs to $\lipfree{M}$ for all $x\in M$. Finally, assume (iv) and we will prove (v).
Notice that the mapping $\delta:M\rightarrow\lipfree{M}$ is continuous and its range is $\mu$-essentially separable:
indeed $\mu$ is almost Radon so it is concentrated on $\supp(\mu)$ which is separable. Thus $\delta$ is $\mu$-measurable and, since
$$
\int_M\norm{\delta(x)}\,d\mu(x) = \int_M\rho(x)\,d\mu(x)
$$
is finite, the Bochner integral in (v) exists as an element of $\lipfree{M}$ (see e.g. \cite[Theorems II.1.2 and II.2.2]{DiUh}), and its action on each $f\in\Lip_0(M)$ is obviously the same as that of $\opint\mu$. This finishes the proof.
\end{proof}

We do not know whether the implication (i)$\Rightarrow$(ii) holds for non-separable $M$ without regularity assumptions on $\mu$.

\begin{remark}
\label{rm:restricted_functionals}
Using the condition that $\rho$ (or $\ucomp{\rho}$) be $\mu$-integrable, we obtain the following simple consequence: if $\opint\mu\in\bidualfree{M}$ then $\opint(\mu\restrict_E)\in\bidualfree{M}$ for any Borel subset $E$ of $\ucomp{M}$. If moreover $M$ is separable or $\mu$ is inner regular then $\opint(\mu\restrict_E)\in\lipfree{M}$ for any Borel $E\subset M$, in particular $\opint(\mu\restrict_M)\in\lipfree{M}$.
\end{remark}

We also get the following immediate consequence:

\begin{corollary}
\label{cr:bounded_measures}
The following are equivalent:
\begin{enumerate}[label={\upshape{(\roman*)}}]
\item $M$ is bounded,
\item every Radon measure on $M$ induces an element of $\lipfree{M}$,
\item every Radon measure on $\ucomp{M}$ induces an element of $\bidualfree{M}$.
\end{enumerate}
\end{corollary}

\begin{proof}
If $M$ is bounded then so are $\rho$ and $\ucomp{\rho}$, and hence $\ucomp{\rho}$ is $\mu$-integrable for every finite measure $\mu$ on $M$ or $\ucomp{M}$. On the other hand, if $M$ is unbounded then we may find points $x_n\in M$ such that $d(x_n,0)\geq 2^n$, and $\mu=\sum_n 2^{-n}\delta_{x_n}$ is a finite Radon measure such that $\int_M\rho\,d\mu=\infty$.
\end{proof}

\subsection{Uniqueness of representing measures}

A natural question is whether it is possible for different measures to induce the same functional on $\Lip_0(M)$. We can show that the representing measure must be unique (up to its content at $\set{0}$) without requiring finiteness; instead, a milder assumption of regularity is enough. This result will follow from the correspondence between the concepts of support introduced for functionals on $\Lip_0(M)$ and the usual notion of support when these functionals are represented by almost Radon measures:

\begin{proposition}
\label{pr:induced_supp_pos_bidual}
Suppose that $\phi\in\bidualfree{M}$ is induced by an almost Radon measure $\mu$ on $\ucomp{M}$. Then $\esupp{\phi}=\supp(\mu)$, and $\phi$ is positive if and only if $\mu$ is positive.
\end{proposition}

\begin{proof}
For the rest of the proof, we fix a Hahn decomposition $A^+,A^-$ of $\ucomp{M}$ associated to $\mu$. We will also consider the open sets
$$
R_n = \set{\zeta\in\ucomp{M}:\ucomp{\rho}(\zeta)>2^{-n}}
$$
for $n\in\NN$. Note that $R_n\subset R_{n+1}$ and $\bigcup_{n\in\NN}R_n=\ucomp{M}\setminus\set{0}$.

We will start with the first assertion. Let $\zeta\in\ucomp{M}$ and suppose $\zeta\notin\supp(\mu)$. Then there is an open neighborhood $U$ of $\zeta$ such that $\abs{\mu}(U)=0$; by passing to a smaller neighborhood, we may assume that $\abs{\mu}(\ucl{U})=0$. If $f\in\Lip_0(M)$ is supported on $U\cap M$, then $\ucomp{f}$ is supported on $\ucl{U}$ and $\duality{f,\phi}=\int_{\ucl{U}}\ucomp{f}\,d\mu=0$. Thus $\zeta\notin\esupp{\phi}$ by Proposition \ref{pr:equiv_points_support_ext}.

Now suppose $\zeta\in\supp(\mu)$ and let $U$ be a neighborhood of $\zeta$. We will construct a function $f\in\Lip_0(M)$ such that $\supp(f)\subset U\cap M$ and $\duality{f,\phi}>0$ and this will show that $\zeta\in\esupp{\phi}$, again by Proposition \ref{pr:equiv_points_support_ext}. Let $U'$ be a neighborhood of $\zeta$ such that $\ucl{U'}\subset U$, then $\abs{\mu}(U'\setminus\set{0})>0$ because $\mu(\set{0})=0$. Writing $U'\setminus\set{0}=\bigcup_{n\in\NN}(U'\cap R_n)$ as a nested union, we may find an $n\in\NN$ such that $\abs{\mu}(U'')>0$ where $U''=U'\cap R_n$. Since $0\notin\ucl{R_n}$, the measure $\mu\restrict_{\ucl{R_n}}$, hence also $\mu\restrict_{U''}$, is Radon by assumption. Let $B^+ = U''\cap A^+$, $B^- = U''\cap A^-$
and, using the inner regularity of $\mu\restrict_{U''}$, choose compact sets $K^+\subset B^+$, $K^-\subset B^-$ such that $\abs{\mu}(B^+\setminus K^+)$ and $\abs{\mu}(B^-\setminus K^-)$ are less than $\frac{1}{4}\abs{\mu}(U'')$. The compact sets $K^+$, $K^-$ and $\ucomp{M}\setminus U''$ are pairwise disjoint, hence by Proposition \ref{pr:uc_separation} they have pairwise disjoint open neighborhoods $V^+$, $V^-$, $W$ whose intersections with $M$ are at a positive distance from each other. Therefore there exists $f\in\Lip_0(M)$ such that $\abs{f}\leq 1$, $f=1$ on $V^+\cap M$, $f=-1$ on $V^-\cap M$, and $f=0$ on $W\cap M$. Then $\supp(f)\subset \cl{U''\cap M}\subset U\cap M$ and $\ucomp{f}$ vanishes outside of $U''$, and we have
$$
\duality{f,\phi} = \int_{\ucomp{M}}\ucomp{f}\,d\mu = \int_{U''}\ucomp{f}\,d\mu = \abs{\mu}(K^+)+\abs{\mu}(K^-)+\int_{U''\setminus (K^+\cup K^-)}\ucomp{f}\,d\mu
$$
and therefore
\begin{align*}
\abs{\mu}(U'')-\duality{f,\phi} &= \abs{\mu}(B^+) + \abs{\mu}(B^-) - \duality{f,\phi} \\
&\leq \abs{\mu}(B^+\setminus K^+) + \abs{\mu}(B^-\setminus K^-) + \abs{\mu}(U''\setminus (K^+\cup K^-)) \\
&= 2\abs{\mu}(B^+\setminus K^+) + 2\abs{\mu}(B^-\setminus K^-) < \abs{\mu}(U''),
\end{align*}
thus $\duality{f,\phi}>0$ as required.

Now we proceed with the second assertion. If $\mu$ is positive then $\phi$ is obviously positive, too. Assume now that $\mu$ is not positive, that is $\abs{\mu}(A^-)>0$. Then, as above, there exists $n\in\NN$ such that $0<\abs{\mu}(A^-\cap R_n)<\infty$ and that $\mu\restrict_{R_n}$ is Radon. Let $B^+=A^+\cap R_n$, $B^-=A^-\cap R_n$. By inner regularity there are compact sets $K^+\subset B^+$, $K^-\subset B^-$ such that $\abs{\mu}(B^+\setminus K^+)$ and $\abs{\mu}(B^-\setminus K^-)$ are less than $\frac{1}{3}\abs{\mu}(B^-)$. Since $K^+\cup (\ucomp{M}\setminus R_n)$ and $K^-$ are disjoint and compact, by Proposition \ref{pr:uc_separation} they have disjoint open neighborhoods $V^+$ and $V^-$ such that $d(V^+\cap M, V^-\cap M)>0$, and there exists $f\in\Lip_0(M)$ such that $0\leq f\leq 1$, $f=0$ on $V^+\cap M$, and $f=1$ on $V^-\cap M$. Thus $\ucomp{f}$ vanishes outside $R_n$ and
\begin{align*}
\duality{f,\phi} = \int_{\ucomp{M}}\ucomp{f}\,d\mu &= \mu(K^-)+\int_{R_n\setminus (K^+\cup K^-)}\ucomp{f}\,d\mu \\
&\leq -\abs{\mu}(K^-)+\abs{\mu}(B^+\setminus K^+)+\abs{\mu}(B^-\setminus K^-) \\
&= -\abs{\mu}(B^-)+\abs{\mu}(B^+\setminus K^+)+2\abs{\mu}(B^-\setminus K^-) < 0
\end{align*}
and, since $f\geq 0$, this implies that $\phi$ is not positive.
\end{proof}

Let us remark that the same argument is valid for $\phi\in\lipfree{M}$ and measures on $M$, with the following adjustments:
\begin{itemize}
\item Proposition \ref{pr:equiv_points_support} is used instead of Proposition \ref{pr:equiv_points_support_ext},
\item Proposition \ref{pr:uc_separation} is replaced by the fact that two disjoint closed subsets $A$, $B$ of $M$, at least one of them compact, can also be separated by neighborhoods at a positive distance of each other.
\end{itemize}
Thus we also have:

\begin{proposition}
\label{pr:induced_supp_pos}
Suppose that $\phi\in\lipfree{M}$ is induced by an almost Radon measure $\mu$ on $M$. Then $\supp(\phi)=\supp(\mu)$, and $\phi$ is positive if and only if $\mu$ is positive.
\end{proposition}

Now, using the supports, we can prove the uniqueness of the measure inducing a functional on $\Lip_0(M)$ as long as it is inner regular.

\begin{proposition}
\label{pr:measure_uniqueness}
Let $\mu,\lambda$ be almost Radon measures on $\ucomp{M}$ (resp. on $M$) inducing the same functional $\opint\mu=\opint\lambda$ in $\bidualfree{M}$ (resp. in  $\lipfree{M}$). Then $\mu=\lambda$.
\end{proposition}

\begin{proof}
We will prove the statement for functionals in $\bidualfree{M}$ and measures on $\ucomp{M}$, the predual case would follow the same lines with $\lipfree{M}$ and $M$ and corresponding references instead.

For $n\in\NN$, let $\mu_n,\lambda_n$ be the measures given by $d\mu_n=\ucomp{G_{-n}}\,d\mu$ and $d\lambda_n=\ucomp{G_{-n}}\,d\lambda$. By Lemma \ref{lm:setwise_convergence}, they are Radon measures on $\ucomp{M}$ and
$$\opint\mu_n=\opint\mu\circ\wop{G_{-n}}=\opint\lambda\circ\wop{G_{-n}}=\opint\lambda_n.$$
Therefore, from the linearity of the operator $\opint$ on finite measures and from Proposition \ref{pr:induced_supp_pos_bidual} we get that
$$
\supp(\mu_n-\lambda_n) = \esupp{\opint(\mu_n-\lambda_n)}=\esupp{\opint\mu_n-\opint\lambda_n}=\esupp{0} = \varnothing ,
$$
which implies that $\mu_n=\lambda_n$ by Radonness. An application of Lemma \ref{lm:setwise_convergence}(b) now yields $\mu=\lambda$ as claimed.
\end{proof}

As we noted in Remark \ref{rm:restricted_functionals}, if $\mu$ is a Radon measure on $\ucomp{M}$ that induces a functional in $\bidualfree{M}$ and $\mu$ is concentrated on $M$, then the induced functional is actually in $\lipfree{M}$. It is natural to ask whether the opposite implication also holds. That is, if $\opint\mu\in\lipfree{M}$, must $\mu$ be concentrated on $M$?

The answer to this question is negative when $\mu$ is defined on a compactification $X$ that is strictly larger than $\ucomp{M}$. For instance, take two distinct elements $\zeta,\xi\in X$ that cannot be separated by Lipschitz functions on $M$, then $\delta_\zeta-\delta_\xi$ induces the element $0$ of $\lipfree{M}$ (see Example \ref{ex:lip_dont_separate} for a particular case of this). However, the answer is positive when we consider the uniform compactification. This result can be viewed as an extension of \cite[Proposition 2.1.6]{Weaver1}, which only covers Dirac measures. In order to prove it, we will require the following property that is established as a part of the proof of that result (or, more indirectly, in \cite[Lemma 2.6]{GPPR_2018}); we prefer to state it independently due to its usefulness:

\begin{lemma}
\label{lm:separation_from_M}
Let $M$ be a complete metric space. Then, for every $\zeta\in\ucomp{M}\setminus M$ there is $r>0$ such that, for every $p\in M$, every net in $M$ that converges to $\zeta$ is eventually disjoint from $B(p,r)$.
\end{lemma}

\begin{theorem}
\label{th:normal_measure}
Let $\mu$ be an almost Radon measure on $\ucomp{M}$ that induces an element of $\bidualfree{M}$. Then the following are equivalent:
\begin{enumerate}[label={\upshape{(\roman*)}}]
\item $\opint\mu\in\lipfree{M}$,
\item $\opint\mu=\opint(\mu\restrict_M)$,
\item $\mu$ is concentrated on $M$.
\end{enumerate}
\end{theorem}

\begin{proof}
(iii)$\Rightarrow$(ii) is obvious and (ii)$\Rightarrow$(i) is contained in Remark \ref{rm:restricted_functionals}.

It is enough to prove the implication (i)$\Rightarrow$(iii) when $\mu$ is Radon. Indeed, if (iii) fails then $\abs{\mu}(\ucomp{M}\setminus M)>0$ and Lemma \ref{lm:setwise_convergence} implies that $\abs{\mu_n}(\ucomp{M}\setminus M)>0$ for some $n\in\NN$, where $d\mu_n=\ucomp{G_{-n}}\,d\mu$. So $\opint\mu_n=\opint\mu\circ\wop{G_{-n}}\in\lipfree{M}$ but $\mu_n$ is not concentrated on $M$. That is, (i)$\Rightarrow$(iii) fails for the Radon measure $\mu_n$.

Hence, for the rest of the proof we will assume (i) and suppose that $\mu$ is Radon. Let $A^+$, $A^-$ be a Hahn decomposition of $\ucomp{M}$ associated to $\mu$, denote $B^\pm=A^\pm\setminus M$, and fix $\varepsilon>0$. We will prove that $\abs{\mu}(B^\pm)<7\varepsilon$ and this will imply (iii).

For every $p\in M$ consider the 1-Lipschitz function $\rho_p$ given by $\rho_p(x)=d(x,p)$ for $x\in M$, and define
$$
\sigma(\zeta)=\inf\set{\ucomp{\rho_p}(\zeta):p\in M}
$$
for every $\zeta\in\ucomp{M}$, where each $\rho_p$ is extended to a continuous function with values in $[0,+\infty]$. It is clear that $\sigma(\zeta)=0$ if $\zeta\in M$, and Lemma \ref{lm:separation_from_M} asserts that $\sigma(\zeta)>0$ when $\zeta\notin M$. Thus $\ucomp{M}\setminus M=\bigcup_{n\in\NN}\mathcal{K}_n$ where
$$
\mathcal{K}_n=\set{\zeta\in\ucomp{M}:\sigma(\zeta)\geq\tfrac{1}{n}} .
$$
Therefore $\mu(B^\pm)=\lim_{n\rightarrow\infty}\mu(B^\pm\cap\mathcal{K}_n)$.
Choose $n\in\NN$ such that $\abs{\mu}(B^\pm\setminus\mathcal{K}_n)<\varepsilon$. By inner regularity of $\mu$ we may choose compact sets $K^\pm\subset B^\pm\cap\mathcal{K}_n$ such that $\abs{\mu}((B^\pm\cap\mathcal{K}_n)\setminus K^\pm)<\varepsilon$, and hence $\abs{\mu}(B^\pm\setminus K^\pm)<2\varepsilon$. Since $K^+$ and $K^-$ are disjoint, by Proposition \ref{pr:uc_separation} there are disjoint open neighborhoods $U^\pm$ of $K^\pm$ such that $d(U^+\cap M,U^-\cap M)>0$. Denote
$$
r=\frac{1}{2}\min\set{\frac{1}{n},d(U^+\cap M,U^-\cap M)} .
$$

Now let $\mathfrak{F}$ be the family of all finite subsets of $M$ containing $0$. For every $F\in\mathfrak{F}$, let
\begin{equation}
\label{eq:net_f_F}
f_F(x)=1\wedge\frac{1}{r}d(x,F\cup(U^-\cap M))
\end{equation}
for $x\in M$. Then $f_F\in\Lip_0(M)$ is such that $0\leq f_F\leq 1$, $\lipnorm{f_F}\leq\frac{1}{r}$, and $f_F=0$ on $U^-\cap M$, hence $\ucomp{f_F}=0$ on $K^-$.
We claim that $\ucomp{f_F}=1$ on $K^+$. Indeed, let $\zeta\in K^+$ and choose a net $(x_i)$ in $M$ that converges to $\zeta$. We may assume that $x_i\in U^+$ for all $i$ by passing to a subnet, so that $d(x_i,U^-\cap M)>r$. Since $\sigma(\zeta)\geq\frac{1}{n}>r$ (because $\zeta\in\mathcal{K}_n$) and $F$ is finite, we may pass to a further subnet such that $d(x_i,p)=\rho_p(x_i)>r$ for every $p\in F$. Therefore $d(x_i,F\cup(U^-\cap M))>r$ and $f_F(x_i)=1$ for all $i$, proving the claim.

Consider the net $(f_F)_{F\in\mathfrak{F}}$ in $\Lip_0(M)$, where $\mathfrak{F}$ is directed by inclusion. It is a norm-bounded net and it converges pointwise (even monotonically) to $0$, as for any fixed $x\in M$ we have $f_F(x)=0$ whenever $F\supset\set{x}$. Hence $f_F\wsconv 0$. Since we assume $\opint\mu\in\lipfree{M}$, we have $\duality{f_F,\opint\mu}\rightarrow 0$. But $\opint(\mu\restrict_M)\in\lipfree{M}$ by Remark \ref{rm:restricted_functionals}, so $\duality{f_F,\opint(\mu\restrict_M)}\rightarrow 0$ as well. Choose $F\in\mathfrak{F}$ such that $\abs{\duality{f_F,\opint\mu-\opint(\mu\restrict_M)}}<\varepsilon$. Then
\begin{align*}
\duality{f_F,\opint\mu-\opint(\mu\restrict_M)} &= \int_{\ucomp{M}\setminus M}\ucomp{f_F}\,d\mu \\
&= \mu(K^+)+\int_{B^+\setminus K^+}\ucomp{f_F}\,d\mu+\int_{B^-\setminus K^-}\ucomp{f_F}\,d\mu
\end{align*}
so that
$$
\mu(K^+) <\varepsilon + \abs{\mu}(B^+\setminus K^+) + \abs{\mu}(B^-\setminus K^-) < 5\varepsilon
$$
and therefore $\mu(B^+)=\mu(B^+\setminus K^+)+\mu(K^+)<7\varepsilon$. A similar construction replacing $U^-$ with $U^+$ in \eqref{eq:net_f_F} shows that $\abs{\mu(B^-)}<7\varepsilon$. This finishes the proof.
\end{proof}

Thanks to Theorem \ref{th:normal_measure}, some results about elements of $\lipfree{M}$ induced by Radon measures in the subsequent sections will follow relatively easily from their counterparts for $\bidualfree{M}$. However, we prefer to state them separately as the required hypotheses might differ.

\subsection{Functionals that admit an integral representation}

To end this section, we now proceed in the opposite direction to Propositions \ref{pr:elements_induced_by_measure_bidual} and \ref{pr:elements_induced_by_measure} and attempt to identify which elements of $\lipfree{M}$ and $\bidualfree{M}$ can be represented by Radon measures. Let us begin with a simple observation:

\begin{lemma}
\label{lm:measure_avoids_infty}
Suppose that $\phi\in\bidualfree{M}$ is induced by a Borel measure on $\ucomp{M}$. Then $\phi$ avoids $0$ and infinity.
\end{lemma}

\begin{proof}
Let $\phi=\opint\mu$ where $\mu$ is a Borel measure on $\ucomp{M}$. Then $\opint\mu^+,\opint\mu^-\in\bidualfree{M}$ by Remark \ref{rm:restricted_functionals}. To show that $\phi$ avoids $0$ and infinity, it is enough to prove that $\phi\circ\wop{\Pi_n}\wsconv\phi$, i.e. that $\duality{f\Pi_n,\phi}\rightarrow\duality{f,\phi}$ for every $f\in\Lip_0(M)$. Now notice that
\begin{align*}
\lim_{n\rightarrow\infty}\duality{f^+\Pi_n,\opint\mu^+} &= \lim_{n\rightarrow\infty}\int_{\rcomp{M}}\ucomp{(f^+\Pi_n)}\,d\mu^+ \\
&= \lim_{n\rightarrow\infty}\int_{\rcomp{M}}\ucomp{(f^+)}\ucomp{\Pi_n}\,d\mu^+ = \int_{\rcomp{M}}\ucomp{(f^+)}\,d\mu^+ = \duality{f^+,\opint\mu^+}
\end{align*}
by Lebesgue's monotone convergence theorem because $(\ucomp{\Pi_n})$ converges pointwise and monotonically to the characteristic function of $\rcomp{M}\setminus\set{0}$; note that we have used the fact that $\mu$ is concentrated on $\rcomp{M}$ by Proposition \ref{pr:elements_induced_by_measure_bidual}. By replacing $f^+$ with $f^-$ and/or $\mu^+$ with $\mu^-$ we get the desired conclusion.
\end{proof}

Lemma \ref{lm:measure_avoids_infty} is evidently not a characterization of functionals induced by measures (consider e.g. derivations). The next two theorems will provide such characterizations, although some of their implications only hold under assumptions of boundedness or positivity.
Let us first illustrate the basic idea on a simple example of a compact metric space $M$. Denote $J$ the canonical bounded linear injection of $\Lip_0(M)$ into $C(M)$. Then the fact that a functional $\phi\in\bidualfree{M}$ is induced by a Radon measure on $M$ means just that $\phi\in J^*(\dual{C(M)})$. In particular, $\delta(x)=J^*(\delta_x)$ for any $x\in M$, so the elements of $\lspan \delta(M)$ correspond to finitely supported measures in $\meas{M}$ and we can consider their measure norm $\norm{\cdot}_1$. If $\phi=J^*(\mu)$ for some $\mu\in {\meas{M}}$, then by approximating $\mu$ weak$^*$ by a bounded net of finitely supported measures according to Krein-Milman theorem, and from the weak$^*$ continuity of $J^*$, we obtain that $\phi$ is the weak$^*$ limit of a $\norm{\cdot}_1$-bounded net of finitely supported elements of $\lipfree{M}$. Conversely, if $\phi$ is the weak$^*$ limit of a $\norm{\cdot}_1$-bounded net in $\lspan \delta(M)$, then by passing to a subnet of the corresponding measures that converges weak$^*$ to some $\mu$ in $\meas{M}$, we infer that $\phi=J^*(\mu)$. 

We now extend this observation to the general non-compact setting. In order to do so, let us formalize the notation: if $m\in\lspan\,\delta(M)$ is a finitely supported element of $\lipfree{M}$, say $m=\sum_{k=1}^n a_k\delta(x_k)$ for distinct $x_k\in M\setminus\set{0}$, let $\norm{m}_1=\sum_{k=1}^n\abs{a_k}$. It is clear that this value is uniquely defined, and that $m=\opint\mu$ where $\mu=\sum_{k=1}^n a_k\delta_{x_k}\in\measz{M}$ satisfies $\norm{\mu}=\norm{m}_1$.

\begin{theorem}
\label{th:induced_elements_bidual}
Let $\phi\in\bidualfree{M}$. If $\phi$ is induced by a Radon measure on $\ucomp{M}$, then it is the weak$^\ast$ limit of a net $(m_i)$ of elements of $\lspan\,\delta(M)$ such that $(\norm{m_i}_1)$ is bounded. If $\phi$ is strongly bounded or $\phi$ is positive and avoids infinity, then the converse also holds. Moreover, $\phi$ is positive if and only if $(m_i)$ can be chosen to be positive.
\end{theorem}

\begin{proof}
Suppose that $\phi=\opint\mu\in\bidualfree{M}$ for some $\mu\in\meas{\ucomp{M}}$, and assume $\norm{\mu}\leq 1$ without loss of generality. We claim that $\phi\in\wscl{\conv}(\pm\delta(M))$, which will prove the forward implication with the bound $\norm{m_i}_1\leq 1$. Indeed, by the Krein-Milman theorem we have
$$
\mu\in\ball{\mathcal{M}(\ucomp{M})}=\wscl{\conv}\ext\ball{\mathcal{M}(\ucomp{M})}=\wscl{\conv}\set{\pm\delta_\zeta:\zeta\in\ucomp{M}} .
$$
Now, if $(x_i)$ is a net in $M$ that converges to $\zeta\in\ucomp{M}$ then clearly $\delta_{x_i}\wsconv\delta_\zeta$ in $\meas{\ucomp{M}}$.
We conclude that $\mu\in\wscl{A}$ where $A=\conv\set{\pm\delta_x:x\in M}$, hence there is a net $(\mu_i)$ in $A$ that converges weak$^\ast$ to $\mu$. This implies in particular that $\duality{f,\opint\mu_i}\rightarrow\duality{f,\opint\mu}=\duality{f,\phi}$ for any bounded $f\in\Lip_0(M)$. It is clear that $m_i=\opint\mu_i\in\conv(\pm\delta(M))$ and that $\norm{m_i}_1=\norm{\mu_i}\leq 1$. 
Now notice that, for any $f\in\Lip_0(M)$ and $n\in\NN$, $fH_n$ is a bounded Lipschitz function and therefore
$$
\lim_i\duality{m_i\circ\wop{H_n},f} = \lim_i\duality{m_i,fH_n} = \duality{fH_n,\phi} = \duality{f,\phi\circ\wop{H_n}}
$$
i.e. $m_i\circ\wop{H_n}\wsconv\phi\circ\wop{H_n}$. But clearly $m_i\circ\wop{H_n}\in\conv(\pm\delta(M))$ for all $i$ and $n$, hence $\phi\circ\wop{H_n}\in\wscl{\conv}(\pm\delta(M))$. To finish, notice that $\phi$ avoids infinity by Lemma \ref{lm:measure_avoids_infty}, hence $\phi\circ\wop{H_n}\rightarrow\phi$ and $\phi\in\wscl{\conv}(\pm\delta(M))$.

If $\phi$ is positive in the argument above, then $\mu$ may be chosen to be positive by Proposition \ref{pr:induced_supp_pos_bidual}, and we can apply the Krein-Milman theorem to $\pos{\ball{\meas{\ucomp{M}}}}$ instead of $\ball{\meas{\ucomp{M}}}$, which is a \weaks-compact convex set whose extreme points are $0$ and $\delta_\zeta$ for $\zeta\in\ucomp{M}$. We then get that $\phi\in\wscl{\conv}\delta(M)$.

For the converse implication, let $\phi\in\bidualfree{M}$ be the weak$^\ast$ limit of elements $m_i$ of $\lspan\delta(M)$ such that $\norm{m_i}_1\leq 1$ for all $i$. Then $m_i=\opint\mu_i$ where $\mu_i\in\meas{\ucomp{M}}$ has finite support and $\norm{\mu_i}=\norm{m_i}_1\leq 1$. Since $\ball{\meas{\ucomp{M}}}=\ball{\dual{C(\ucomp{M})}}$ is \weaks-compact, we may replace $(\mu_i)$ with a subnet that converges weak$^\ast$ to some $\mu\in\meas{\ucomp{M}}$, that is, such that $\lim_i \int_{\ucomp{M}} g\,d\mu_i=\int_{\ucomp{M}} g\,d\mu$ for each $g\in C(\ucomp{M})$. In particular, if $f\in\Lip_0(M)$ is bounded then we have
$$
\int_{\ucomp{M}}\ucomp{f}\,d\mu = \lim_i\int_{\ucomp{M}}\ucomp{f}\,d\mu_i = \lim_i\duality{f,m_i} = \duality{f,\phi} .
$$

Suppose first that $\phi$ is strongly bounded, i.e. $\phi=\phi\circ\wop{H_n}$ for some $n\in\NN$. Then $fH_n$ is bounded for every $f\in\Lip_0(M)$, so we have
$$
\duality{f,\phi} = \duality{f,\phi\circ\wop{H_n}} = \duality{fH_n,\phi} = \int_{\ucomp{M}}\ucomp{(fH_n)}\,d\mu = \int_{\ucomp{M}}\ucomp{f}\,d\lambda
$$
where $d\lambda=\ucomp{H_n}\,d\mu$ defines a measure in $\meas{\ucomp{M}}$.
Thus $\phi=\opint\lambda$ as required. 

Now assume that $\phi$ is positive and avoids infinity, i.e. $\phi=\lim_{n}\phi\circ\wop{H_n}$. For any function $f\in\Lip_0(M)$ and any $n$, the function $fH_n$ is bounded and $\ucomp{(fH_n)}=0$ on $\ucomp{M}\setminus\rcomp{M}$. Moreover, $(\ucomp{H_n})$ converges pointwise and increasingly to the characteristic function of $\rcomp{M}$. Therefore for a positive $f\in\Lip_0(M)$ we have
\begin{align*}
\duality{f,\phi}&=\lim_n\duality{f,\phi\circ\wop{H_n}}=\lim_n\duality{fH_n,\phi}=\lim_n\int_{\ucomp{M}}\ucomp{(fH_n)}\,d\mu\\
&=\lim_n\int_{\rcomp{M}}\ucomp{(fH_n)}\,d\mu=\lim_n\int_{\rcomp{M}}\ucomp{f}\ucomp{H_n}d\mu=\int_{\rcomp{M}}\ucomp{f}d\mu=\int_{\ucomp{M}}\ucomp{f}d(\mu\restrict_{\rcomp{M}})
\end{align*}
by Lebesgue's monotone convergence theorem. Decomposing any $f\in\Lip_0(M)$ as $f=f^+-f^-$, we may conclude that $\phi=\opint\mu\restrict_{\rcomp{M}}$. Obviously $\mu\restrict_{\rcomp{M}}\in\meas{\ucomp{M}}$ since $\mu\in\meas{\ucomp{M}}$.

To conclude, let us remark that $\mu$, and hence $\lambda$ (for strongly bounded $\phi$) or $\mu\restrict_{\rcomp{M}}$ (for positive $\phi$ that avoids infinity) are positive if all $\mu_i$ are.
\end{proof}

The converse of Theorem \ref{th:induced_elements_bidual} does not hold in general. We have proved that there is a measure $\mu$ whose action coincides with that of $\phi$ \textit{on all bounded functions $f\in\Lip_0(M)$}, but it might differ for unbounded $f$ when $\phi$ is not assumed to avoid infinity. The following shows that we can find a counterexample whenever $M$ is unbounded:

\begin{example}
\label{ex:counterexample_unbounded}
Suppose that $M$ is unbounded. Let $(x_n)$ be a sequence in $M$ such that $d(x_n,0)\rightarrow\infty$ and define $m_n=\delta(x_n)/d(x_n,0)$. Then $\norm{m_n}=1$, so there is a subnet $(m_{n_i})$ that converges weak$^\ast$ to some $\phi\in\bidualfree{M}$ which is clearly positive. We claim that $\phi$ cannot be represented by a positive Borel measure $\mu$ on $\ucomp{M}$, even if we allow it to be $\sigma$-finite. Suppose otherwise, then we may also assume that $\mu(\set{0})=0$ and Proposition \ref{pr:elements_induced_by_measure_bidual} implies that $\mu$ is concentrated on $\rcomp{M}$. Since $\rcomp{M}=\bigcup_{n=1}^\infty B_n$ where $B_n=\ucl{B(0,n)}$, we have $\mu(\rcomp{M})=\lim_n\mu(B_n)$.
Let $\rho_n=\rho\wedge n$, then
$$
\int_{B_n}\ucomp{\rho}\,d\mu = \int_{B_n}\ucomp{\rho_n}\,d\mu \leq \int_{\ucomp{M}}\ucomp{\rho_n}\,d\mu = \duality{\rho_n,\phi} = \lim_i\frac{\rho_n(x_{n_i})}{d(x_{n_i},0)} = 0
$$
which implies that $\mu(B_n)=0$ for all $n$. Thus $\mu=0$ and so $\phi=0$, but this contradicts the fact that $\duality{\rho,\phi}=1$.

So $\phi$ is not induced by a measure even if $\norm{m_n}_1=1/d(x_n,0)$ converges to $0$. The argument in the proof of Theorem \ref{th:induced_elements_bidual} still yields a measure $\mu$ such that the actions of $\phi$ and $\mu$ agree on bounded functions of $\Lip_0(M)$: it is just $\mu=0$.
\end{example}

For elements of $\lipfree{M}$ we have a similar result:

\begin{theorem}
\label{th:induced_elements}
Let $\phi\in\lipfree{M}$. If $\phi$ is induced by a Radon measure on $M$, then it is the limit of a sequence $(m_n)$ of elements of $\lspan\,\delta(M)$ such that $(\norm{m_n}_1)$ is bounded. If $\supp(\phi)$ is bounded or $\phi$ is positive, then the converse also holds. Moreover, $\phi$ is positive if and only if the $(m_n)$ can be chosen to be positive.
\end{theorem}

\begin{proof}
Assume that $\phi=\opint\mu\in\lipfree{M}$ for some $\mu\in\meas{M}$, and identify $\mu$ with an element of $\meas{\ucomp{M}}$ that is concentrated on $M$. Theorem \ref{th:induced_elements_bidual} then yields a net $(v_i)$ in $\lspan\,\delta(M)$ that converges to $\phi$ in $(\bidualfree{M},w^\ast)$, or equivalently in $(\lipfree{M},w)$, and such that $\norm{v_i}_1$ is bounded. By Mazur's lemma, $\phi$ is the limit of a sequence $(m_n)$ of convex combinations of the $v_i$. Clearly $\norm{m_n}_1$ is bounded by the same value as $\norm{v_i}_1$, and the $m_n$ are positive if all $v_i$ are.

For the converse implication, since all elements of $\lipfree{M}$ avoid infinity and the ones with bounded support are strongly bounded as functionals in $\bidualfree{M}$, for a $\phi$ as in the hypothesis, Theorem \ref{th:induced_elements_bidual} yields a measure $\mu\in\meas{\ucomp{M}}$ such that $\phi=\opint\mu$, which is moreover positive if the $m_n$ are. But $\phi\in\lipfree{M}$, so the measure has to be concentrated on $M$ by Theorem \ref{th:normal_measure} and therefore can be regarded as a Radon measure on $M$.
\end{proof}

Note that in Theorems \ref{th:induced_elements_bidual} and \ref{th:induced_elements} the norm $\norm{\mu}$ of the representing measure $\mu$ is related to the sum $\norm{m}_1$ of the coefficients of the approximating elements of finite support. So we get the following, more succinct characterization:

\begin{corollary}
The set of elements of $\lipfree{M}$ that can be represented by a measure in $\pos{\ball{\meas{M}}}$ is precisely $\cl{\conv}\,\delta(M)$. If $M$ is bounded, then the set of elements of $\bidualfree{M}$ that can be represented by a measure in $\pos{\ball{\meas{\ucomp{M}}}}$ is precisely $\wscl{\conv}\,\delta(M)$.
\end{corollary}

The next example shows that boundedness or positivity are again essential for the converse to hold in Theorem \ref{th:induced_elements}. In particular, it also shows that the converse in Theorem \ref{th:induced_elements_bidual} does not hold if $\phi$ avoids infinity but is not positive.

\begin{example}
Let $M\subset c_0$ consist of $0$ as the base point and the sequences $x_n=2^n e_n$ and $y_n=(2^n+1)e_n$, where $e_n$ are the standard basis vectors, and let
$$
m_n=\sum_{k=1}^n 2^{-k}(\delta(x_k)-\delta(y_k)) .
$$
Since $\norm{\delta(x_n)-\delta(y_n)}=d(x_n,y_n)=1$, the sequence $(m_n)$ is Cauchy and converges to $m\in\lipfree{M}$. Moreover $\norm{m_n}_1<2$ for every $n$. Nevertheless, $m$ cannot be represented by a Radon measure on $M$. Indeed, suppose $m=\opint\mu$ where $\mu\in\measz{M}$, then $\mu$ is supported on $\supp(m)=M\setminus\set{0}$ by Proposition \ref{pr:induced_supp_pos}, and it is clear that every $x_n$ belongs to the support of $\mu^+$. Denote $X_n=\set{x_1,\ldots,x_n}$, then $\rho\chi_{X_n}\in\Lip_0(M)$ and
$$
\int_M \rho\,d\abs{\mu} \geq \int_{X_n} \rho\,d\abs{\mu}
\geq \int_M \rho\chi_{X_n}\,d\mu=\duality{m,\rho\chi_{X_n}}=
\sum_{k=1}^n 2^{-k}\rho(x_k) = n
$$
for every $n$, which contradicts Proposition \ref{pr:elements_induced_by_measure}.
\end{example}

\section{Majorizable elements of \texorpdfstring{$\lipfree{M}$}{F(M)} and \texorpdfstring{$\bidualfree{M}$}{Lip0(M)*}}

Let us introduce the following definition:

\begin{definition}
\label{def:majorizable}
Let $X$ be an ordered vector space. We will say that an element $x\in X$ is \textit{majorizable} (more specifically, \textit{majorizable in $X$}) if there exists a positive element $x^+\in\pos{X}$ such that $x\leq x^+$. Equivalently, $x$ is majorizable if it may be written as the difference $x=x^+-x^-$ between two positive elements $x^+,x^-$ of $\pos{X}$.
\end{definition}

In this section we will study the majorizable elements of $\lipfree{M}$ and $\bidualfree{M}$. Let us start by recalling that not all elements of $\lipfree{M}$ are majorizable in general, as illustrated e.g. by Example 3.24 in \cite{Weaver2}, however all elements of $\lspan\delta(M)$ trivially are.

Notice that, for an element $m\in\lipfree{M}$, there are two separate notions of majorizability: $m$ may be majorizable in $\lipfree{M}$ or in $\bidualfree{M}$. That is, there may exist a positive $m^+$ in $\lipfree{M}$ such that $m\leq m^+$, or there may exist a positive $\phi^+$ in $\bidualfree{M}$ such that $m\leq\phi^+$. The latter is formally a weaker condition, and it is by no means obvious whether both conditions are equivalent. They are, as a matter of fact, as we will prove in Theorem \ref{th:minimal_majorant}.

\begin{remark}
\label{remark:majorants_avoiding_stuff}
If $\phi\in\bidualfree{M}$ is majorizable and avoids $0$ and infinity, and $\phi^+\in\pos{(\dual{\Lip_0(M)})}$ is its majorant, then there is a majorant $\psi\in\pos{(\bidualfree{M})}$ such that $\phi\leq\psi\leq\phi^+$ and $\psi$ also avoids $0$ and infinity. Indeed, just consider 
$\psi=(\phi^+)_s=\lim_n\phi^+\circ\wop{\Pi_n}$. 
Then clearly $\psi\leq\phi^+$ from positivity and
$$
\phi = \lim_n\phi\circ\wop{\Pi_n} \leq \lim_n\phi^+\circ\wop{\Pi_n} = (\phi^+)_s = \psi.
$$
A similar reasoning holds for functionals just avoiding either $0$ or infinity, or, on the contrary, concentrated at $0$ or at infinity. To see this, replace the function $\Pi_n$ above with $G_{\pm n}$ or $H_{\pm n}$ accordingly. From the uniqueness of the decomposition \eqref{eq:decomp_0_infty}, it follows that $\phi$ is majorizable if and only if each term of its decomposition is.
\end{remark}

\subsection{Characterizations of majorizable functionals}
\label{subs:Characterizations of majorizable functionals}

Let us now tackle the problem of characterizing the majorizable elements of $\lipfree{M}$ and $\bidualfree{M}$. It is possible to use similar reasoning to handle both cases simultaneously, as was done in the previous section. Observe first that any element $\phi\in\bidualfree{M}$ that is induced by a Borel measure $\mu$ on $\ucomp{M}$ is majorizable. Indeed, suppose $\phi=\opint\mu$, then $\abs{\mu}$ also induces an element of $\bidualfree{M}$ by Proposition \ref{pr:elements_induced_by_measure_bidual}, therefore $\phi$ is majorized by $\opint(\abs{\mu})$. Moreover, $\phi$ avoids $0$ and infinity by Lemma \ref{lm:measure_avoids_infty}. A similar argument applies for $\phi\in\lipfree{M}$.

It turns out that the converse of this observation is ``almost'' true. We start by showing that it holds for the elements that are supported away from the base point and that avoid infinity:

\begin{theorem}
\label{th:majorizable_elements_bidual_not0}
Suppose that $\phi\in\bidualfree{M}$ avoids infinity and $0\notin\esupp{\phi}$. Then $\phi$ is majorizable in $\bidualfree{M}$ if and only if it is induced by a Radon measure on $\ucomp{M}$.
\end{theorem}

\begin{proof}
As we have discussed above, it is not difficult to see that measure-induced functionals are majorizable. For the other implication, assume that $\phi=\phi^+-\phi^-$ where $\phi^+,\phi^-\in\bidualfree{M}$ are positive.
Note that by hypothesis and Remark \ref{remark:majorants_avoiding_stuff} we may assume that $\phi^+$ and $\phi^-$ avoid infinity.
Moreover, the second part of the hypothesis together with Proposition \ref{pr:separation_classes_equivalences_part2}(b) yields
$$\phi=\phi\circ\wop{G_{-n}}=\phi^+\circ\wop{G_{-n}}-\phi^-\circ\wop{G_{-n}}$$
for some $n$. So altogether we may assume that $\phi^\pm$ avoid infinity and $\phi^\pm=\phi^\pm\circ\wop{G_{-n}}$.

By Lemma \ref{lm:positive_sum_closure}, there is a net $(m_i)$ that converges weak$^\ast$ to $\phi^+$ and such that every $m_i$ is of the form
$$
m_i=\sum_{k=1}^N a_k^i\delta(x_k^i)
$$
for some $N\in\NN$, $a_k^i>0$ and $x_k^i\in M$. Thus $m_i\circ\wop{G_{-n}}\wsconv\phi^+\circ\wop{G_{-n}}=\phi^+$. Each $m_i\circ\wop{G_{-n}}$ is positive and
$$
\lim_i\norm{m_i\circ\wop{G_{-n}}} = \lim_i\duality{\rho,m_i\circ\wop{G_{-n}}} = \duality{\rho,\phi^+} = \norm{\phi^+} ,
$$
hence we may assume that $\norm{m_i\circ\wop{G_{-n}}}$ is bounded. Since $G_{-n}(x)=0$ for $x\in B(0,2^{-n})$, we have
\begin{align*}
\norm{m_i\circ\wop{G_{-n}}} =\duality{\rho,m_i\circ\wop{G_{-n}}} &= \sum_{k=1}^N a_k^i G_{-n}(x_k^i) d(x_k^i,0) \\
&\geq 2^{-n}\sum_{k=1}^N a_k^i G_{-n}(x_k^i) = 2^{-n}\norm{m_i\circ\wop{G_{-n}}}_1
\end{align*}
so $(\norm{m_i\circ\wop{G_{-n}}}_1)$ is bounded. Since $\phi^+$ is positive and avoids infinity, we may apply Theorem \ref{th:induced_elements_bidual} to conclude that $\phi^+$ is represented by a positive Radon measure $\mu^+$ on $\ucomp{M}$.

The same argument shows that $\phi^-$ is represented by a positive Radon measure $\mu^-$. Thus $\phi$ is represented by $\mu=\mu^+-\mu^-\in\meas{\ucomp{M}}$ and this ends the proof.
\end{proof}

In particular, if $M$ is bounded then the result applies to all $\phi$ such that $0\notin\esupp{\phi}$. Notice that the hypothesis of avoiding infinity cannot be omitted, as witnessed by the positive element constructed in Example \ref{ex:counterexample_unbounded}.

By virtue of the fact that elements of $\lipfree{M}$ avoid infinity, Theorem \ref{th:majorizable_elements_bidual_not0} and Theorem \ref{th:normal_measure} yield the corresponding result for $\lipfree{M}$. We also obtain the equivalence of majorizability in $\lipfree{M}$ and in $\bidualfree{M}$ under the additional assumption that the support does not contain the base point.

\begin{theorem}
\label{th:majorizable_elements_not0}
Let $m\in\lipfree{M}$ be such that $0\notin\supp(m)$.
Then the following are equivalent:
\begin{enumerate}[label={\upshape{(\roman*)}}]
\item $m$ is majorizable in $\lipfree{M}$,
\item $m$ is majorizable in $\bidualfree{M}$,
\item $m$ is induced by a Radon measure on $M$.
\end{enumerate}
\end{theorem}

\begin{remark}
\label{rm:ambrosio_puglisi}
Theorem \ref{th:majorizable_elements_not0} does not hold (and therefore, neither does Theorem \ref{th:majorizable_elements_bidual_not0}) for the case where the base point is contained in the (extended) support of a positive element. To see this, assume that the base point is not isolated and choose $x_n\in M\setminus\set{0}$ such that $d(x_n,0)<2^{-n}$ and $d(x_{n+1},0)<d(x_n,0)$ for every $n\in\NN$. Consider $m=\sum_{n=1}^{\infty} \delta(x_n)$, which is clearly an element of $\pos{\lipfree{M}}$ by absolute convergence. Suppose that $m$ is represented by a Radon measure $\mu$ on $M$. We may assume that $\mu\in\measz{M}$ and so $\mu$ is positive by Proposition \ref{pr:induced_supp_pos}. Let $g_n$ be a Lipschitz function on $\RR$ with $0\leq g_n\leq 1$, $g_n(t)=0$ for $t\leq d(x_{n+1},0)$ and $g_n(t)=1$ for $t\geq d(x_n,0)$. Then
$$
\norm{\mu} = \int_M 1\,d\mu \geq \int_M (g_n\circ\rho) \,d\mu = \duality{m,g_n\circ\rho} = n
$$
for any $n\in\NN$, therefore $\mu$ is not finite, a contradiction. Note however that $\mu$ is $\sigma$-finite.
\end{remark}

The construction in the preceding remark provides the following equivalence:

\begin{proposition}
\label{pr:all_majorizable_are_radon_measures}
The following are equivalent:
\begin{enumerate}[label={\upshape{(\roman*)}}]
\item the base point is an isolated point of $M$,
\item every majorizable element of $\lipfree{M}$ can be represented by a Radon measure on $M$,
\item every majorizable element of $\bidualfree{M}$ which avoids infinity can be represented by a Radon measure on $\ucomp{M}$.
\end{enumerate}
\end{proposition}

We will now extend Theorem \ref{th:majorizable_elements_bidual_not0} to all positive elements of $\bidualfree{M}$ that avoid $0$ and infinity and show that they can also be represented by a positive measure. The representing measure may not be finite, but it will always be \textit{almost Radon} (see Definition \ref{def:almost radon}); recall that this implies that it is inner regular and $\sigma$-finite, and its restriction to every closed set separated from $0$ is Radon. We will construct the desired measure as an inverse limit of Radon measures that are supported away from the base point, reversing the construction from Lemma \ref{lm:setwise_convergence}.

\begin{theorem}
\label{th:positive_elements_bidual}
Suppose that $\phi\in\bidualfree{M}$ is positive and avoids $0$ and infinity. Then $\phi$ is represented by a positive, almost Radon measure $\mu$ on $\ucomp{M}$. In fact, $\mu$ can be chosen as the setwise limit of $\mu_n$, where $\mu_n$ are Radon measures representing $\phi\circ\wop{G_{-n}}$ for $n\in\NN$.
\end{theorem}

\begin{proof}
For $n\in\NN$ denote
$$
A_n = \set{\zeta\in\ucomp{M}: 2^{-n} \leq \ucomp{\rho}(\zeta)}
$$
and $\phi_n=\phi\circ\wop{G_{-n}}$. Since $\phi$ avoids $0$, we have $\phi=\lim_n\phi_n$. Note that $\esupp{\phi_n}\subset A_{n}$ by Proposition \ref{pr:support_Th}, hence each $\phi_n$ avoids infinity and $0\notin\esupp{\phi_n}$, so by Theorem \ref{th:majorizable_elements_bidual_not0} it is induced by a Radon measure $\mu_n$ on $\ucomp{M}$ which obviously satisfies $\mu_n(\set{0})=0$. Proposition \ref{pr:induced_supp_pos_bidual} then shows that $\supp(\mu_n)\subset A_{n}$, and in particular
\begin{equation}
\label{eq:mu_evtlly_0}
\mu_n(E)=0 \text{ for every Borel set } E\subset \ucomp{M}\setminus A_{n} .
\end{equation}
We also have
$$
\opint(\mu_{n+1}-\mu_n) = \phi_{n+1}-\phi_n = \phi\circ\wop{G_{-(n+1)}-G_{-n}}.
$$
Since $G_{-(n+1)}-G_{-n}=\Lambda_{-n}\geq 0$, this is a positive element of $\bidualfree{M}$ and so $\mu_{n+1}-\mu_n\geq 0$ by Proposition \ref{pr:induced_supp_pos_bidual}. That is,
\begin{equation}
\label{eq:mu_increasing_setwise}
\mu_{n+1}(E)\geq\mu_n(E) \text{ for every Borel set } E\subset\ucomp{M} .
\end{equation}
Moreover, combining Propositions \ref{pr:induced_supp_pos_bidual} and \ref{pr:support_Th}
$$
\supp(\mu_{n+1}-\mu_n) = \esupp{\phi_{n+1}-\phi_n} \subset \ucl{\supp(\Lambda_{-n})}
$$
hence in particular we get
\begin{equation}
\label{eq:mu_evtlly_constant}
\mu_{n+1}(E)=\mu_n(E) \text{ for every Borel set } E\subset A_{n-2} .
\end{equation}

By \eqref{eq:mu_increasing_setwise}, the limit
$$
\mu(E)=\lim_{n\rightarrow\infty}\mu_n(E)
$$
exists and is positive (possibly infinite) for every Borel set $E\subset\ucomp{M}$. It is straightforward to check that $\mu$ is a Borel measure on $\ucomp{M}$. Notice also that $\mu(\set{0})=\mu(\ucomp{M}\setminus\rcomp{M})=0$ by Proposition \ref{pr:elements_induced_by_measure_bidual}. Moreover, \eqref{eq:mu_evtlly_constant} implies
\begin{equation}
\label{eq:mu_restrict_mun}
\mu\restrict_{A_n}=\mu_{n+2}\restrict_{A_n}
\end{equation}
for every $n$. Thus, if $K$ is any closed subset of $\ucomp{M}$ such that $0\notin K$, we have $K\subset A_n$ for some $n$ and therefore  $\mu\restrict_K=\mu_{n+2}\restrict_K$, which is a Radon measure. Hence $\mu$ is almost Radon.

We will now finish the proof by showing that $\phi$ is represented by $\mu$. Let $f\in\Lip_0(M)$ be positive. Since $\phi$ avoids $0$ and $\wop{G_{-n}}\circ\wop{G_{-(n+2)}}=\wop{G_{-n}}$, we may write
\begin{align*}
\duality{f,\phi} &= \lim_{n\to\infty}\duality{f,\phi_n} = \lim_{n\to\infty}\duality{f,\phi\circ\wop{G_{-n}}\circ\wop{G_{-(n+2)}}} \\
&= \lim_{n\to\infty}\duality{fG_{-n},\phi_{n+2}} = \lim_{n\rightarrow\infty}\int_{\ucomp{M}}\ucomp{\left(fG_{-n}\right)}\,d\mu_{n+2}\\
&= \lim_{n\rightarrow\infty}\int_{\ucomp{M}}\ucomp{f}\ucomp{G_{-n}}\,d\mu_{n+2} \\
&=  \lim_{n\rightarrow\infty}\int_{\ucomp{M}}\ucomp{f}\ucomp{G_{-n}}\,d\mu 
\end{align*}
where the last equality follows from the fact that $\ucomp{G_{-n}}=0$ on $\ucomp{M}\setminus A_n$ and identity \eqref{eq:mu_restrict_mun}. But $\ucomp{G_{-n}}$ converge pointwise and increasingly to the characteristic function of the set $\ucomp{M}\setminus\{0\}$, so by Lebesgue's monotone convergence theorem we obtain
$$
\duality{f,\phi} = \lim_{n\rightarrow\infty}\int_{\ucomp{M}}\ucomp{f}\ucomp{G_{-n}}\,d\mu = \int_{\ucomp{M}}\ucomp{f}\,d\mu
$$
for every positive $f\in\Lip_0(M)$, and hence for every $f\in\Lip_0(M)$.
\end{proof}

Since elements of $\lipfree{M}$ also avoid $0$ and infinity, appealing to Theorem \ref{th:normal_measure} we obtain one of the implications from \cite[Proposition 2.7]{AmPu_2016} as a corollary:

\begin{corollary}
\label{cr:positive_elements_free}
Every positive element of $\lipfree{M}$ can be represented by a positive, almost Radon measure on $M$.
\end{corollary}

Let us remark that, as is the case with finiteness, full regularity cannot be achieved in Theorem \ref{th:positive_elements_bidual} and Corollary \ref{cr:positive_elements_free} in general. Indeed, it is straightforward to check that any Borel measure $\mu$ representing the functional $m$ constructed in Remark \ref{rm:ambrosio_puglisi} satisfies $\mu(U)=\infty$ for every open neighborhood $U$ of $0$. Therefore, if $\mu$ is outer regular then $\mu(\set{0})=\infty$ and so $\mu$ cannot be $\sigma$-finite.

Although stated in terms of positive elements, Theorem \ref{th:positive_elements_bidual} and Corollary \ref{cr:positive_elements_free} also yield representation results for majorizable elements, as they can be written as the difference between two positive elements and therefore represented as the difference between two almost Radon positive measures. It is tempting to state that a majorizable functional $\phi$ is represented by a signed $\sigma$-finite measure, but that would be inaccurate because of a potential indeterminacy of the form $\infty-\infty$ around the base point. Therefore the most accurate statement would be the following:

\begin{theorem}
\label{th:majorizable_characterization}
Let $\phi\in\bidualfree{M}$. Then the following are equivalent:
\begin{enumerate}[label={\upshape{(\roman*)}}]
\item $\phi$ is majorizable and avoids $0$ and infinity,
\item $\phi$ is the difference of two elements of $\bidualfree{M}$ induced by positive almost Radon measures on $\ucomp{M}$.
\end{enumerate}
The same holds if $\bidualfree{M}$ and $\ucomp{M}$ are replaced by $\lipfree{M}$ and $M$, respectively, and the condition ``avoids $0$ and infinity'' is removed.
\end{theorem}

\begin{proof}
For (i)$\Rightarrow$(ii), write $\phi=\phi^+-\phi^-$ where $\phi^\pm$ are positive and may be assumed to avoid $0$ and infinity by Remark \ref{remark:majorants_avoiding_stuff}, and then apply Theorem \ref{th:positive_elements_bidual} (or Corollary \ref{cr:positive_elements_free} for the $\lipfree{M}$ case). The implication (ii)$\Rightarrow$(i) follows directly from Lemma \ref{lm:measure_avoids_infty}. 
\end{proof}

Let us now summarize our results for a few important particular cases. When $M$ is bounded, every element of $\bidualfree{M}$ avoids infinity (strongly, in fact) and hence by Theorem \ref{th:decomp_0_infty} it can be expressed as a derivation at $0$ plus a functional that avoids $0$ and infinity. So every positive element of $\bidualfree{M}$ is ``a derivation at $0$ plus a measure'':

\begin{corollary}
Suppose that $M$ is bounded, and let $\phi\in\bidualfree{M}$ be positive. Then $\phi=\phi_0+\opint\mu$ where $\phi_0$ is a positive derivation at $0$ and $\mu$ is a positive almost Radon measure on $\ucomp{M}$. If moreover $0\notin\esupp{\phi}$, then $\phi_0=0$ and $\mu\in\meas{\ucomp{M}}$.
\end{corollary}

Recall that there exist nontrivial positive derivations at the base point if and only if it is not isolated; see the comment before Proposition \ref{pr:norm_sum_deriv} for a simple example. In the case where $0$ is an isolated point we get the following characterization, combining with Theorems \ref{th:induced_elements_bidual} and \ref{th:induced_elements}. A particular case of this arises with $\Lip(M)$ spaces; see Section \ref{sec:conclusion} for more details.

\begin{corollary}
\label{cr:equiv_bounded_isolated}
Suppose that $M$ is bounded and the base point is isolated. Then for $\phi\in\bidualfree{M}$ the following are equivalent:
\begin{enumerate}[label={\upshape{(\roman*)}}]
\item $\phi$ is majorizable,
\item $\phi$ is represented by a Radon measure on $\ucomp{M}$,
\item $\phi$ is the weak$^\ast$ limit of a net $(m_i)$ in $\lspan\,\delta(M)$ such that $\norm{m_i}_1$ is bounded.
\end{enumerate}
The same holds if we replace $\bidualfree{M}$ with $\lipfree{M}$, $\ucomp{M}$ with $M$, ``weak$^\ast$ limit'' with ``norm limit'' and ``net'' with ``sequence'', respectively.
\end{corollary}

The most significant particular case of the preceding analysis is given by compact metric spaces $M$. In that case $\ucomp{M}=M$, so Proposition \ref{pr:elements_induced_by_measure} shows that $\opint\mu$ actually belongs to $\lipfree{M}$. Thus the majorizable elements of $\lipfree{M}$ and $\bidualfree{M}$ are almost the same:

\begin{corollary}
\label{cr:bidual_majorizable_compact}
Suppose that $M$ is compact, and let $\phi\in\bidualfree{M}$ be majorizable. Then $\phi=m+\phi_0$ where $m\in\lipfree{M}$ is majorizable and $\phi_0$ is a derivation at $0$. If moreover $0\notin\esupp{\phi}$, then $\phi\in\lipfree{M}$ and it is represented by a Radon measure on $M$.
\end{corollary}

Note that the weak$^*$ continuity of any positive functional $\phi\in\bidualfree{M}$ when $M$ is compact with isolated base point can also be proved directly without using measures. Let us sketch the argument. If $M$ is compact and $(f_i)\subset\Lip_0(M)^+$ is a bounded net converging pointwise to $0$, then it already converges uniformly. If, moreover, $0\in M$ is isolated, then the function $e$ on $M$ defined by $e(x)=1$ for every $x\in M\setminus\{0\}$ and $e(0)=0$ belongs to $\Lip_0(M)$. Hence, for every $\varepsilon>0$ there exists $i$ such that $f_j\leq\varepsilon e$ for every $j\succcurlyeq i$, and the conclusion follows from the positivity of $\phi$.

\subsection{Minimum majorants}

Now we turn to the following problem. An element $\phi$ of $\lipfree{M}$ or $\bidualfree{M}$ is majorizable precisely when it can be expressed as the difference between two positive elements. Does there exist a canonical, ``minimal'' representation as such a difference? Let us fix some notation:

\begin{definition}
Let $X$ be an ordered vector space and $x\in X$. A \textit{majorant} of $x$ is a positive element $x^+\in\pos{X}$ such that $x\leq x^+$. If there exists a majorant $x^+$ of $x$ with the property that every majorant of $x$ is also a majorant of $x^+$, we call such $x^+$ the \textit{minimum majorant} of $x$.
\end{definition}

Our question can be rephrased as: does every majorizable element have a minimum majorant? It is obvious that the minimum majorant is unique whenever it exists, and that if $x^+$ is the minimum majorant of $x$ then $x^-=x^+-x$ is the minimum majorant of $-x$. Moreover, the existence of such a minimum majorant is equivalent to the existence of an optimal representation $x=x^+-x^-$ where $x^+,x^-$ are positive and satisfy the following minimum property:
\textit{for every expression $x=y^+-y^-$ where $y^+,y^-\geq 0$ we have $y^+\geq x^+$ and $y^-\geq x^-$}. This behaviour is found in $\Lip_0(M)$ with the optimal decomposition $f=f^+-f^-$, and also in finite measures where the Jordan decomposition into their positive and negative parts is optimal in that sense. In both cases, the minimum majorant of an element can be identified with its positive part.

It is reasonable to expect similar properties from majorizable functionals on $\Lip_0(M)$, given their close relationship to measures. We will see in Corollary \ref{cr:minimal_majorant_general} that our suspicions are correct. Let us first establish that fact for functionals that avoid $0$ and infinity, where the representation by measures provides additional information.

\begin{theorem}
\label{tm:minimal_majorant_bidual}
Suppose that $\phi\in\bidualfree{M}$ avoids $0$ and infinity. If $\phi$ is majorizable then it has a minimum majorant $\phi^+$ that also avoids $0$ and infinity.
Moreover, $\phi^+$ and $\phi^-=\phi^+-\phi$ are represented by positive almost Radon measures on $\ucomp{M}$ that are concentrated on disjoint Borel subsets of $\ucomp{M}$.
\end{theorem}

In the proof we will use the following simple lemma about minimum majorants of weighted functionals:

\begin{lemma}
\label{lm:minimal_majorant_weighted}
Suppose that $\phi$ is a majorizable element in $\lipfree{M}$ or $\bidualfree{M}$ that has a minimum majorant $\phi^+$. Let $h\in\pos{\Lip(M)}$ be such that $\norm{h}_{\infty}<\infty$ and either of the functions $h$ or $\norm{h}_{\infty}-h$ has a bounded support. Then $\phi\circ\wop{h}$ is majorizable and its minimum majorant is $\phi^+\circ\wop{h}$.
\end{lemma}

\begin{proof}
Without loss of generality assume that $\norm{h}_\infty=1$ and notice that both $\wop{h}$ and $\wop{1-h}=I-\wop{h}$ are \weaks-\weaks-continuous operators on $\Lip_0(M)$. It is clear that $\phi\circ\wop{h}$ is majorized by $\phi^+\circ\wop{h}\geq 0$ and, similarly, $\phi\circ\wop{1-h}$ is majorized by $\phi^+\circ\wop{1-h}\geq 0$. Now suppose that $\psi\in\pos{(\bidualfree{M})}$ is another majorant for $\phi\circ\wop{h}$, then
$$
\psi + \phi^+\circ\wop{1-h} \geq \phi\circ\wop{h} + \phi\circ\wop{1-h} = \phi.
$$
Hence $\psi + \phi^+\circ\wop{1-h}\geq \phi^+$ by minimality and therefore $\psi\geq \phi^+ -\phi^+\circ\wop{1-h}=\phi^+\circ\wop{h}$ as claimed.
\end{proof}

\begin{proof}[Proof of Theorem \ref{tm:minimal_majorant_bidual}]
We will first prove the theorem under the assumption that $0\notin\esupp{\phi}$. By Theorem \ref{th:majorizable_elements_bidual_not0} we have $\phi=\opint\mu$ for some $\mu\in\meas{\ucomp{M}}$. Let $\mu=\mu^+-\mu^-$ be the Jordan decomposition of $\mu$. Then $\phi^+=\opint\mu^+$ and $\phi^-=\opint\mu^-$ are elements of $\bidualfree{M}$ by Proposition \ref{pr:elements_induced_by_measure_bidual} and Remark \ref{rm:restricted_functionals}, and avoid $0$ and infinity by Lemma \ref{lm:measure_avoids_infty}. Moreover, $0\notin\esupp{\phi^{\pm}}$ by Proposition \ref{pr:induced_supp_pos_bidual}. Clearly, $\phi^+$ is a majorant for $\phi$; we claim that it is the minimum majorant. Let $\psi$ be another majorant for $\phi$, we need to check that $\psi\geq\phi^+$. By Remark \ref{remark:majorants_avoiding_stuff}, $\psi$ may be replaced by a smaller majorant that avoids infinity. Moreover, there is $n\in\NN$ such that $\phi=\phi\circ\wop{G_{-n}}$ by assumption, hence $\psi\geq\psi\circ\wop{G_{-n}}\geq\phi$.
Theorem \ref{th:majorizable_elements_bidual_not0} and Proposition \ref{pr:induced_supp_pos_bidual} now imply that $\psi\circ\wop{G_{-n}}$ is represented by a positive measure $\lambda\in\meas{\ucomp{M}}$. Then $\opint\lambda=\psi\circ\wop{G_{-n}}\geq\opint\mu$ and thus $\lambda\geq\mu$ again by Proposition \ref{pr:induced_supp_pos_bidual}. By the Hahn decomposition theorem, we must have $\lambda\geq\mu^+$ and therefore $\psi\geq\opint\lambda\geq\opint\mu^+=\phi^+$. This proves our claim. Moreover $\phi=\phi^+-\phi^-$, and $\mu^+$ and $\mu^-$ are concentrated on disjoint Borel subsets of $\ucomp{M}$. That ends the proof of this case. 

Let us now handle the general case. For each $n\in\NN$, $\phi_n=\phi\circ\wop{G_{-n}}$ is majorizable, avoids infinity, and $0$ is not in its extended support. Therefore it has a minimum majorant $\phi_n^+$ by the previous paragraph. Now let $\psi\in\bidualfree{M}$ be any majorant of $\phi$. Then $\psi \geq \psi\circ\wop{G_{-n}} \geq \phi_n$ and so $\psi\geq\phi_n^+$. Notice also that $(\phi_n^+)$ is a bounded sequence since
$$
\norm{\phi_n^+} = \duality{\rho,\phi_n^+} \leq \duality{\rho,\psi} = \norm{\psi} ,
$$
therefore it must have a \weaks-cluster point $\phi^+\in\pos{(\bidualfree{M})}$, which will obviously satisfy $\phi^+\leq\psi$. Taking weak$^\ast$ limits in $\phi_n^+\geq\phi_n$ for the appropriate subnet yields $\phi^+\geq\lim_n\phi_n=\phi$. Thus $\phi^+$ is a majorant for $\phi$, and it is the minimum one because $\psi$ was arbitrary. This proves the existence.
The fact that $\phi^+$ avoids $0$ and infinity is an immediate consequence of Remark \ref{remark:majorants_avoiding_stuff}.

Let $\phi^-=\phi^+-\phi$ and $\phi_n^-=\phi_n^+-\phi_n$. We have already proved that $\phi_n^+,\phi_n^-$ are represented by positive Radon measures $\mu_n^+,\mu_n^-$ concentrated on disjoint Borel sets $A_n^+,A_n^-$. Since $\phi_m^+=\phi^+\circ\wop{G_{-m}}\geq\phi^+\circ\wop{G_{-n}}=\phi^+_n$ for every $m\geq n$ by Lemma \ref{lm:minimal_majorant_weighted}, the linearity of $\opint$ on Radon measures and Proposition \ref{pr:induced_supp_pos_bidual} imply that $\mu_{m}^+\geq\mu_n^+$. Therefore $\mu_n^+(A_n^+\setminus A_{m}^+)=0$ and $\mu_n^+$ is concentrated on $A_m^+$.
It follows that each $\mu_n^+$ is concentrated on $\bigcap_{m=n}^{\infty}A_m^+$, and thus also on the set
$$
A^+=\bigcup_{n=1}^\infty\bigcap_{m=n}^\infty A_m^+.
$$
By Theorem \ref{th:positive_elements_bidual}, $\phi^+$ is represented by a positive almost Radon measure $\mu^+$ on $\ucomp{M}$ which is constructed as the setwise limit of the Radon measures on $\ucomp{M}$ representing functionals $\phi^+\circ\wop{G_{-n}}$. But we have $\phi^+\circ\wop{G_{-n}}=\phi^+_n$ by Lemma \ref{lm:minimal_majorant_weighted}, so $\mu^+$ is the setwise limit of $\mu_n^+$ as $n$ tends to $\infty$ and thus $\mu^+$ is also concentrated on $A^+$.

Finally, recall that if $\phi^+$ is the minimum majorant of $\phi$ then $\phi^-=\phi^+-\phi$ is the minimum majorant of $-\phi$. Therefore a similar argument with $-\phi$, $\mu_n^-$ and $A_n^-$ shows that the almost Radon measure $\mu^-$ representing $\phi^-$ is concentrated on $A^-=\bigcup_n\bigcap_{m\geq n} A_m^-$. Since $A^+$ and $A^-$ are disjoint Borel subsets of $\ucomp{M}$, this finishes our proof.
\end{proof}

Since elements of $\lipfree{M}$ avoid $0$ and infinity, we can now deduce that the majorizable ones also have minimum majorants. It is possible to prove this following the argument used in Theorem \ref{tm:minimal_majorant_bidual}, but we will instead deduce it as a consequence. Moreover, we finally obtain in full generality the promised result that the properties of being majorizable in $\lipfree{M}$ and $\bidualfree{M}$ are equivalent.

\begin{theorem}
\label{th:minimal_majorant}
Let $m\in\lipfree{M}$. Then $m$ is majorizable in $\lipfree{M}$ if and only if it is majorizable in $\bidualfree{M}$. In that case, it has a minimum majorant $m^+$ that belongs to $\lipfree{M}$. Moreover, $m^+$ and $m^-=m^+-m$ are represented by positive almost Radon measures on $M$ that are concentrated on disjoint Borel subsets of $M$.
\end{theorem}

\begin{proof}
It is clear that if $m$ is majorizable in $\lipfree{M}$ then it is also majorizable in $\bidualfree{M}$. Now suppose that $m$ is majorizable in $\bidualfree{M}$. Since $m$ satisfies the hypothesis of Theorem \ref{tm:minimal_majorant_bidual}, it has a minimum majorant $m^+\in\pos{(\bidualfree{M})}$ that avoids $0$ and infinity. Thus $m^+$ is the norm limit of $m^+\circ\wop{\Pi_n}$, so in order to prove that $m^+\in\lipfree{M}$ it will be enough to show that $m^+\circ\wop{\Pi_n}\in\lipfree{M}$ for every $n\in\NN$. Fix $n$, then Lemma \ref{lm:minimal_majorant_weighted} implies that $m^+\circ\wop{\Pi_n}$ is the minimum majorant of $m\circ\wop{\Pi_n}$. By Theorem \ref{th:majorizable_elements_not0} there exists $\mu_n\in\meas{M}$ such that $m\circ\wop{\Pi_n}=\opint\mu_n$. Then $\opint\mu_n^+\in\pos{\lipfree{M}}$ by Remark \ref{rm:restricted_functionals} and it clearly is a majorant of $m\circ\wop{\Pi_n}$. Hence by the minimality we get that $m^+\circ\wop{\Pi_n}\leq\opint\mu_n^+$, and Lemma \ref{lm:positive_functional_lemma} yields that $m^+\circ\wop{\Pi_n}\in\lipfree{M}$ as well. Thus $m$ is majorizable in $\lipfree{M}$, and clearly its minimum majorant in $\lipfree{M}$ is also $m^+$.

Finally, by Theorem \ref{tm:minimal_majorant_bidual}, $m^+$ and $m^-$ are represented by positive almost Radon measures $\mu^+$, $\mu^-$ on $\ucomp{M}$ concentrated on disjoint Borel sets. But $m^\pm\in\lipfree{M}$, hence $\mu^+$, $\mu^-$ are actually concentrated on $M$ by Theorem \ref{th:normal_measure}. This ends the proof.
\end{proof}

Notice that Theorems \ref{th:minimal_majorant} and \ref{tm:minimal_majorant_bidual} show that majorizable elements in $\lipfree{M}$ and $\bidualfree{M}$ that avoid $0$ and infinity can \textit{almost} be represented as measures with a Hahn decomposition: in general they cannot be represented as a single measure, but they are always given by a difference of two minimal positive measures that are concentrated on disjoint Borel sets. Let us remark that this separation property is valid only in terms of the Borel sets on which the measures concentrate but not for their closed supports, hence neither for the supports of the functionals themselves, i.e. we can find simple examples where $\supp(m^+)\cap\supp(m^-)\neq\varnothing$.

Let us now extend Theorem \ref{tm:minimal_majorant_bidual} to all majorizable $\phi\in\bidualfree{M}$. We may do so by employing a technique developed by Albiac, Ansorena, C\'uth and Doucha in the recent paper \cite{AACD_2021} (see also \cite[Theorem 2.20]{Weaver2}). To any metric space $M$ one may assign another metric space $\mathcal{B}$, defined as the closed metric subspace of $\lipfree{M}$ determined by the normalized evaluation functionals, i.e.
\begin{equation}
\label{eq:aacd_space}
\mathcal{B}=\cl{\set{m_x:x\in M}}\subset\lipfree{M} \quad\text{where}\quad m_x=\begin{cases} \dfrac{\delta(x)}{\rho(x)} &\text{, if $x\neq 0$} \\ 0 &\text{, if $x=0$.} \end{cases}
\end{equation}
Then $\mathcal{B}$ is complete and bounded with isolated base point, and there is an onto linear isomorphism $P:\lipfree{M}\to\lipfree{\mathcal{B}}$ given by
\begin{equation}
\label{eq:aacd_isomorphism}
P(\delta_M(x))=\rho(x)\cdot\delta_{\mathcal{B}}(m_x)
\end{equation}
for $x\in M$ (see \cite[Theorem 3.9]{AACD_2021}; note that the metric space $\mathcal{B}$ used in \cite{AACD_2021} is not the same as in \eqref{eq:aacd_space}, but a dense subset thereof). Lemma \ref{lm:positive_sum_closure} implies easily that $P$ is order-preserving. Thus its extension $\ddual{P}:\bidualfree{M}\to\bidualfree{\mathcal{B}}$ is also an onto, order-preserving linear isomorphism. In particular, it preserves majorizable elements and minimum majorants. So we conclude:

\begin{corollary}
\label{cr:minimal_majorant_general}
Every majorizable element of $\bidualfree{M}$ has a minimum majorant.
\end{corollary}

\begin{proof}
Because $\mathcal{B}$ is bounded and its base point is isolated, all elements of $\bidualfree{\mathcal{B}}$ avoid $0$ and infinity strongly. Theorem \ref{tm:minimal_majorant_bidual} and the remarks above then yield the desired conclusion.
\end{proof}

\medskip
We conclude this section by introducing a notion of variation for majorizable functionals in analogy with measures. Its existence is guaranteed by Corollary \ref{cr:minimal_majorant_general}.

\begin{definition}
\label{def:variation}
Suppose that $\phi\in\bidualfree{M}$ is majorizable. The minimum majorant of $\phi$ will be denoted by $\phi^+$, and $\phi^-=\phi^+-\phi$. We will call the \textit{variation} of $\phi$ the functional $\abs{\phi}\in\bidualfree{M}$ defined by
$$\abs{\phi}=\phi^++\phi^-.$$
\end{definition}

Note that by Theorem \ref{th:minimal_majorant} the variation of a majorizable element from $\lipfree{M}$ also belongs to $\lipfree{M}$. We always have $\norm{\phi}\leq\norm{\phi^+}+\norm{\phi^-}=\norm{\abs{\phi}}$, and equality is possible even if $\phi$ is not positive or negative. For instance, in $M=\RR$ let $\phi=\delta(1)-\delta(-1)$, then $\abs{\phi}=\delta(1)+\delta(-1)$ and $\norm{\phi}=\norm{\abs{\phi}}=2$.

The variation of $\phi$ is obviously also its majorant, but moreover, it majorizes the modulus of $\phi$ in the following sense:

\begin{proposition}
Let $\phi\in\bidualfree{M}$ be majorizable. Then its variation $\abs{\phi}$ is the smallest element of $\bidualfree{M}$ such that $\abs{\phi}\geq\phi^+$ and $\abs{\phi}\geq\phi^-$. It is also the smallest element of $\bidualfree{M}$ such that
$$
\abs{\duality{f,\phi}}\leq\duality{\abs{f},\abs{\phi}}\quad\textup{ for every }f\in\Lip_0(M).
$$
\end{proposition}

\begin{proof}
Let us start with the first claim. It is clear that $\abs{\phi}\geq\phi^\pm$. Now let $\psi\in\bidualfree{M}$ be such that $\psi\geq\phi^\pm$, and we will prove that $\psi\geq\abs{\phi}$.

Suppose first that $\phi$ avoids $0$ and infinity. Then we may assume that $\psi$ also avoids $0$ and infinity by Remark \ref{remark:majorants_avoiding_stuff}. By Theorem \ref{th:positive_elements_bidual}, $\phi^+$, $\phi^-$ and $\psi$ are represented by positive almost Radon measures $\mu^+$, $\mu^-$ and $\lambda$ on $\ucomp{M}$, so Proposition \ref{pr:induced_supp_pos_bidual} yields $\lambda\geq\mu^+$ and $\lambda\geq\mu^-$. However, by Theorem \ref{tm:minimal_majorant_bidual}, $\mu^+$ and $\mu^-$ are concentrated on disjoint Borel sets, and so it is immediate that we actually have $\lambda\geq\mu^++\mu^-$. Thus $\psi\geq\phi^++\phi^-=\abs{\phi}$ by Proposition \ref{pr:induced_supp_pos_bidual}.

In the general case, let us consider again the space $\mathcal{B}$ and the isomorphism $P$ described in \eqref{eq:aacd_space} and \eqref{eq:aacd_isomorphism}. Then $\psi\geq\phi^{\pm}$ implies $\ddual{P}\psi\geq\ddual{P}(\phi^{\pm})=(\ddual{P}\phi)^{\pm}$. But all functionals in $\bidualfree{\mathcal{B}}$ avoid $0$ and infinity, so by the previous case we get $\ddual{P}\psi\geq (\ddual{P}\phi)^++(\ddual{P}\phi)^-=\ddual{P}(\phi^+)+\ddual{P}(\phi^-)$ and thus $\psi\geq\phi^++\phi^-=\abs{\phi}$ again. This finishes the proof of the first statement.

Let us now move on to the second statement, and check that $\abs{\phi}$ satisfies the desired property. Assume first that $f\in\Lip_0(M)$ and $f\geq 0$. Then $$\duality{f,\phi}\leq\duality{f,\phi^+}\leq\duality{f,\abs{\phi}}$$
and
$$-\duality{f,\phi}=\duality{f,-\phi}\leq\duality{f,\phi^-}\leq\duality{f,\abs{\phi}},$$ hence
$\abs{\duality{f,\phi}}\leq\duality{f,\abs{\phi}}$.
Now decompose any $f\in\Lip_0(M)$ as $f=f^+-f^-$ where $f^+,f^-\geq 0$. Then
\begin{align*}
\abs{\duality{f,\phi}}&\leq\abs{\duality{f^+,\phi}}+\abs{\duality{f^-,\phi}}\leq\duality{f^+,\abs{\phi}}+\duality{f^-,\abs{\phi}}\\
&=\duality{|f|,\abs{\phi}}.
\end{align*}

Finally, suppose that $\psi\in\bidualfree{M}$ is such that $\abs{\duality{f,\phi}}\leq\duality{\abs{f},\psi}$ for every $f\in\Lip_0(M)$. This clearly implies $\psi\geq 0$. For positive $f$ we get $\duality{f,\psi}\geq\abs{\duality{f,\phi}}$, hence $\psi\geq\phi$ and $\psi\geq -\phi$, and therefore also $\psi\geq\phi^+$ and $\psi\geq (-\phi)^+=\phi^-$. The first part of the proposition now yields $\psi\geq\abs{\phi}$ and this finishes the proof.
\end{proof}

We finish by establishing the intuitively obvious fact that passing from a majorizable functional to its minimum majorant or its variation does not increase its support.

\begin{proposition}
\label{pr:variation_support}
Let $\phi\in\bidualfree{M}$ be majorizable and avoid infinity. Then
$$
\esupp{\abs{\phi}}=\esupp{\phi^+}\cup\esupp{\phi^-}=\esupp{\phi} .
$$
If $\phi\in\lipfree{M}$, the equality holds also for supports in $\lipfree{M}$ in place of extended supports.
\end{proposition}

\begin{proof}
Apply first Corollary \ref{cr:support_of_sum} to observe that $\esupp{\abs{\phi}}\subset\esupp{\phi^+}\cup\esupp{\phi^-}$ and $\esupp{\phi}\subset\esupp{\phi^+}\cup\esupp{\phi^-}$. The inclusion $\esupp{\phi^+}\cup\esupp{\phi^-}\subset\esupp{\abs{\phi}}$ follows from statement (c) in Corollary \ref{cr:positive_facts_bidual}. Thus it only remains to be proved that $\esupp{\phi^+},\esupp{\phi^-}\subset\esupp{\phi}$. It will suffice to show that $\esupp{\phi^+}\subset\esupp{\phi}$, then $\esupp{\phi^-}\subset\esupp{\phi}$ follows from Corollary \ref{cr:support_of_sum} again. Moreover, $\phi^+$ avoids infinity by Remark \ref{remark:majorants_avoiding_stuff}, hence in view of Proposition \ref{pr:extended_support_avoid_inf} it is enough to prove that $\esupp{\phi^+}\cap\rcomp{M}\subset \esupp{\phi}$.

Let $\zeta\in \rcomp{M}\setminus\esupp{\phi}$. By Proposition \ref{pr:equiv_points_support_ext} there is a neighbourhood $U$ of $\zeta$ such that $\duality{f,\phi}=0$ for any $f\in\Lip_0(M)$ with the support contained in $U\cap M$. We may moreover assume that $U\subset \set{\xi\in\ucomp{M}: \ucomp{\rho}(\xi)<\ucomp{\rho}(\zeta)+1}$, so that $U\cap M$ is bounded. Take other neighbourhoods $V$, $W$ of $\zeta$ such that $\ucl{W}\subset V\subset \ucl{V}\subset U$ and define $h\in\Lip(M)$ so that $0\leq h\leq 1$ and that $h=1$ on $W\cap M$ and $h=0$ on $M\setminus V$. The support of such $h$ is bounded, so we may define $\psi=\phi^+\circ\wop{1-h}$. Then $\psi\leq \phi^+$, and for any positive $f\in\Lip_0(M)$ we have
$$
\duality{f,\psi} = \duality{f(1-h),\phi^+} \geq \duality{f(1-h),\phi} = \duality{f,\phi} - \duality{fh,\phi} = \duality{f,\phi}
$$
since $\supp(fh)\subset U\cap M$. Thus $\psi$ is a majorant for $\phi$, and from the minimality of $\phi^+$ it follows that $\phi^+=\psi$. Hence $\phi^+\circ\wop{h}=0$. But then $\duality{f,\phi^+}=\duality{fh,\phi^+}=0$ for any $f\in\Lip_0(M)$ such that $\supp(f)\subset W\cap M$. Therefore $\zeta\notin\esupp{\phi^+}$ by Proposition \ref{pr:equiv_points_support_ext}. This completes the proof for functionals in $\dual{\Lip_0(M)}$.

If $\phi\in\lipfree{M}$, then also $\phi^+,\phi^-\in\lipfree{M}$ by Theorem \ref{th:minimal_majorant},
and the equality for supports in $\lipfree{M}$ follows by intersecting with $M$ and applying Corollary \ref{cr:extended_support_fm}.
\end{proof}

The first equality in Proposition \ref{pr:variation_support} is valid for all majorizable functionals $\phi\in\bidualfree{M}$, but we do not know whether the second one holds in general when $\phi$ does not avoid infinity.

\section{Radially discrete spaces}

We have already witnessed that, in general, not all elements of a Lipschitz-free space $\lipfree{M}$ can be represented by a measure, or as the difference between two positive elements (see \cite[Example 3.24]{Weaver2}). In this section, we will identify the scenarios where such representations are always possible. To this end, recall first that $M$ is \textit{uniformly discrete} if
$$
\theta(M):=\inf\set{d(x,y):x\neq y\in M}>0 .
$$
This value $\theta(M)$ will be called the \textit{uniform separation constant} of $M$. Now we introduce two more classes of metric spaces:

\begin{definition}
\label{def:rud}
We will say that a pointed metric space $M$ is \textit{radially discrete} if there exists $\alpha>0$ such that $d(x,y)\geq\alpha\cdot d(x,0)$ for every pair $x,y$ of distinct points of $M$. We will say that $M$ is \textit{radially uniformly discrete} if it is radially discrete and uniformly discrete.
\end{definition}

Note that if $M$ is radially discrete, then the set $M\setminus B(0,r)$ is uniformly discrete for every $r>0$, and its uniform separation constant increases linearly with $r$. In particular, every point of $M\setminus\set{0}$ is isolated. However the base point need not be isolated, so in particular $M$ is not necessarily discrete. In fact, $M$ is uniformly discrete if and only if the base point is also isolated. Thus, the property of being radially discrete depends on the choice of base point, but it is easy to see that the property of being radially uniformly discrete does not.

Notice also that if $M$ is uniformly discrete and bounded then it is also radially uniformly discrete, taking $\alpha=\theta(M)/\diam(M)$ where $\theta(M)$ is the uniform separation constant of $M$. Unbounded uniformly discrete spaces are not necessarily radially discrete: consider e.g. $M=\NN$ with the metric inherited from $\RR$.

The following result is our main reason to introduce these classes of metric spaces:

\begin{theorem}
\label{th:all_majorizable}
The following are equivalent:
\begin{enumerate}[label={\upshape{(\roman*)}}]
\item $M$ is radially discrete,
\item all elements of $\lipfree{M}$ are majorizable,
\item all elements of $\bidualfree{M}$ are majorizable,
\item there exists a linear order-preserving isomorphism from $\lipfree{M}$ onto $\ell_1(M\setminus\{0\})$.
\end{enumerate}
If (iv) holds, then one such isomorphism is given by $\delta(x)\mapsto\rho(x) e_x$ for $x\in M\setminus\{0\}$, where $e_x$ are the standard basis elements of $\ell_1(M\setminus\{0\})$.
\end{theorem}

\begin{proof}
We will prove the implications (i)$\Rightarrow$(iv)$\Rightarrow$(iii)$\Rightarrow$(ii)$\Rightarrow$(i).

To start with the implication (i)$\Rightarrow$(iv), suppose $M$ is radially discrete and let $\alpha>0$ be such that $d(x,y)\geq \alpha\cdot d(x,0)$ for every $x\neq y\in M$. Define $T:\lspan \delta(M)\to \ell_1(M\setminus\{0\})$ linearly by $T(\delta(x))=\rho(x)e_x$ for every $x\in M\setminus\{0\}$. Fix some $\sum_{i=1}^na_i\delta(x_i)\in \lspan \delta(M)$ and let $f(0)=0$ and $f(x_i)=\sgn(a_i)\,d(x_i,0)$ for all $1\leq i\leq n$. We have
$$
|f(x_i)-f(x_j)|\leq |f(x_i)|+|f(x_j)|\leq d(x_i,0)+d(x_j,0)\leq \frac{2}{\alpha}d(x_i,x_j)
$$
and $f$ can be extended to a function in $\Lip_0(M)$ with $\lipnorm{f}\leq\frac{2}{\alpha}$. It then follows that 
$$
\norm{T\left(\sum_{i=1}^na_i\delta(x_i)\right)}_{\ell_1}=\sum_{i=1}^n|a_i|\rho(x_i)=\duality{f,\sum_{i=1}^na_i\delta(x_i)}\leq \frac{2}{\alpha}\norm{\sum_{i=1}^na_i\delta(x_i)}_{\lipfree{M}}.
$$
On the other hand,
$$
\norm{\sum_{i=1}^na_i\delta(x_i)}_{\lipfree{M}}\leq\sum_{i=1}^n|a_i|\norm{\delta(x_i)}_{\lipfree{M}}=\sum_{i=1}^n|a_i|\rho(x_i)=\norm{T\left(\sum_{i=1}^na_i\delta(x_i)\right)}_{\ell_1}.
$$
Since finitely supported elements are dense both in $\lipfree{M}$ and in $\ell_1(M\setminus\set{0})$, we have thus verified that $T$ defines an isomorphism from $\lipfree{M}$ onto $\ell_1(M\setminus\{0\})$.
Moreover, clearly $\sum_{i=1}^na_i\delta(x_i)\in\lipfree{M}^+$ if and only if $a_i\geq 0$ for all $1\leq i\leq n$, and Lemma \ref{lm:positive_sum_closure} implies that $T$ is order-preserving.

For (iv)$\Rightarrow$(iii), assume that there exists a linear order-preserving isomorphism $T$ from $\lipfree{M}$ onto $\ell_1(M\setminus\{0\})$. Then $T^{\ast\ast}$ is also a linear order-preserving isomorphism from $\bidualfree{M}$ onto $\ell_1(M\setminus\{0\})^{\ast\ast}$. But every element of $\ell_1(M\setminus\{0\})^{\ast\ast}$ is majorizable because $\ell_1(M\setminus\{0\})^{\ast\ast}$ is a Banach lattice, and majorizability is preserved by order-preserving isomorphisms.

The implication (iii)$\Rightarrow$(ii) follows immediately from Theorem \ref{th:minimal_majorant}.

Lastly, we prove (ii)$\Rightarrow$(i). Assume that every element of $\lipfree{M}$ is majorizable. We will first show that every point of $M$ is isolated, except possibly the base point. The argument follows a construction extracted from Example 3.24 in \cite{Weaver2}, where an interval augmented by a base point was considered. Let $x\in M\setminus\set{0}$ and suppose that it is not an isolated point of $M$. Then we can find sequences $(x_n)$ and $(y_n)$ of points of $B(x,r)$ where $r<\frac 13 d(x,0)$, all of them different from each other, that converge to $x$ and such that
$$
0<d(x_{n+1},y_{n+1})<\frac{1}{2}d(x_n,y_n)
$$
for every $n\in\NN$. For $n\in\NN$, let $f_n\in\Lip_0(M)$ be such that $0\leq f_n\leq 1$, $\supp(f_n)\subset B(x,2r)$, $f_n(x_k)=1$ for $k\leq n$, $f_n(x_k)=0$ for $k>n$, and $f_n(y_k)=0$ for all $k\in\NN$. Also let $h\in\Lip_0(M)$ be such that $h=1$ on $B(x,2r)$ and $h=0$ on $M\setminus B(x,3r)$. It is clear that $f_n\leq h$ for every $n\in\NN$. Now let $m=\sum_{n=1}^\infty (\delta(x_n)-\delta(y_n))$ and notice that $m\in\lipfree{M}$ as the series is absolutely convergent. By hypothesis there is a positive $m^+\in\lipfree{M}$ with $m^+\geq m$. Thus we get
$$
\duality{m^+,h}\geq\duality{m^+,f_n}\geq\duality{m,f_n}=n
$$
for every $n$. This is a contradiction. So $x$ must be isolated, as we claimed.

Suppose now, for contradiction, that $M$ is not radially discrete. We claim that we may find sequences $(x_n)$ and $(y_n)$ of points of $M\setminus\set{0}$, all of them different from each other, such that $d(x_n,y_n)\leq 2^{-n}d(x_n,0)$. Indeed, we proceed by induction. Choose any pair $(x_1,y_1)$ such that $d(x_1,y_1)<\frac{1}{2}d(x_1,0)$. Now suppose that the different points $x_1,y_1,\ldots,x_{n-1},y_{n-1}$ have been selected. Since all of them are isolated, we may choose $\delta>0$ such that none of them have any point of $M$ at distance less than or equal to $\delta$. We may also take $R>0$ such that all of them are contained in $B(0,R)$. Now let
$$
\alpha=\min\set{2^{-n},\frac{\delta}{\delta+R}}
$$
and, using the fact that $M$ is not radially discrete, choose a pair $(x_n,y_n)$ of different points such that $d(x_n,y_n)<\alpha\cdot d(x_n,0)$. It is easy to see that neither of them can be $0$. If $y_n$ was one of the points $x_k$ or $y_k$ with $k<n$ then we would have
$$
d(x_n,y_n)\leq\frac{\alpha}{1-\alpha}d(y_n,0)\leq\frac{\delta}{R}d(y_n,0)\leq\delta
$$
so it would have a point $x_n$ at distance less than or equal to $\delta$, which is impossible by construction. Analogously we check that $x_n$ cannot be one of the points $x_k$, $y_k$ with $k<n$. Thus the points $x_1,y_1,\ldots,x_n,y_n$ are all different, and the claim is proved.

Let
$$
m=\sum_{n=1}^\infty \frac{\delta(x_n)-\delta(y_n)}{d(0,x_n)}
$$
and notice that the norm of the $n$-th term is bounded by $2^{-n}$, so the series is absolutely convergent. For any $n\in\NN$, let $f_n\in\Lip_0(M)$ have bounded support and satisfy $0\leq f_n\leq 1$, $f_n(x_k)=1$ for $k\leq n$, $f_n(x_k)=0$ for $k>n$, and $f_n(y_k)=0$ for all $k\in\NN$. Then $\rho f_n\in\Lip_0(M)$ and $\rho f_n\leq \rho$ for every $n\in\NN$. Once again, by hypothesis we find a positive $m^+\in\lipfree{M}$ such that $m^+\geq m$. But then 
$$
\duality{m^+,\rho}\geq\duality{m^+,\rho f_n}\geq\duality{m,\rho f_n}=n
$$
for every $n$, a contradiction. This finishes the proof.
\end{proof}

Combining Theorem \ref{th:all_majorizable} with our previous results on the relation between majorizability and representability by measures, we immediately obtain a characterization of those metric spaces such that every element of $\lipfree{M}$ is represented by a Radon measure on $M$. Note that this extends \cite[Theorem 3.19]{Weaver2} which covers just the equivalence (i)$\Leftrightarrow$(ii) in the compact case.

\begin{corollary}
\label{cr:rud}
The following are equivalent:
\begin{enumerate}[label={\upshape{(\roman*)}}]
\item $M$ is radially uniformly discrete,
\item all elements of $\lipfree{M}$ are represented by Radon measures on $M$,
\item all elements of $\lipfree{M}$ are majorizable, and the base point of $M$ is isolated,
\item all elements of $\bidualfree{M}$ that avoid infinity are represented by Radon measures on $\ucomp{M}$,
\item all elements of $\bidualfree{M}$ are majorizable, and the base point of $M$ is isolated.
\end{enumerate}
\end{corollary}

\begin{proof}
(i)$\Leftrightarrow$(iii)$\Leftrightarrow$(v): This follows immediately from Theorem \ref{th:all_majorizable} and Definition \ref{def:rud}.

(iii)$\Rightarrow$(ii): If $0$ is isolated in $M$ then it is not contained in the support of any element of $\lipfree{M}$, so this implication follows immediately from Theorem \ref{th:majorizable_elements_not0}.

(ii)$\Rightarrow$(iii): It follows from Proposition \ref{pr:all_majorizable_are_radon_measures} that the base point of $M$ is isolated. Moreover, all elements of $\lipfree{M}$ represented by Radon measures on $M$ are majorizable as explained at the beginning of Section \ref{subs:Characterizations of majorizable functionals}.

(v)$\Rightarrow$(iv): This is implied by Theorem \ref{th:majorizable_elements_bidual_not0}.

(iv)$\Rightarrow$(ii): It is enough to apply Theorem \ref{th:normal_measure}.
\end{proof}

\section{Concluding remarks}
\label{sec:conclusion}

Throughout this document, we have focused on the integral representation of functionals on the Lipschitz space $\Lip_0(M)$. However it is also possible to derive results for the related space $\Lip(M)$ of all real-valued Lipschitz functions on $M$, provided that $M$ is bounded. Indeed, suppose that $M$ is bounded, not necessarily pointed. We may assume without loss of generality that $\diam(M)\leq 2$ by scaling the metric on $M$, as this only modifies the Lipschitz norm by a constant multiplicative factor. Then every $f\in\Lip(M)$ is bounded and $\Lip(M)$ becomes a Banach space endowed with the norm $\norm{f}=\max\set{\lipnorm{f},\norm{f}_\infty}$. This is in fact a renorming of the space $\mathrm{BL}(M)$ considered in \cite{HiWo_2009}, but this norm has the following advantage: it allows $\Lip(M)$ to be identified with the space $\Lip_0(M^0)$, where $M^0=M\cup\set{0}$ is constructed by attaching an extra base point $0$ to $M$ and defining $d(x,0)=1$ for every $x\in M$ (see \cite[Proposition 2.13]{Weaver2}). The identification $\Lip(M)\equiv\Lip_0(M^0)$ is not just a linear isometry, it also preserves the algebra and order structure. Notice that $M^0$ is bounded and its base point is isolated, hence every continuous linear functional on $\Lip_0(M^0)$ avoids $0$ and infinity strongly. As a consequence, many of the results obtained in the previous sections can be expressed for $\Lip(M)$ in a simplified form.

To end the paper, let us mention some problems and questions that we have dealt with throughout this work but were not able to answer.

For any $m\in\lipfree{M}$, its support $\supp(m)$ is always a separable subset of $M$. Hence, by Corollary \ref{cr:extended_support_fm}, its extended support $\esupp{m}=\ucl{\supp(m)}$ is a separable subset of $\ucomp{M}$. One may ask whether this is true for any functional $\phi\in\dual{\Lip_0(M)}$. A partial step was done in Corollary \ref{cor:extended_support_separable}, where we show that $\esupp{\phi}\cap M$ is always a separable metric space.

\begin{question}
For a complete metric space $M$ and a functional $\phi\in\dual{\Lip_0(M)}$, is $\esupp{\phi}$ a separable subset of $\ucomp{M}$? Or does $\esupp{\phi}$ at least satisfy the countable chain condition?
\end{question}

Our other question involves Proposition \ref{pr:elements_induced_by_measure}. The argument used in the proof leads through Bochner integration, which is why we place separability, resp. regularity, requirements necessary for Bochner integrability. Note that this would be the case also if we wanted to apply directly Proposition \ref{pr:elements_induced_by_measure_bidual} combined with Lebesgue's dominated convergence theorem, so that we would only need to consider sequences of functions instead of nets.

\begin{question}
Can we remove the hypothesis that $M$ is separable or $\mu$ is regular from Proposition \ref{pr:elements_induced_by_measure}? That is, if $\mu$ is any Borel measure on $M$ such that $\int_M\rho\,d\abs{\mu}<\infty$, does it induce an element of $\lipfree{M}$?
\end{question}

Lastly, we have seen in Section 5.2 that the method for reducing the setting into bounded metric spaces presented in \cite{AACD_2021} was helpful for extending results concerning majorizability from functionals avoiding infinity to full generality. This was thanks to its order-preserving nature. It might be interesting to study also how the isomorphism defined in (\ref{eq:aacd_isomorphism}) affects extended supports (see \cite[Proposition 7.4]{AACD_2021} for a corresponding result in Lipschitz-free spaces), as this information might allow us to apply the reduction technique to other types of problems.

\section*{Acknowledgments}

The authors wish to thank Gilles Godefroy for his suggestion that we study the relationship between supports in $\lipfree{M}$ and $\meas{M}$, which was the original motivation for this work, and Daniele Puglisi for directing our attention to his related work. We also thank Marek C\'uth and Antonio Jos\'e Guirao for some very helpful corrections to our manuscript and Tomasz Kochanek for interesting discussions on the topic, as well as the anonymous referees who provided corrections and suggested improvements and connections between our work and other areas of functional analysis.

This research was carried out during visits of the first author to the Faculty of Information Technology at the Czech Technical University in Prague in 2020 and of the second author to the Universitat Polit\`ecnica de Val\`encia in 2019. Both authors are grateful for the opportunity and the hospitality.

R. J. Aliaga was partially supported by the Spanish Ministry of Economy, Industry and Competitiveness under Grant MTM2017-83262-C2-2-P. E. Perneck\'a was supported by the grant GA\v CR 18-00960Y of the Czech Science Foundation.


\end{document}